\documentclass[a4paper,12pt]{article}
\usepackage{indentfirst}
\usepackage{setspace}
\usepackage{amsmath,amsthm,amsfonts,amssymb,csquotes,epsfig, titlesec, epsfig, float, caption,wrapfig, mathrsfs}
\usepackage[left=2cm,right=2cm]{geometry}
\usepackage[style=alphabetic,backend=bibtex]{biblatex} %Imports biblatex package
\usepackage[colorlinks = true,
linkcolor = blue,
urlcolor  = blue,
citecolor = purple,
anchorcolor = blue]{hyperref}
\usepackage[capitalize]{cleveref}
\usepackage{tikz-cd}
\usepackage{mathtools}
\usepackage{thmtools}
\usepackage{thm-restate}
\usepackage{quiver}
\usepackage[utf8]{inputenc}
\usepackage{graphicx}
\graphicspath{ {./Diagrams_Prelim/} }
\usepackage[colorinlistoftodos]{todonotes}
\usepackage{commath} 
\usepackage{tkz-fct}  
\usepackage[sc,osf]{mathpazo}
\usepackage[nottoc]{tocbibind}
\usepackage{enumerate}  
\usepackage{comment}
\usepackage{mathabx}
\usepackage{lipsum}
\usepackage{caption}
\usepackage{subcaption}
\usepackage{xcolor}
\definecolor{LinkColor}{RGB}{145,80,170}
\definecolor{CiteColor}{RGB}{230,125,65}
\definecolor{FileColor}{RGB}{75,170,130}
\usepackage{hyperref}
\hypersetup{
	colorlinks=true,     %Colore le texte plutôt que créer une boîte
	linkcolor=LinkColor,    %Couleur des liens internes normaux
	citecolor=CiteColor,     %Couleur des citations bibliographiques
	filecolor=FileColor,      %Couleur des URL locaux
	urlcolor=FileColor,       %Couleur des URL externes
}
\theoremstyle{plain}
\newtheorem{thm}{Theorem}[section]
\newtheorem{proposition}[thm]{Proposition}
\newtheorem{lemma}[thm]{Lemma}

\newtheorem{theorem}[thm]{Theorem}
\newtheorem{corollary}[thm]{Corollary}
\newtheorem*{theorem*}{Theorem}

\theoremstyle{plain}
\newtheorem{definition}[thm]{Definition}
\newtheorem{exmp}[thm]{Example}
\newtheorem*{definition*}{Definition}
\newtheorem*{proposition*}{Proposition}
\newtheorem*{lem*}{Lemma}

\theoremstyle{plain}

\theoremstyle{plain}
\newtheorem{rem}[thm]{Remark}

\usepackage{etoolbox}
\AtEndEnvironment{equation}{\nonumber}
\AtEndEnvironment{multline}{\nonumber}  

\theoremstyle{plain}
\newtheorem{remark}[thm]{Remark}

\theoremstyle{definition}

\theoremstyle{definition}

\theoremstyle{definition}
\newtheorem{lem}[thm]{Lemma}

\newcommand\inner[2]{\langle #1, #2 \rangle}

\newcommand\im[1]{\text{im } #1}

\newcommand\supp[1]{\text{supp } #1}

\renewcommand{\Im}{\text{Im}}

	\def \scurve{\Sigma_{\phi}}
	\def \snetwork{\mathcal{S}}
	\def \primitive{W}
	
	\def \canliouvile{\lambda_{re}}
	\def \holliouvile{\lambda}
	
	\def \slit{\xi}
	\def \alp{d}
	\def \halfradius{\tau}

	\def \tloc{\widetilde{(L,\nabla)}}
	\def \loc{(L,\nabla)}

	\def \injradius{r}
	\def \sheetgap{\rho}
	\def \heightR{R}

	\def \tloc{\widetilde{(L,\nabla)}}
	\def \loc{(L,\nabla)}
	\def \slit{\xi}
	
	\def \stdm{\triangle_m}
	\def \npunct{m}
\graphicspath{ {./images/} }

\numberwithin{equation}{subsection}
\title{Family Floer theory, non-abelianization, and Spectral Networks}
\author{Yoon Jae Nho}
\date{}
\begin{document}
%\spacing{1.5}
	\maketitle
	\begin{abstract}
		In this paper, we study the relationship between Gaiotto-Moore-Neitzke's non-abelia\\nization map and Floer theory. Given a complete GMN quadratic differential $\phi$ defined on a closed Riemann surface $C$, let $\tilde{C}$ be the complement of the poles of $\phi$. In the case where the spectral curve $\scurve$ is exact with respect to the canonical Liouville form on $T^{\ast}\tilde{C}$, we show that an ``almost flat" $GL(1;\mathbb{C})$-local system $\mathcal{L}$ on $\scurve$ defines a Floer cohomology local system $HF_{\epsilon}(\scurve,\mathcal{L};\mathbb{C})$ on $\tilde{C}$ for $0< \epsilon\leq 1$. Then we show that for small enough $\epsilon$, the non-abelianization of $\mathcal{L}$ is isomorphic to the family Floer cohomology local system $HF_{\epsilon}(\scurve,\mathcal{L};\mathbb{C})$.
	\end{abstract}
	\tableofcontents
%	\section{Thesis Summary}\label{section:thesissummary}
%\input{./section_thesissummary.tex}
\section{Introduction}\label{introduction}
\subsection{Statement of results}
 Let $C$ be a closed Riemann surface, and let $\omega_C$ be the canonical line bundle of $C$. A \textit{quadratic differential} $\phi$ is a meromorphic section of $\omega_C^{\otimes 2}$. We say that a quadratic differential is a \textit{complete GMN quadratic differential} if all of its zeroes are simple, it does not have poles of order $1$ and has at least one pole.

Let $\tilde{C}$ be the complement of the poles of $\phi$, and let $(T^{\ast}_{\mathbb{C}})^{1,0}\tilde{C}$ denote the total space of $\omega_{\tilde{C}}$. The complete GMN 
quadratic differential $\phi$ defines a smooth embedded algebraic subvariety $\scurve$ of $(T^{\ast}_{\mathbb{C}})^{1,0}\tilde{C}$ 
called the \textit{spectral curve associated to $\phi$}. It becomes a simple branched double covering of $\tilde{C}$ by restricting the projection map $\pi:(T^{\ast}_{\mathbb{C}})^{1,0}\tilde{C}\to \tilde{C}$ to $\scurve$. This curve is defined using the canonical \textit{holomorphic Liouville form} $\holliouvile=p^z dz$ via
\begin{align}
	\scurve:=\{\holliouvile^2-\pi^{\ast}\phi=0\}\subset  {T^{1,0}_{\mathbb{C}}}^{\ast}\tilde{C}.
\end{align}
Here, $p^z$ is the complex fibre coordinate, and $z$ is the complex base coordinate \footnote{For the definition, see \cref{introquadss}}. 

We have the canonical \textit{holomorphic symplectic form} $\Omega$ on $(T^{\ast}_{\mathbb{C}})^{1,0}\tilde{C}$ defined by $\Omega:=d\holliouvile$. The spectral curve $\scurve$ is a \textit{holomorphic Lagrangian submanifold} in the sense that it is a holomorphic submanifold of $(T^{\ast}_{\mathbb{C}})^{1,0}\tilde{C}$ and the holomorphic symplectic form $\Omega$ vanishes on $\scurve$. There is also a diffeomorphism between $(T^{\ast}_{\mathbb{C}})^{1,0}\tilde{C}$ and the real cotangent bundle $T^{\ast}\tilde{C}$, sending the real part of the holomorphic Liouville form to the canonical real Liouville form $\canliouvile=\sum p^i dq_i$. Then the spectral curve $\scurve$ becomes an $\omega$-Lagrangian submanifold of $T^{\ast}\tilde{C}$ under this identification, where $\omega=d\canliouvile$.

Suppose now that the spectral curve $\scurve$ is exact with respect to $\canliouvile$, then so is $\epsilon\scurve$ for any $\epsilon\in \mathbb{R}_{>0}$. Such quadratic differentials cut out a totally real submanifold of the space of quadratic differentials of fixed pole type. Now let $C^{\circ}$ be the complement of the zeroes and poles of $\phi$, and $\scurve^{\circ}=\pi^{-1}(C^{\circ})$. Following \cite[Section 4.2]{SpectralNetworkFenchelNielsen}, we say that a rank $1$ local system $\mathcal{L}$ over $\scurve^{\circ}$ is \textit{almost flat} if the monodromy along a small loop around any of the ramification points in $\pi^{-1}(zero(\phi))$ is $-Id$.  Let $\mathfrak{s}$ be a spin structure on $C$, then $\mathfrak{s}$ induces a continuous family of spin structures $\mathfrak{f}_z$ on the cotangent fibres $F_z$, for $z\in C$.  Given an almost flat $GL(1;\mathbb{Z})$-local system $\mathcal{B}$, the spin structure $\tilde{\mathfrak{s}}=\mathfrak{s}\otimes \mathcal{B}$ on $\scurve^{\circ}$ extends to a global spin structure on $\scurve$, which we still denote as $\tilde{\mathfrak{s}}$.  Furthermore, given an almost flat $GL(1;\mathbb{C})$-local system $\mathcal{L}$, $\mathcal{L}\otimes \mathcal{B}$ extends to a $GL(1;\mathbb{C})$-local system on $\scurve$.

We show that together with an almost flat $GL(1;\mathbb{C})$-local system $\mathcal{L}$ on $\scurve$, spin structures $\tilde{\mathfrak{s}}$ and $\mathfrak{f}_z$, and a choice of $\mathcal{B}$, we can define the family Floer cohomology local system
\begin{align}\label{localsystemflat}
	HF_\epsilon(\scurve,\mathcal{L},\mathfrak{s},\mathcal{B};\mathbb{C}): z\mapsto HF(\epsilon\scurve,F_{z},\mathfrak{\tilde{s}},\mathfrak{f}_z,\mathcal{L}\otimes \mathcal{B};\mathbb{C}),
\end{align}
for any $\epsilon\in \mathbb{R}_{>0}$. Here we are taking Floer cohomology over $\mathbb{C}$ twisted by the $GL(1;\mathbb{C})$-local system $\mathcal{L}\otimes \mathcal{B}$. It turns out that \eqref{localsystemflat} is concentrated in the zeroth degree, is free and has rank $2$. 

In \cite{GNMSN,WallcrossingHitchinSystemWKBapproximation}, Gaiotto, Moore and Neitzke constructed the non-abelianization map which sends an almost flat $GL(1;\mathbb{C})$-local system on $\scurve^{\circ}$ to a $GL(2;\mathbb{C})$-local system on $\tilde{C}$. The main theorem of this paper is that for small enough $\epsilon$, $HF_{\epsilon}(\scurve,\mathcal{L},\mathfrak{s},\mathcal{B};\mathbb{C})$ and the non-abelianization of $\mathcal{L}$ are isomorphic. 

\begin{theorem} Suppose $\scurve$ is exact with respect to the real Liouville form $\canliouvile$. Let $\mathfrak{s}$ be a global spin structure on $C$, and let $\mathcal{L}$ be an almost flat $GL(1;\mathbb{C})$-local system on the spectral curve. Let $\mathcal{B}$ be an almost flat $GL(1;\mathbb{Z})$-local system, and extend $\tilde{\mathfrak{s}}=\pi^{\ast}\mathfrak{s}\otimes \mathcal{B}$ and $\mathcal{L}\otimes \mathcal{B}$ to $\scurve$. 
	
	There exists an $\epsilon_0(\delta;E)>0$ such that for $0<\epsilon<\epsilon_0$, the Floer cohomology local system 
	\[	HF_\epsilon(\scurve,\mathcal{L},\mathfrak{s},\mathcal{B};\mathbb{C}): z\mapsto HF(\epsilon\scurve,F_{z},\mathfrak{\tilde{s}},\mathfrak{f}_z,\mathcal{L}\otimes \mathcal{B};\mathbb{C})\]	is isomorphic to the non-abelianization of $\mathcal{L}$.\label{maintheoremimprecise} 
\end{theorem}
In particular, the isomorphism class of the local system does not depend on the choice of $\mathfrak{s}$ and $\mathcal{B}$ and so we write $HF_\epsilon(\scurve,\mathcal{L};\mathbb{C})$ instead. 

The key analytic theorem needed to show \cref{maintheoremimprecise} is the following adiabatic degeneration type theorem.  There exists a natural flat metric $g^{\phi}$ on $C^{\circ}$ given by $g^{\phi}=\abs{\phi(z)}\abs{dz}$. As we will see in \cref{{introquadss}}, $g^{\phi}$ gives rises to a codimension $1$ network $\snetwork(0)$ inside $\tilde{C}$ called the \textit{spectral network}. Given some small parameter $\delta>0$, we find a suitable K\"{a}hler metric $g^{\phi}_{\delta}$ over $\tilde{C}$ which agrees with $g^{\phi}$ outside a small neighbourhood of the zeroes of $\phi$. Given $\phi$ such that $\scurve$ is exact, and a large energy cut-off $E\gg 0$, we construct some bounded open subdomain $C(\delta;E)$ of $C$ which is a deformation retract of $\tilde{C}-\snetwork(0)$, such that $g^{\phi}_{\delta}=g^{\phi}$ on $C(\delta;E)$. 

The metric $g^{\phi}_{\delta}$ gives rise to an induced almost complex structure $J=J(\delta)$ on $T^{\ast}\tilde{C}$ called the Sasaki almost complex structure (see \cref{sasakialmostcomplexstructuredefini}). For $z\in C^{\circ}$, and $0<\epsilon\leq 1$, a $J$-holomorphic strip $u$ bounded between $\epsilon\scurve$ and $F_z$, that travels between \textbf{distinct} lifts of $z$ on $\scurve$, is called an \textit{$\epsilon$-BPS disc} ending at $z$. The main analytic theorem of the paper is the following \textit{non-existence} result for $\epsilon$-BPS discs ending at $z$ for small enough $\epsilon>0$ and $z\in C(\delta;E)$.

%For $0<t\leq 1$, , the cotangent fibre at $z$ for $z\in C(\delta;E)$ %\textit{$J$-holomorphic disc bounded between $F_z$ and $\epsilon\scurve$}. 

%$J$-holomorphic discs bounded between $F_z$ and $\epsilon\scurve$ 

\begin{theorem}\label{maintheorem}
	Given $E\gg 0$, $\delta\ll 1$, there exists a metric $g^{\phi}_{\delta}$ on $\tilde{C}$, a deformation retract $C(\delta;E)$ of $\tilde{C}-\snetwork(0)$ over which $g^{\phi}_{\delta}=g^{\phi}$ such that the following holds.
	\begin{itemize}\item[] Let $J$ be the Sasaki almost complex structure associated to $g^{\phi}_{\delta}$. Then there exists a scaling parameter $\epsilon_0=\epsilon_0(\delta;E)>0$ such that for $0<\epsilon\leq \epsilon_0$, there are \textit{no} non-constant $J$-holomorphic strips bounded between $F_z$ and $\epsilon\scurve$ for $z\in C(\delta;E)$. 
	\end{itemize}
\end{theorem}
The main motivation behind the proof of \cref{maintheorem} comes from the following general expectation. Suppose we have a sequence $\epsilon=\epsilon_n$ convering to zero, points $z_{\epsilon}$ converging to some $z\in\tilde{C}$ and $\epsilon$-BPS discs $u_{\epsilon}$ ending at $z_{\epsilon}$. Regarding $\epsilon \scurve$ as an exact multi-graph, we expect the sequence $u_\epsilon$ to degenerate to sets of solutions of Morse-like local differential equations on $\tilde{C}$, after possibly passing to a subsequence. The resulting set of solutions on $\tilde{C}$ is called the \textit{adiabatic degeneration} of the sequence $u_\epsilon$. The quadratic differential gives rise to a \textit{singular foliation} called the horizontal foliation. This horizontal foliation is obtained from the kernel of the $1$-form $\Im(\sqrt{\phi}dz)$. One can check that the leaves are up to re-parametrization solutions of the $g^{\phi}$-geodesic equation. However, in our case, the solutions of the resulting Morse-like local differential equation lie on the leaves of the horizontal foliation. Now, we can choose the domain $C(\delta;E)$ to be such that the leaves of the horizontal foliation passing through a point in $C(\delta;E)$ never enter some neighbourhood of the branch points. 

Using these observations, in \cref{ProofSubsect}, we modify Ekholm's Morse flow tree techniques \cite[Section 2-5]{Morseflowtree} to show that after passing to a subsequence, $u_\epsilon$ maps arbitrarily close to the leaf $\gamma$ of the horizontal foliation  passing through some $z\in C(\delta;E)$ for small enough $\epsilon$. By construction, we can find a small neighbourhood of $\gamma$ contained in $C^{\circ}$ over which the metric $g^{\phi}_{\delta}$ agrees with $g^{\phi}$. We show that such discs cannot exist under the finite energy assumption and prove the the main analytic theorem. 
\subsection{Set-up} \label{Introsetup}

The set-up of the paper is as follows. In \cref{introquadnon-Ab}, we review the theory of quadratic differentials and non-abelianization. In  \cref{Floer Theory on Open Manifolds Sect 1}, we introduce and gather the necessary ingredients from pseudo-holomorphic curve theory (namely, monotonicity) to establish the Floer theoretic set-up that we will use throughout the paper. In  \cref{Spectral Curves}, we discuss the geometry of $\phi$-metrics and the wall-chamber decomposition induced by $\snetwork(0)$. We will then study conditions under which the spectral curve is real exact and find particular deformation retracts of $C-\snetwork(0)$ called $C(\delta;E)$ with the properties described in \cref{introduction}. In \cref{Adia Degen}, we will adapt the adiabatic degeneration techniques as in \cite[Section 5]{Morseflowtree} to prove \cref{maintheorem}. Finally, in \cref{Wall-Crossing Analysis Section}, we will use the Gromov compactness argument to show that the local system is a non-abelianization up to signs. We will then compute the signs and prove \cref{FullNon-Abelianization}. 
\subsection{Conventions}\label{sec:convention}
We use the following conventions:
\begin{itemize}
	\item The canonical symplectic form on the cotangent bundle is $\textbf{dp}\wedge \textbf{\text{dq}}$.
	\item The Hamiltonian vector field associated to a smooth function $H$ on $T^{\ast}M$ is defined by
	\[i_{X_H}\omega=-dH.\]
	\item All the holomorphic polygons are given anticlockwise boundary orientations regarded as the unit disc with punctures on the boundary in $\mathbb{C}$. 
	\item When we take the identification $\mathbb{C}^n\simeq T^{\ast}\mathbb{R}^n$ as a target symplectic manifold, we take the induced ``standard complex structure" on $\mathbb{C}^n$ to be given by $z_k=x_k-iy_k,k=1,...,n$.
	\item When we regard $\mathbb{C}$ as a Riemann surface or a conformal domain, we take the complex structure given by $z=x+iy$. 
	\item The contact form on the jet bundle is given by $dz-\textbf{p}\textbf{dq}$. 
	\item $W^{k,p}$ denotes the $(k,p)$-Sobolev Space. 
	\item Given a topological metric space $(X,d)$, a subset $N\subset X$, and $x\in X$, we set $d(N,x)$ to be the distance between $N$ and $x$. Given $l>0$, we set $B_{l}(N)$ to be the set of points $x$ in $X$ with $d(x,N)<l$. 
	\item Given a complete Riemannian manifold $(M,g)$ and $x,y\in M$, we consider the induced topological metric $d$ on $M$, and we define $d(N,x)$ and $B_l(N)$ accordingly.
	\item Given a complete Riemannian manifold $(M,g)$, we set
	\[r:T^{\ast}M\to \mathbb{R}, r(q,p)=\abs{p}.\]
	Here, the norm of the covector $p$ is taken with respect to $g$. Then we set
	\begin{align}
		D_l^{\ast}M&=\{(q,p)\in T^{\ast}M: r(q,p)<l\}\\
		S_l^{\ast}M&=\{(q,p)\in T^{\ast}M, r(q,p)=l\}.
	\end{align}
	\item For $l>0$, we set:
	\begin{align}
		A_l&=\{z\in \mathbb{C}:\abs{z}<l\}\\
		\partial A_l&=\{z\in \mathbb{C}:\abs{z}=l\}\\
		E_l&=\{z\in \mathbb{C}:\abs{z}<l, Im(z)\geq 0\}\\
		\partial E_l&= \{z\in \mathbb{C}:\abs{z}<l, Im(z)= 0\}.
	\end{align}
	\item Given a real positive function $a(x)$ defined on some subset $I$ of $\mathbb{R}_x$ and $\alpha\in \mathbb{R}_{>0}$, we say that $a(x)$ is of size  $O(x^{\alpha})$ if there exists some $C>0$ such that
	\[a(x)<Cx^{\alpha}\]
	for all $x\in I$. 
	\item We adopt the convention that the infinite strip $\mathcal{Z}=(-\infty,\infty)\times [0,1]$ is given the  conformal coordinate $z=s+it$. 
\end{itemize}
\section{Quadratic differentials}\label{introquadnon-Ab}

In this section, we review the theory of quadratic differentials and GMN non-abelianization. In \cref{introquadss}, we review the theory of quadratic differentials following \cite{bridgeland2014quadratic}. In \cref{introNonAb}, we review the GMN non-abelianization map. In \cref{Proofguidess}, we roughly explain how to arrive at \cref{maintheoremimprecise} using \cref{maintheorem}. 

\subsection{Quadratic differentials}\label{introquadss}

We describe the local structure of zeroes and poles, the spectral curve, and the induced singular flat metric on the base. We follow the expositions from \cite{quaddiff} and \cite{bridgeland2014quadratic}. Again, let $C$ be a closed Riemann surface, and let $\omega_{C}$ be the canonical line bundle. Then  

\begin{definition}
	A \textbf{quadratic differential} is a meromorphic section of $\omega_C^{\otimes 2}$. Equivalently, a quadratic differential is a collection of open conformal charts $(U_{\mu},z_{\mu})$ where $z_{\mu}:U_{\mu}\to \mathbb{C}$ is a biholomorphism onto its image, together with a collection of meromorphic functions $\phi_{\mu}$ on $U_{\mu}$ such that
	\begin{align}
		\phi_{\mu'}=\phi_{\mu} \left(\frac{dz_{\mu}}{dz_{\mu'}}\right)^2 \text{ on } U_{\mu}\cap U_{\mu'}. \label{transformationrule}
	\end{align}
	We then locally write $\phi=\phi_{\mu}dz_{\mu}^2$
	on $U_{\mu}$.
\end{definition}
A zero or a pole of $\phi$ is called a \textit{critical point}. 
We say that a critical point is \textit{finite} if it is either a simple pole or a zero of $\phi$. Otherwise, we say a critical point is an \textit{infinite} critical point. The critical points of $\phi$ have the following local structure. For details, see \cite[Section 6]{quaddiff}. The following is Theorems 6.1-6.4 in \cite{quaddiff} combined.

\begin{proposition}\label{zeroofphiprop}
	Let $b$ be either a finite critical point of $\phi$ or a pole of an odd order. Let $n$ be the exponent of $b$. Then there exists a neighbourhood $U_b$ of $b$, an open set $D$ of $\mathbb{C}$ containing zero, and a biholomorphism $\xi=\xi_b:(D,0)\to (U_b,b)$ such that
	\begin{align}\label{finoddform}
		\phi(\xi)d\xi^2=\Big(\frac{n+2}{2}\Big)^2 \xi^n d\xi^2.\end{align} 
	Furthermore, the germ of the biholomorphism is unique up to a factor of some $c=\exp(\frac{k}{n+2}(2\pi i))$ for $k=0,1,2,..,n+1$. 
	In particular, for $n=1$, we get
	\begin{align}\label{zeroofphilocalform}
		\phi(\xi)d\xi^2=\Big(\frac{3}{2}\Big)^2 \xi d\xi^2.
	\end{align} \\
	Let $b$ be a pole of order $2$. Then there exists a local conformal parameter $\xi$ which is unique up to a factor of a constant $a_{-2}\in \mathbb{C}$ such that
	\begin{align}\label{order2polelocalform}
		\phi(\xi)d\xi^2=a_{-2}\xi^{-2} d\xi^2.\\
	\end{align}
	Let $b$ be a pole of $\phi$ with even order $n=2m\geq 4$. Then there exists a local conformal parameter $\xi$ and a constant $r\in \mathbb{C}$ such that 
	\begin{align}\label{orderevenpolelocalform}
		\phi(\xi)d\xi^2=\Big({\frac{1}{2}(2-n)}\xi^{-m}+r\xi^{-1}\Big)^2 d\xi^2.
	\end{align}
\end{proposition}

%	The pole structures of quadratic differentials are well-behaving as well. For $x\in pole(\phi)$, suppose $ord(x)=1$, then we have local conformal co-ordinates $z$ such that
%	\begin{align}\label{simplepole}
	%		\phi=(\frac{1}{2})^2 z^{-1}dz^2.
	%	\end{align}
%	If the order of $x$ is $2$, then we have
%	\begin{align}\label{doublepole}
	%		\phi=\frac{c}{z^2} dz^2
	%	\end{align}
%	for some $c\in \mathbb{C}$. 	
%	For $ord(x)>2$, let $n=ord(x)$, and $c=\frac{1}{2}(2-n)$. Suppose $ord(x)$ is odd, then we have local conformal coordinates $z$ such that
%	\begin{align}\label{oddpole}
	%		\phi=\frac{c^2}{z^n}dz^2,
	%	\end{align}
%	and in the case $ord(x)=2m$, we have
%	\begin{align}\label{evenpole}
	%		\phi=(\frac{c}{z^{m}}+\frac{b}{z})^2 dz^2
	%	\end{align} for some
\paragraph{Spectral curves}\label{introquadssspectralcurves}
Let $\phi$ be a complete GMN quadratic differential. Recall that we wrote $\tilde{C}$ for the complement of poles of $\phi$. Then $\phi$ gives rise to a holomorphic Lagrangian submanifold of $T^{\ast}\tilde{C}$ called the spectral curve $\scurve$.  To define this, let $(T^{\ast}_{\mathbb{C}})^{1,0}\tilde{C}$ denote the holomorphic cotangent bundle of $\tilde{C}$. There exists a \textit{canonical holomorphic Liouville 1-form} $\holliouvile$ on $(T^{\ast}_{\mathbb{C}})^{1,0}\tilde{C}$; for $(q,p)\in(T^{\ast}_{\mathbb{C}})^{1,0}\tilde{C}$ and $V\in T(T^{\ast}_{\mathbb{C}})^{1,0}\tilde{C}$, we define
\[\holliouvile(q,p)(V)=p(\pi_{\ast}(V))\]
where $\pi:(T^{\ast}_{\mathbb{C}})^{1,0}\tilde{C}\to \tilde{C}$ is the projection map. We evaluate $\pi_{\ast}V\in T_q \tilde{C}$ on $p\in (T_q \tilde{C})^{\ast}$ with respect to the canonical pairing
\[(T_q \tilde{C})^{\ast}\otimes (T_q\tilde{C})\to \mathbb{C}.\] 
In local coordinates, we can write $\holliouvile=p^z dz$ where $p^z$ is the complex fibre coordinate, and $z$ is the complex base coordinate. We see that $\holliouvile$ gives the canonical section of the line bundle $\pi^{\ast}\omega_{\tilde{C}}$. We obtain the canonical \textit{holomorphic symplectic form} on $(T^{\ast}_{\mathbb{C}})^{1,0}\tilde{C}$ by taking the exterior derivative $\Omega=d\holliouvile$. 

There exists a diffeomorphism of the total space of the real fibre bundles 
\[T^{\ast}\tilde{C}\to (T^{\ast}_{\mathbb{C}})^{1,0}\tilde{C}\]
between the real cotangent bundle and the holomorphic cotangent bundle, under which the real part of $\holliouvile$ is pull-backed to the canonical real Liouville form $\canliouvile$ on $T^{\ast}\tilde{C}$. The diffeomorphism is induced by the identification of $V_{\mathbb{C}}^{1,0}\simeq V$ with $V$ a real vector space with a complex structure $I:V\to V,\:I^2=-Id$ (see \cite[Section 5.2]{exactWKB}).\newline\\
The algebraic variety 
\begin{align}
	\scurve:=\{\holliouvile^2-\pi^{\ast}\phi=0\}\subset  (T^{\ast}_{\mathbb{C}})^{1,0}\tilde{C}
\end{align}
is called the \textit{spectral curve} associated to the quadratic differential $\phi$. It is smooth if the zeroes of $\phi$ are simple. In this case, the projection $\pi:\scurve\to \tilde{C}$ gives a simple branched double covering of $\tilde{C}$. To see this, note that by Proposition \ref{zeroofphiprop}, if $z_0\in \phi$ is a zero of $\phi$, then one can find some conformal coordinate charts near $z_0$ such that $\phi$ reads locally $zdz^2$. Then realizing $\mathbb{C}^2$ as the holomorphic cotangent bundle of $\mathbb{C}$, we see that the germ of the spectral curve near $z_0$ is equivalent to the germ of $\{(p^z)^2-z=0\}\subset \mathbb{C}^2$ at $(z,p^z)=(0,0)$, which is smooth. Now observe that the holomorphic symplectic form $\Omega$ vanishes on any smooth codimension one algebraic subvariety of $(T^{\ast}_{\mathbb{C}})^{1,0}\tilde{C}$. Under the identification of the holomorphic and the real cotangent bundle, we see that $\scurve$ becomes a real Lagrangian submanifold of $T^{\ast}\tilde{C}$. 
\paragraph{$\phi$- metric.}\label{introquadssflatmetric}

We need a further ingredient to describe non-abelianization, the natural flat singular metric structure on $C$. To define this, let $C^{\circ}$ denote the complement of the zeros and poles of $\phi$. On $C^{\circ}$, we have a corresponding Riemannian metric 
%that captures the symplectic topology of the spectral curve $\scurve$. This metric structure will be crucial in defining and understanding the family Floer cohomology local system.
\begin{align}
	g^{\phi}=\abs{\phi(z)}\abs{dz}^2\label{flatsingularmetric}
\end{align} 
which we regard as a singular metric on $\tilde{C}$. The metric $g^{\phi}$ is actually flat because in local conformal coordinate $W=\int \sqrt{\phi}$, $\phi\equiv dW^2$ and $g^{\phi}\equiv\abs{dW}^2$ by (\ref{transformationrule}). 
\\\newline 
We will be interested in the following class of quadratic differentials with nice $g^{\phi}$-metric properties. 
\begin{definition}{\cite[Definition 2.1]{bridgeland2014quadratic}} A meromorphic quadratic differential is GMN\footnote{For Gaiotto, Moore and Neitzke who first introduced the theory of spectral networks with which we are concerned.} if:
	\begin{itemize}
		\item all the zeroes of $\phi$ are simple,
		\item $\phi$ has at least one pole,
		\item $\phi$ has at least one finite critical point (either an order one pole or zero).
	\end{itemize}
	We say that a GMN quadratic differential $\phi$ is \textit{complete} if $\phi$ has no simple poles. 
\end{definition}
Note that $g^{\phi}$ induces a metric space structure on $\tilde{C}$ because in the metric on $C^{\circ}$, the zeroes and the simple poles are at finite distance. If $\phi$ is complete, then the metric space is also complete. To see this, note that the integral $\lim_{a\to 0^{+}}\int_a^{1} \frac{1}{x^b}dx$   
for $0<b<\infty$ converges for $b=1/2$, but not for $b\geq 1$. Now comparing with the local forms in Proposition \ref{zeroofphiprop}, we see that the integral of the line element $\abs{\sqrt{\phi}}\sim \frac{1}{\abs{z}^b}$ for $b\geq 1$ blows up as $z\to 0$. For Floer theoretic purposes, we will restrict to complete GMN quadratic differentials. 

Each $g^{\phi}$-geodesic, or $\phi$-geodesic for short, admits a unique phase in $\mathbb{R}/{\pi \mathbb{Z}}$ since $\phi$-geodesics are  just straight lines in the $W$-coordinate. We call geodesics with phase $\theta=0$ \textit{horizontal} 
and geodesics with phase $\theta= \frac{\pi}{2}$ \textit{vertical}. We call maximal solutions of the $\phi$-geodesic equation \textit{trajectories}.	
\begin{figure}[t]
	\includegraphics[width=0.4\textwidth]{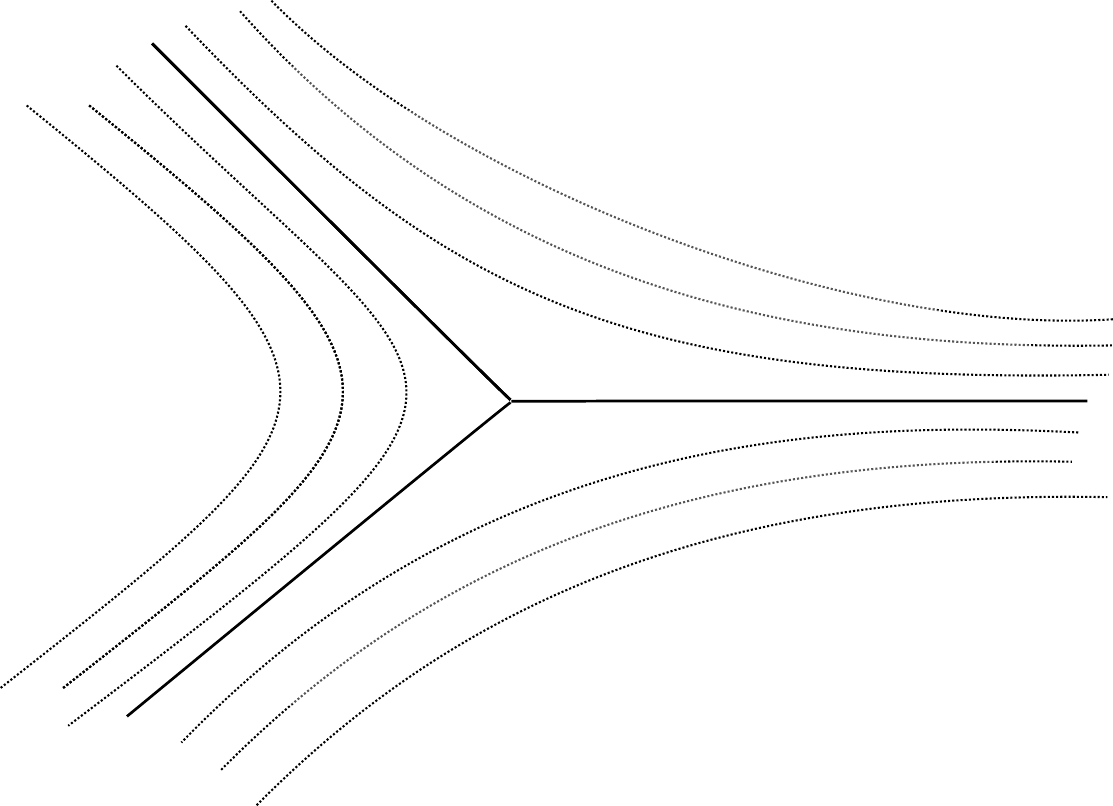}
	\centering
	\caption{Horizontal trajectories for $\phi=zdz^2$. The dotted lines are generic horizontal trajectories.}
	\label{Snetworksimpfig}
\end{figure}

There are several types of trajectories. If the trajectory $\gamma$ has its maximal interval of definition, a finite open interval, or equivalently, approaches finite critical points at both ends, we say that $\gamma$ is a \textit{saddle trajectory}. If it 
is defined over $(-\infty,\infty)$ or equivalently approaches infinite critical points at both ends, then we say that it is a \textit{generic trajectory}. If the trajectory approaches a finite critical point at a single end, we say it is a \textit{separating trajectory}. Note that horizontal generic trajectories do not intersect with each other. We say that $\phi$ is \textit{saddle-free} if there are no horizontal saddle trajectories on $C$. We can always rotate $\phi$ by $e^{2i\theta}$ for a generic $\theta$ to obtain a saddle-free quadratic differential \cite[Lemma  4.11]{bridgeland2014quadratic}.\\\newline 
The phase $\theta$ trajectories in $C^{\circ}$ give a singular foliation on $\tilde{C}$. The critical graph of this singular foliation is called the \textit{spectral network} $\snetwork(\theta)$. The spectral network $\snetwork(\theta)$ is stratified into a $0$-th dimensional stratum consisting of all the zeroes of $\phi$ and a 1-dimensional stratum consisting of the separating $\theta$-trajectories, also called \textit{walls}. %that only meet at the zeroes of $\phi$. These walls are $\theta$-trajectories with at least one finite end. 

The complement of the spectral network for a saddle-free GMN quadratic differential is a disjoint union of \textit{chambers}; chambers are connected contractible conformal subdomains of $\tilde{C}$. Given a chamber $\mathcal{Z}^h$, there exists a conformal equivalence of $(\mathcal{Z}^h,\phi)\simeq (\mathcal{Z}^h(a,b),dz^2)$ where $\mathcal{Z}^h(a,b)$ is either the upper half-plane or a finite horizontal strip subdomain of $\mathbb{C}$ (c.f. \cite[Section 3.4, Section 3.5 Lemma 3.1]{bridgeland2014quadratic}). These chambers are maximal horizontal domains, meaning that they are spanned by generic horizontal trajectories. Thus we have a cellular decomposition of $\tilde{C}$, where the 2-cells are the chambers, the 1-cells the walls, and the 0-cells the zeroes of $\phi$. The spectral curve $\scurve$ restricted to a chamber is sent under the conformal equivalence $(\mathcal{Z}^h,\phi)\simeq (\mathcal{Z}^h(a,b),dz^2)$, to the two disjoint affine hyperplanes $\{p^x=\pm 1, p^y=0\}$, where $p^z=p^x-ip^y$. In other words, they are covector coordinates.
\begin{figure}[t]
	\begin{subfigure}{0.5\textwidth}
		\includegraphics{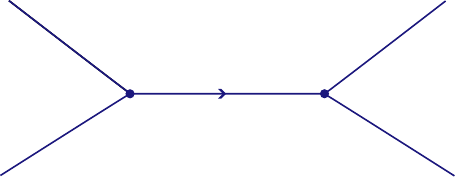} 
		\caption{$\snetwork(0)$ with a saddle trajectory.}
		\label{fig:saddlefig}
	\end{subfigure}
	\begin{subfigure}{0.5\textwidth}
		\includegraphics{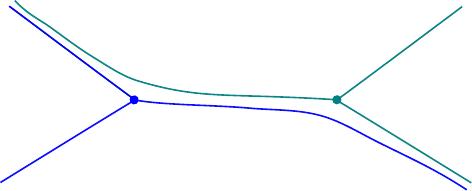}
		\caption{$\snetwork(\theta)$ for small $\theta$}
		\label{fig:perturbsaddlefig}
	\end{subfigure}
	\label{fig:saddlewallcrossing}
\end{figure}

\subsection{Non-abelianization}\label{introNonAb}

We now state what we mean by non-abelianization. Let $\mathcal{R}$ be either $\mathbb{Z}$ or $\mathbb{C}$. We define a $GL(k;\mathcal{R})$-local system to be a rank $k$ locally constant sheaf of free $\mathcal{R}$-modules.  Inspired by \cite[Section 6]{SpectralNetworkFenchelNielsen}, we look at the ``integrated version" of local systems or the \textit{path groupoid representation} of $\mathcal{R}$-local systems (Definition \ref{pathgropoidrepresentationdefinition}). 

\begin{definition}\label{pathgropoidrepresentationdefinition}
	Let $\mathcal{P}_M$ be a set of finite points in $M$ called the set of \textit{base points}. The \textit{path groupoid} $\mathcal{G}_M$ is a category whose objects are points in $\mathcal{P}_M$, and whose morphisms are path-homotopy classes (in $M$) between the points in $\mathcal{P}_M$. A collection of morphisms of $\mathcal{G}_M$ is said to be a path groupoid generating set if their concatenations generate $\mathcal{G}_M$. 
	
	A \textit{path groupoid representation} of a $GL(k;\mathcal{R})$-local system consists	of the following data.
	\begin{itemize}
		\item A free rank $k$ $\mathcal{R}$-module $E_b$ for each $b\in \mathcal{P}_M$ together with an isomorphism
		\begin{align}\label{basisfixing}
			\mathcal{R}^{\oplus k}\simeq E_b.
		\end{align}
		\item A morphism
		\begin{align}\label{groupoidparalleltransport}
			\Gamma(\alpha):E_{b}\to E_{b'}
		\end{align}
		given a path homotopy class $\alpha\in \pi_1(b,b')_M$, $b,b'\in \mathcal{P}_M$, such that $\Gamma(\alpha)$ is compatible with path concatenations.
	\end{itemize}
	Two path groupoid representations $(\mathcal{P}_M,E',P)$ and $(\mathcal{P}_M,E',P')$ are said to be equivalent if for each $b\in M$ there are isomorphisms
	\[g_b:E_b\to E_b'\]
	such that (i) the following diagram commutes for $\alpha\in \pi_1(b,b')$, $b,b'\in \mathcal{P}_M$:
	\[\begin{tikzcd}
		{E_b} & {E_{b'}} \\
		{E'_b} & {E'_{b'}}
		\arrow["{\Gamma'(\alpha)}", from=2-1, to=2-2]
		\arrow["{\Gamma(\alpha)}", from=1-1, to=1-2]
		\arrow["{g_b}"', from=1-1, to=2-1]
		\arrow["{g_b'}"', from=1-2, to=2-2]
	\end{tikzcd},\]
	and (ii) the isomorphisms $g_b$ are compatible with the isomorphisms \eqref{basisfixing} to $	\mathcal{R}^{\oplus k}$ above. 
\end{definition}
%We call the points in $\mathcal{P}_M$ the \textit{base points} for the path groupoid of $M$ with respect to $\mathcal{P}_M$. Furthermore, a finite collection of homotopy classes of paths in $\pi_1(b,b'),b,b'\in \mathcal{P}_M$, whose concatenations generate $\pi_1(M)$ is called the \textit{path groupoid generating set} with respect to $\mathcal{P}_M$. The elements are called \textit{path groupoid generators}. 

Given a path groupoid representation of a $GL(k;\mathcal{R})$-local system, we can build a genuine $\mathcal{R}$-local system on $M$. To see this, we borrow the argument in \cite[Section 6]{SpectralNetworkFenchelNielsen}. Consider the following space
\[\tilde{\mathcal{P}}^M:\{\gamma:I\to M: \gamma(0)\in \mathcal{P}_M\}/\{\sim\}\]
of path homotopy classes that begin at some $m\in \mathcal{P}_M$ and end at some other point $m'$. Let $\tilde{P}^M_b$ be the connected component of $\tilde{P}^M$ containing the constant path at $b\in \mathcal{P}^M$. Then we glue the constant sheaves $E_b\times \tilde{P}^M_b$ by $(v,b)\sim (\Gamma(\alpha)v,b')$ for $\alpha\in \pi_1(b,b')$. 

The spectral network $\snetwork(0)$ induces a cellular decomposition on $\tilde{C}$. Choose a point $b(w)$ away from the zeroes of $\phi$ over each wall $w$ in $\snetwork(0)$. The wall $w$ picks out a unique sheet of $\sqrt{\phi}$ in the following sense: choose any parametrization $w:[0,\infty)\to \tilde{C}$ in the outward orientation; there exists a unique sheet of $\sqrt{\phi}$ such that the function $s\to \int_0^s \sqrt{\phi}$ along $w$ takes values in $\mathbb{R}_{\geq 0}$ independent to the choice of an oriented parametrization of $w$. We can then similarly choose a pair of points $b^u(w)$ and $b^d(w)$ connected by an oriented vertical arc $\alpha$ called a \textit{short path} passing through $b(w)$ such that the integral $s\to \int_{\alpha(0)}^{\alpha(s)} Im\sqrt{\phi}$ is non-negative and increasing. Furthermore, we can give $\pm$ labels for the lifts of $b^{u}(w)$ (or $b^{d}(w)$) by letting $b(w)^{u,+}$ (or $b(w)^{d,+}$) to be the lift corresponding to the positive sheet of $\sqrt{\phi}$ along $w$. 

Let $\mathcal{P}_C$ be the resulting collection of points $b^{u}(w)$ and $b^d(w)$ for $w$ a wall in $\snetwork(0$). Let $\mathcal{P}_{\scurve^{\circ}}=\pi^{-1}(\mathcal{P}_{C^{\circ}})$, and lift the wall-chamber decomposition of $C$ to a wall-chamber decomposition of $\scurve$. Recall that we call a $GL(1;\mathcal{R})$-local system on $\scurve$ \textit{almost flat} if the monodromy around a ramification point is $-Id$. Similar to Definition \ref{pathgropoidrepresentationdefinition}, we introduce a path-groupoid representation analogue of an almost flat $GL(1;\mathcal{R})$-local system introduced in \cite[Section 4.2]{SpectralNetworkFenchelNielsen}. 
\begin{definition}\label{almostflatpathgroupoid}
	A path groupoid representation of an almost flat  $GL(1;\mathcal{R})$-local system $\mathcal{L}$ on ${\scurve^{\circ}}$ is a collection of the following data:
	\begin{itemize}
		\item A one-dimensional free $\mathcal{R}$-module $\mathcal{L}_{\tilde{b}}$ for each of the points $\tilde{b}\in \mathcal{P}_{\scurve^{\circ}}$ with a preferred choice of basis.
		\item 
		A morphism of vector spaces
		\[\Phi^\mathcal{L}(\alpha): \mathcal{L}_{\tilde{b}}\to \mathcal{L}_{\tilde{b}'}\] 
		given a morphism $\alpha\in Hom(\tilde{b},\tilde{b}')$ of the path groupoid $\mathcal{G}_{\scurve^{\circ}}=\mathcal{G}_{\scurve^{\circ}}(\mathcal{P}_{\scurve^{\circ}})$. %Here, the path groupoid $\mathcal{G}_{\scurve^{\circ}}$ consists of objects corresponding to the elements of $\mathcal{P}_{\scurve^{\circ}}$ and the morphisms corresponding to the $\Sigma^{\circ}_{\phi}$-homotopy classes of paths between $\tilde{b},\tilde{b}'\in \mathcal{P}_C$. 
	\end{itemize}
	This data is subject to the following conditions:
	\begin{itemize}
		\item The morphisms $\Phi^\mathcal{L}(\alpha)$ are compatible with composition of path homotopy classes.
		\item  The holonomy around a based loop encircling a ramification point of $\pi$ is $-Id$. 
	\end{itemize}
	
\end{definition}
We now define non-abelianization.
\begin{definition}\label{NonAb}
	Given a path groupoid representation $\mathcal{L}$ of an almost flat $GL(1;\mathbb{C})$-local system on $\scurve^{\circ}$ and a path groupoid representation $E$ of a $GL(2;\mathbb{C})$-local system on $\tilde{C}$, we say that $\mathcal{L}$ and $E$ form a \textit{$\mathcal{W}$-pair}, or equivalently that $E$ is a non-abelianization of $\mathcal{L}$, if:
	\begin{itemize}
		\item There is an isomorphism
		\[i_b:E_b\to \pi_{\ast}(\mathcal{L})_b \]
		for each $b\in \mathcal{P}_C$.
		\item If $\alpha$ does not cross walls of $\snetwork(0)$, then 
		\begin{align}\label{pushforwardoutside}
			\Gamma(\alpha)=i_{f(\alpha)}^{-1}(\pi_{\ast} \Phi^\mathcal{L}(\alpha)) i_{i(\alpha)}.
		\end{align}
		\item If $\alpha$ is a short path between $b(w)^{-}$ and $b(w)^{+}$, then
		\begin{align}\label{wallcrossingterm}
			\Gamma(\alpha)=i_{f(\alpha)}^{-1}(Id+\mu_w)(\Phi^\mathcal{L}(\alpha)) i_{i(\alpha)}
		\end{align}
		where $\mu_w$ is some $\mathbb{C}$-morphism
		\begin{align}\label{wallcrossingcount}
			\mu_w:\mathcal{L}_{b(w)^{d,-}}\to \mathcal{L}_{b(w)^{u,+}}.
		\end{align} 
	\end{itemize}
	Furthermore, we say that the induced local systems on $\tilde{C}$ and ${\Sigma}_{\phi}$ form a $\mathcal{W}$-pair if their path groupoid representations form a $\mathcal{W}$-pair. 
\end{definition}
One of the main insights of \cite{GNMSN} was that homotopy invariance and $\mathcal{L}$ uniquely determine the matrices $\mu_w$. We will revisit this idea in Section \ref{Wall-Crossing Analysis Section}. \newline \\
Consider $\scurve$ as a real Lagrangian submanifold in $T^{\ast}\tilde{C}$ with respect to the real canonical Liouville form $\canliouvile$ under the identification of the real cotangent bundle and the holomorphic cotangent bundle. We are interested in complete GMN quadratic differentials $\phi$ such that the corresponding spectral curve $\scurve$ is \textit{exact} with respect to the canonical \textit{real} Liouville form on $T^{\ast}\tilde{C}$. We call such quadratic differentials \textit{real exact}. The space of real exact quadratic differentials constitutes a totally real submanifold of the space of quadratic differentials (see the remark after Proposition \ref{realexactnesscrit}; the space of GMN quadratic differentials is a complex manifold by \cite[Theorem 4.12]{bridgeland2014quadratic}).

We show in Section \ref{continuationmapsss} that given a real exact quadratic differential $\phi$, the Floer cohomology local system  
\[	HF_\epsilon(\scurve,\mathcal{L},\mathfrak{s},\mathcal{B};\mathbb{C}): z\mapsto HF(\epsilon\scurve,F_{z},\mathfrak{\tilde{s}},\mathfrak{f}_z,\mathcal{L}\otimes \mathcal{B};\mathbb{C})\]
is well-defined. The construction of the precise Floer-theoretic set-up uses only standard techniques but is slightly involved. This is carried out in Section \ref{Floer Theory on Open Manifolds Sect 1}. 

We also show that the points $\mathcal{P}_{C}$ and the $\mathbb{C}$-vector spaces  $HF(\epsilon\scurve,F_{z},\mathfrak{\tilde{s}},\mathfrak{f}_z,\mathcal{L}\otimes \mathcal{B};\mathbb{C})$ for $z\in \mathcal{P}_{C}$, along with the Floer-theoretic continuation maps, define a path groupoid representation $HF_\epsilon(\scurve,\mathcal{L},\mathfrak{s},\mathcal{B},\mathcal{P}_C;\mathbb{C})$ of a $GL(2;\mathbb{C})$-local system over $\tilde{C}$. In addition, we show that compact Hamiltonian isotopy of $\epsilon\scurve$ which are supported away from the points in $\pi^{-1}(\mathcal{P}_C)$ define an equivalent path groupoid representation (Proposition \ref{pathgroupoidrepresentation}). 
\\
\newline
We can now restate our main theorem as follows. Note that there are some constants involved for technical reasons. 
\begin{theorem}\label{FullNon-Abelianization} Let $\scurve$ be the spectral curve associated to a real-exact GMN quadratic differential on a closed Riemann surface $C$. Given a small deformation parameter $\delta>0$ and a large energy cut-off $E\gg 1$, there exists an $\epsilon_0>0$ and a collection of points $\mathcal{P}_C=\mathcal{P}_C(\delta;E)$ (with lifts $P_{\scurve^{\circ}}$) such that the following holds for all $0<\epsilon<\epsilon_0$.
	
	Let $\mathcal{L}=\mathcal{L}(P_{\scurve^{\circ}})$ be a path groupoid representation of an almost flat $GL(1;\mathbb{C})$-local system, $\mathfrak{s}$ be a spin structure on $C$, and $\mathcal{B}$ be an almost flat $GL(1;\mathbb{Z})$-local system. Then
	$HF_\epsilon(\scurve,\mathcal{L},\mathfrak{s}, \mathcal{B}, \mathcal{P}_C;\mathbb{C})$ and $\mathcal{L}(P_{\scurve^{\circ}})$ form a $\mathcal{W}$-pair, or equivalently, $HF_\epsilon(\scurve,\mathcal{L},\mathfrak{s},\mathcal{B};\mathbb{C})$ is a non-abelianization of $\mathcal{L}$.
\end{theorem}
\subsection{Towards the proof of Theorem \ref{FullNon-Abelianization}}\label{Proofguidess}
%The fundamental ingredient behind the proof of Theorem \ref{FullNon-Abelianization} is the adiabatic degeneration method. Suppose we have a family $u_\epsilon$ of pseudo-holomorphic curves in $T^{\ast}M$ such that as $t\to 0$, the sequences $u_\epsilon$ degenerate to sets of solutions of flow-like local differential equations on $M$.  
Having reviewed the theory of quadratic differentials and non-abelianization, we now outline the strategy towards the proof of Theorem \ref{FullNon-Abelianization} from \cref{maintheorem}. Let $z\in \mathcal{P}_C$. By choosing a suitable grading, we show that the chain complex $CF(\scurve,F_z)$ is concentrated in degree $0$. Thus the intersection points in $F_z\pitchfork \scurve$ give rises to a natural decomposition of $HF(\scurve,F_z)$ for $z\in \mathcal{P}_C$. The key part is using Theorem \ref{maintheorem} to show that the parallel transport is \textit{diagonal} along (homotopy classes of) paths that are strictly contained in a connected component of $C(\delta;E)$. This is necessary to show that \eqref{pushforwardoutside} holds. To understand the heuristics, consider the holomorphic strips that contribute to the \textit{non-diagonal terms} in the Floer-theoretic parallel transport map along horizontal or vertical arcs in $C(\delta;E)$. We show that these strips, Gromov-converge to broken holomorphic strips bounded between $F_z$ and $\epsilon\scurve$ as the corresponding arcs converge to the point $z$. Since Theorem \ref{maintheorem} implies that such holomorphic strips cannot exist, we deduce that the parallel transport must be diagonal. 

We now explain why the parallel transport along a short arc is of the form \eqref{wallcrossingterm}. This is essentially due to the positivity of energy. From real exactness, we have $\primitive:\scurve\to \mathbb{R}$ such that $\canliouvile=d\primitive$. From Stokes' theorem,  we see that the energy of $\epsilon$-BPS discs ending at $z$ must be bounded above by $\pm \epsilon (\primitive(z^{+})-\primitive(z^{-}))$. We show that $\primitive(z^{+})=\primitive(z^{-})$ if and only if $z$ lies on the spectral network $\snetwork(\pi/2)$. Thus for $z\notin \snetwork(\pi/2)$, we can choose the ordering $z^{+}$ and $z^{-}$ in such a way that $\primitive(z^{+})>\primitive(z^{-})$. This means that there are no $\epsilon$-BPS discs ending at $z$ that travel from $z^{+}$ to $z^{-}$, regarding them as $J$-holomorphic strips.  A similar energy argument applies to show that the parallel transport \eqref{wallcrossingterm} should be strictly upper-triangular for short enough $\alpha$. We leave the discussions on holonomy contributions of $\mathcal{L}$ in \eqref{pushforwardoutside} and \eqref{wallcrossingterm} to Section \ref{gradings and spin structures}. 

It is worth remarking here that similar applications of Morse flow-tree techniques to study the degeneration of holomorphic discs for (certain generalizations of) spectral curves also appeared in \cite{SSquantummirror} and implicitly in \cite{CasalsDGAcubic}.

\section{Floer theory on cotangent bundles of open manifolds}\label{Floer Theory on Open Manifolds Sect 1}
The main aim of this section is to establish a Floer-theoretic set-up on  $T^{\ast}\tilde{C}$ such that the Floer cohomology local system $HF(\epsilon \scurve,F_z)$ is well-defined on $\tilde{C}$. To do this, we define Floer theory on cotangent bundles of the more general class of Riemannian open manifolds that are ``flat at infinity". 

In Section \ref{Flatness at infinity and finiteness conditions} we introduce the notion of flatness at infinity, define finiteness conditions for Lagrangians, Hamiltonians and almost complex structures. In particular, we will introduce the class of\textit{ vertically finite Lagrangians}, which includes spectral curves associated to GMN complete quadratic differentials $\phi$. In Section \ref{Geometric Boundedness}, we review the notion of geometric boundedness. In Section \ref{Monotonicity Techniques}, we discuss the basic monotonicity techniques. In Section \ref{Compactness and Transversality}, we show using the monotonicity techniques and the arguments in \cite[Section 3]{GPSCV} that the moduli space of Floer solutions satisfy the usual compactness and transversality properties. The key is showing that the relevant pseudo-holomorphic curves do not escape off to infinity. Here, the boundary conditions are given with respect to the classes of Lagrangians and almost complex structures defined in Section \ref{Flatness at infinity and finiteness conditions}. This allows us to define, for instance, $CF(\epsilon \scurve,F_z)$. In Section \ref{Floer Moduli SpaceS}, we show that the Floer chain complex satisfies certain invariance properties up to isomorphism in cohomology. This section is heavily based on the works of Sikorav\cite{Audin1994HolomorphicCI}, Groman\cite{groman2021floer,groman2019wrapped} and Ganatra-Pardon-Shende\cite{GPSCV}. 

\subsection{Flatness at infinity and finiteness conditions}\label{Flatness at infinity and finiteness conditions}

We start with the following definition.

\begin{definition}\label{flatatinfinitydefinition}
	A Riemannian manifold $(M,g)$ is \textit{flat at infinity} if $g$ is complete,  there exists a compact subset $K\subset M$ such that $g\vert_{M-K}$ is flat, and there exists an $r_g>0$ such that the injectivity radius of $g$ is bounded below by $r_g$.  
\end{definition}

\begin{exmp}
	The real line $\mathbb{R}$ equipped with the standard flat metric. We will also see in Section \ref{Flat at infinity} that $\tilde C$ equipped with the flat metric desingularized at the branch points is also flat at infinity. 
\end{exmp}	

Consider the cotangent bundle $T^{\ast}M$. Since $M$ is non-compact, it is not a Liouville manifold, but it is very close to being one. $T^{\ast}M$ admits the canonical Liouville form $\canliouvile=\textbf{p}\cdot\textbf{dq}$ and the canonical symplectic form $\omega=\textbf{dp}\wedge \textbf{dq}$. Furthermore, the standard Liouville vector field $Z:=p\partial_p$ satisfies $L_Z\omega=\omega$ and $\iota_Z \omega=\canliouvile$. In addition to this, given any metric $g$ on $M$, the unit sphere bundle $S^{\ast}M$ is a codimension 1 submanifold of $T^{\ast}M$ and the restriction of the Liouville form defines a contact form $\alpha$ on $S^{\ast}M$. \\
Consider the diffeomorphism of the positive cone of $(S^{\ast}M,\alpha)$ into $T^{\ast}M$ 
\[[1,\infty)\times S^{\ast}M\to T^{\ast}M\]
given by sending the point $(r,(\textbf{q},\textbf{p}))$ to its time-$\log(r)$ image under the Liouville flow. Since this is simply the map
\[(r,(\textbf{q},\textbf{p}))\to (\textbf{q},r\textbf{p}),\]
the pullback of the canonical Liouville form $\canliouvile=\textbf{p}\textbf{dq}$ is equal to $r\alpha$ and the canonical symplectic form reads $d(r\alpha)$ on the positive cone. The Liouville vector field then takes the form $r\frac{d}{dr}$ over $[1,\infty)\times S^{\ast}M$. We see that the positive flow of the Liouville vector field is complete and that the image of $[1,\infty)\times S^{\ast}M$ covers the neighbourhood of the vertical infinity. 

The Reeb field $R$ over $S^{\ast}M$ is the unique vector field defined by the condition $\alpha(R,-)=1$, $d\alpha(R,-)=0$. At $(\textbf{q},\textbf{p})\in S^{\ast}M$, in geodesic normal coordinates, the Reeb vector field reads:
\[R:=\sum_{i=1}^n \frac{p_i}{\abs{p}}\frac{\partial }{\partial q_i}.\]
The Reeb field over $[1,\infty)\times S^{\ast}M$ is defined as the field $0\oplus R\subset TS^{\ast}M$. This is the Hamiltonian vector field associated to the linear function $r$. \newline

We now introduce objects which are compatible with the Liouville-like structure. We say that an almost complex structure or a Lagrangian submanifold is \textit{cylindrical} if it is invariant under the positive Liouville flow. Furthermore,

\begin{definition}\label{generalcontact}
	An $\omega$-compatible almost complex structure $J$ is of \textit{general contact type} if there exists a positive smooth function $h:\mathbb{R}_{>0}\to \mathbb{R}_{>0}$ such that
	\[h(r)dr=\canliouvile \circ J.\]
	If this condition holds over $\{r>R\}$ for some $R\gg 1$, we say that it is of general contact type at vertical infinity.
	
	The almost complex structure $J$ is of \textit{contact type} if $h(r)=1$ and of \textit{rescaled contact type} if $h(r)=r$. 
\end{definition}
The notion of almost complex structures of rescaled contact type comes from \cite[Section 3.1 (3.18)]{Ganatrathesis}. 
Definition \ref{generalcontact} is equivalent to $J$ mapping the kernel of $\alpha$ to itself on the level sets of $r$ and swapping the Liouville flow $Z$ with $h(r)R$. \newline
\\
In this paper, we will utilize the "canonical" $\omega$-compatible almost complex structure on $T^{\ast}M$ induced from the metric $g$,  called the \textit{Sasaki almost complex structure} (See \cite[Section 4.4]{Morseflowtree}, or for the full exposition, \cite[Chapter 4, Chapter 8]{TangentBundleYano}). To define this, first we note that given the projection $\pi: T^{\ast}M\to M$, the  kernel $V$ of the derivative $d\pi:TT^{\ast}M\to M$ gives the canonical \textit{vertical distribution} on $T^{\ast}M$. Then the metric $g$ gives rises to a distribution $H$ on $T^{\ast}M$ called the \textit{horizontal distribution} for which the restriction $d\pi:H_p\to T_{\pi(p)}M$ for $p\in M$ gives a vector space isomorphism. We then identify $H$ with $(\pi^{\ast}TM,g)$ via $d\pi$; we have the following covariant decomposition 
\begin{align}
	TT^{\ast}M&=H\oplus V\\
	&=(\pi^{\ast}TM,g)\oplus (T^{\ast}M,g)\label{horverdecom},
\end{align}
of $TT^{\ast}M$. Regarding $g$ as a real vector bundle isomorphism $g:TM\to T^{\ast}M$, we get:
\begin{definition}\label{sasakialmostcomplexstructuredefini}
	The Sasaki almost complex structure $J_g$ is the almost complex structure on $T^{\ast}M$ defined by the following matrix
	\begin{align}J_g:=
		\begin{bmatrix}
			0 & +g^{-1} \\ 
			-g & 0
		\end{bmatrix},
	\end{align}
	with respect to the covariant decomposition \eqref{horverdecom}. We write $g^S$ for the metric on $T^{\ast}M$ induced from $\omega$ and $J_g$.   
\end{definition}

The Sasaki almost complex structure is not of contact type at infinity, so we deform the almost complex $J_g$ as in \cite[Section 8.1]{cieliebak2017knot} to find some \textit{conical deformations} of $J_g$. The same conical deformation also appeared in \cite[Section 5.1.3]{NadlerZaslowconstructiblesheaves}.

\begin{definition}
	Let $\rho:[1,\infty)\to [1,\infty)$ be a smooth increasing positive function such that $\rho(r)=1$ for $r<3/2$ and $\rho(r)=r$ for $r\gg 2$. The following deformation of the Sasaki almost complex structure
	\begin{align}
		J_{con}=\begin{bmatrix}
			0 & +\rho(r)^{-1} {g}^{-1} \\ 
			-\rho(r) g &0.
		\end{bmatrix},\label{deformSasaki}
	\end{align}
	is called the ($\rho$-)\emph{conical deformation} of $J_g$. We write $g_{con}$ for the Riemannian metric induced from $\omega$ and $J_{con}$. 
\end{definition}
Here the matrix is taken with respect to the decomposition (\ref{horverdecom}). Fixing a smooth $\rho$ once and for all, we obtain our \textit{background almost complex structure} $J_{con}$ and our \textit{reference metric} $g_{con}$ on $T^{\ast}M$. The following proposition is due to {\cite[Section 8.1]{cieliebak2017knot}}.

\begin{proposition}\label{proposition:contacttype}
	Let $\canliouvile$ be the canonical Liouville form on $T^{\ast}M$. The Sasaki almost complex structure is of rescaled contact type. The deformed almost complex structure is invariant under the Liouville flow and satisfies
	\[\canliouvile \circ J_{con}=dr\label{contact type prop}\]
	for $r\gg 1$. Hence, it is of contact type at infinity.
	
\end{proposition}

We introduce the class of \textit{horizontally finite} Hamiltonians and Lagrangians.\\
\begin{definition}\label{horizontalfinitenessdefini}
	Let $(M,g)$ be a Riemannian manifold which is flat at infinity. Equip the cotangent bundle $T^{\ast}M$ with its background almost complex structure $J_{con}$.
	\begin{itemize}
		\item Let $L$ be a Lagrangian submanifold in $T^{\ast}M$ which is cylindrical at infinity. We say that $L$ is \textit{horizontally finite} if $\pi(L)\subset K$ for some compact subset $K\subset M$.
		\item  Let $H$ be a Hamiltonian function on $T^{\ast}M$. We say that it is cylindrical if $ZH=H$ at infinity, or equivalently, if $H=hr$ for $r\geq R, R\gg 1$, where $h:S^{\ast}M\to \mathbb{R}$ is a contact Hamiltonian.  We say that $H$ is \textit{horizontally finite} if there exists a compact subset $K\subset M$ such that the support of $H$ lies inside $T^{\ast}K$.\\ %We say compact subset $\bar{K}$ is called the horizontal support of $H$. 
	\end{itemize}
\end{definition}

We restrict to the following class of almost complex structures on $T^{\ast}M$. 
\begin{definition}		
	Let $J$ be an $\omega$-compatible almost complex structure. We say that $J$ is an \textit{admissible} almost complex structure if $J$ is cylindrical at infinity and if there exists a compact subset $K\subset M$ such that $J=J_{con}$ outside of $T^{\ast}K$. We say that $K$ is the horizontal support of $J$. 
\end{definition}

Let $\mathcal{J}(T^{\ast}M)$ denote the space of $\omega$-compatible admissible almost complex structures. Let $S$ be a Riemann surface with boundary. A family of admissible almost complex structures parametrized by $S$ is a smooth map
\[J:S\to \mathcal{J}(T^{\ast}M).\]  Here $S$ is either compact with boundary, or with strip-like ends. We will be concerned with a family of almost complex structures that is uniform in the following sense.

\begin{definition}
	Let $J:S\to \mathcal{J}(T^{\ast}M)$ be a family of admissible almost complex structures, then $J$ is \emph{uniformly cylindrical}, if there exists a subset of $S\times T^{\ast}M$, which is proper over $S$, such that outside of this subset, the almost complex structures $J_{s}\vert_{s\in S}$ are invariant under the Liouville flow.  A family of admissible almost complex structures is called \emph{uniformly admissible} if there exists a uniform horizontal support, and if the family is uniformly cylindrical at infinity.
\end{definition}
\begin{figure}[t]
	\includegraphics[width=0.8\textwidth, height=7cm]{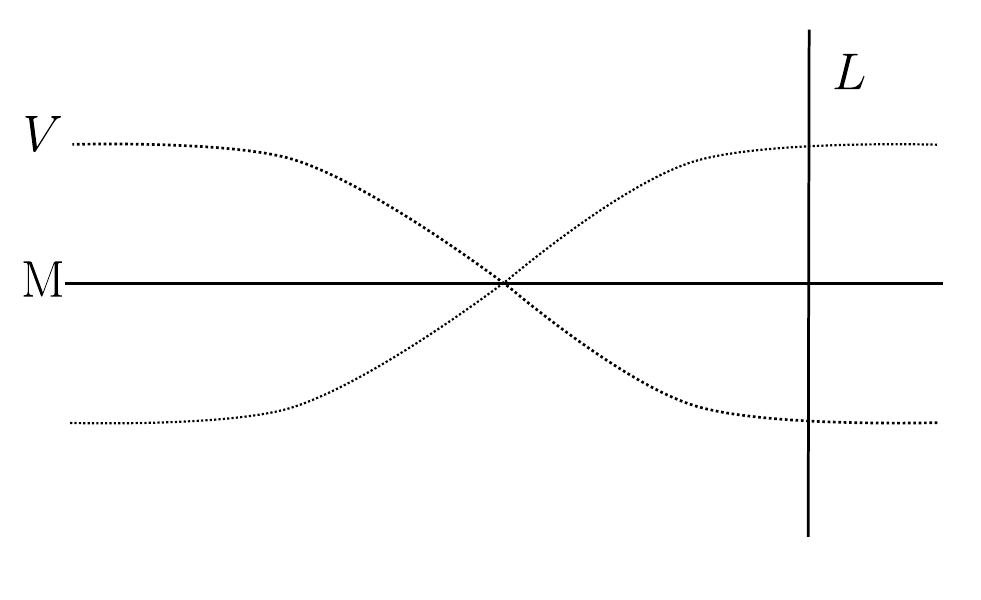}
	\centering
	\caption{A cross-section of $T^{\ast}M$: the zero section $M$; a vertically finite Lagrangian $V$; a horizontally finite Lagrangian $L$. }
	\label{finfig}
\end{figure}	
%To introduce the notion of vertically finite Lagrangians as in the introduction, we need a small lemma.

%\begin{lemma}
%Let $B_1$ be the open unit disc in $\mathbb{R}^n$ equipped with the standard Euclidean metric. Let $T^{\ast}B_1$ be equipped with the Sasaki metric induced from the flat Euclidean metric on $B_1$. Let $L$ be a totally geodesic Lagrangian submanifold of$T^{\ast}M$ which is closed in the subspace topology and suppose the projection $L\to M$ is proper. Suppose furthermore $L$ belongs in $\{r<1\}$. Then $L$ consists of disjoint unions of finite affine planes $L_1,...,L_n$.
%\end{lemma}  
%\begin{proof}
%First, note that $T^{\ast}B_1$ equipped with the Sasaki metric induced from the flat Euclidean metric on $B_1$ is just $B_1\times \mathbb{R}^{n}$. with the Euclidean metric. Any totally real geodesic submanifold of $B_1\times \mathbb{R}^{n}$ is a disjoint union of affine planes. Suppose there are infinitely many sheets of $L$. For $0\in B_1$, let $x_1,...$ denote the lifts of $0$ to $L$.  
%	\end{proof}
We now introduce the notion of vertically finite Lagrangians.	
\begin{definition}\label{vertifin}
A properly embedded Lagrangian submanifold $L$ in $T^{\ast}M$ is vertically finite if there exists an $R\gg 1$, $r_L>0$ and a compact subset $K_L\subset M$ such that the following holds.
\begin{itemize}
	\item (Finite height) $L$ is contained in $D_R^{\ast}M$.
	\item (Singularities of $d\pi$ are away from infinity) The complement $M-K_L$ is an open submanifold of $M$ and outside of $T^{\ast}K_L$, the projection $\pi:L\to M$ is a proper finite covering map
	%	\item (finitely many sheets at infinity) the space $\pi^{-1}(K_L)\cap L$ is a manifold with boundary and consists of finitely many connected components, 
	\item (Flatness) The submanifold $L\cap T^{\ast}(M-K_L)$ is totally $g_{con}$-geodesic and contained in the subset $D_1^{\ast}M$.
	\item (Finite vertical gap) For all $x\in M-K_L$, and $x'\in \pi^{-1}(x)$ , $B^{g_{con}}_{r_L}(x')\cap L{\vert}_{(M-K_L)}$ is connected. 
\end{itemize}
We say that a Lagrangian is \emph{finite at infinity} if it is either horizontally finite or vertically finite. 
%We say that a totally real submanifold $W$ which is Lagrangian at infinity is finite at infinity, if its Lagrangian part is. 
\end{definition}

%\begin{remark}
%	We may regard the constant $r''>0$ in Definition \ref{vertifin} as giving "gap bounds" for the sheets of $L$ over $M-K$. To understand the behaviours of $L$ at infinity, we make the following observations. First, note that we can assume that the compact subset $K$ is big enough so that $g_{M-K}$ is flat, and that $M-K$ is an open submanifold of $M$. For the sake of simplicity, assume that $M-K$ is also connected. Take the universal cover $\widetilde{M-K}$ of $M-K$, then $\widetilde{M-K}$ is a simply connected flat Riemannian manifold so it is isometric to an open unbounded, contractible subset of $\mathbb{R}^n$ equipped with the standard Euclidean metric. The induced submanifold $\tilde{L}$ of $T^{\ast}\mathbb{R}^n$ is totally geodesic and so it must be a linear Lagrangian plane with finite height. This happens if and only if the $\textbf{p}$-coordinates stay constant. It follows that $L$ on $M-K$ is a union of local translates of $M-K$. \end{remark}
We will show in Corollary \ref{spectralcurvevertifin} that spectral curves associated to complete GMN quadratic differentials are vertically finite. Note that on $D_1^{\ast}M$, $g_{con}$ coincides with $g^S$. This is why we had $\rho(r)=1$ in the neighbourhood of $ST^{\ast}M$.
%Let $\mathcal{J}(X)$ be the space of admissible almost complex structures. Let $S$ be a Riemann surface with corners. We consider a family of almost complex structures in $\mathcal{J}(X)$ parametrized by points on $s\in S$ that are uniformly admissible. Similarly, let $\mathcal{J}^{\tau}(X)$ be the space of $\tau$-admissible almost complex structures. 
\subsubsection{Geometric boundedness}\label{Geometric Boundedness}
We review the notion of geometric boundedness and tameness for almost complex manifolds $(V,J)$ and totally real submanifolds of $V$, following \cite{Audin1994HolomorphicCI,GPSCV,groman2019wrapped}. This will be necessary to  control the $C^0$ images of pseudo-holomorphic curves using monotonicity techniques. Recall that an almost complex manifold $(V,\omega,J)$ equipped with a symplectic form $\omega$ such that $J$ is $\omega$-compatible is called \textit{almost K\"{a}hler}, and a half-dimensional submanifold $W$ of $V$ such that $JTW\cap TW=0$ is called \textit{totally real}. In this paper, we will only consider the following class of totally real submanifolds.

\begin{definition}\label{def:Lagatinfinity}
A properly embedded totally real submanifold $W$ is \emph{Lagrangian at infinity} if $W$ is Lagrangian outside a compact subset. 
\end{definition}
The following definition of geometric boundedness is due to Ganatra-Pardon-Shende. 
\begin{definition}\cite[Definition 2.42]{GPSCV}\label{actualgeometricboundednesscondition}
Let $(V,\omega,J)$ be a $2n$-dimensional almost K\"{a}hler manifold equipped with a symplectic form. We say that $(V,\omega,J)$ is geometrically bounded if there is an open cover $\{V_{\alpha}\}$ of $V$ and charts $\phi_{\alpha}: B_1(0)\subset \mathbb{R}^{2n}\to U_{\alpha}$ such that:
\begin{itemize}
\item the collection $\{\phi_{\alpha}(B_{1/2}(0))\}$ also covers $V$,
\item with respect to the standard metric on $B_1(0)$, 
\begin{align}
	&\sup_{\alpha}\norm{\phi_{\alpha}^{\ast}J}_{C^r}<\infty,\\
	&\sup_{\alpha} \norm{\phi_{\alpha}^{\ast}\omega}_{C^r}<\infty,
\end{align}
\item there exists some $r_0>0$ such that 
\[\omega(v,(\phi_{\alpha})^{\ast}Jv)>r_0 g_{std}(v,v).\]
\end{itemize}
Furthermore, we say that a submanifold $W$ of $V$ which is Lagrangian at infinity is \emph{geometrically bounded} if the charts $\phi_{\alpha}$ can be chosen in a way that $\phi_{\alpha}^{-1}(W)$ is either empty or a linear subspace of $B_1(0)$.

Let $(S,\omega_S,j_S)$ be an almost K\"{a}hler manifold. Suppose we have a family $(V,\omega_s,J_s)$ of almost K\"{a}hler structures over $S$. Then we say that $(V,\omega_s,J_s)$ is \textit{uniformly geometrically bounded} if the almost K\"{a}hler manifold  $(V\times S,\omega_s\oplus \omega_S,J_s\oplus j_S)$ is geometrically bounded. We say that $W$ is uniformly geometrically bounded over $(V,\omega_s,J_s)$ if $W$ is $\omega_s$-Lagrangian outside a fixed compact subset for all $s$, and $(\partial S\times W,\omega_s\oplus \omega_S,J_s\oplus j_S)$ is geometrically bounded.

\end{definition}
From geometric boundedness, one can obtain the \textit{tameness} condition, which is originally due to Sikorav \cite[Definition 4.1.1]{Audin1994HolomorphicCI}.  
%\begin{definition}{\cite[Definition B.5]{groman2019wrapped}}. \label{isoperimetricitydefini}
%Let $(V,g)$ be a Riemannian manifold. We say that $g$ is $(\delta,c)$-isoperimetric at $p$ 
%Here $\ell(\gamma)$ is the length of $\gamma$.
%\end{definition}
\begin{definition}\label{GeometricalBoundednessgeneraldefini}
Let $(V,J,\omega,g=\omega(-,J))$ be an almost K\"{a}hler manifold. We say that $(V,J,\omega,g_J)$ is \emph{tame} if there exist constants $r_V,C_0,C_1,C_2>0$ such that the following holds.
\begin{itemize}
\item The metric is complete, $r_g=\inf_{x\in M} \text{inj}_x g>0$ and $r_V<r_g$.
\item There exists some $R_1,C_1>0$ such that the following holds. Given any $p\in V$, and a closed curve $\gamma:S^1\to B_{R_1}(p)$, there is a disc $D$ in $V$ such that $\partial D=\gamma$ and
\[Area(D)\leq C_1\ell (\gamma)^2.\]
%\item Over each ball $B(p,r_V)$, there exists a local symplectic form $\omega_p$ such that $\abs{\omega_p}_g\leq C_0$. Furthermore,  $\abs{X}_g^2\leq C_2 \omega_p(X,JX)$. 
\end{itemize}
Suppose $(S,\omega_S,j_S,g_S)$ is an almost K\"{a}hler surface. A family of quadruples $(V,J_s,\omega_s,g_s)$ on $S$ is said to be \emph{uniformly tame} if $(S\times V,\omega_S\oplus \omega,j\oplus J_s,g_s\oplus g)$ is tame. 
\end{definition}

%\begin{definition}\label{GeometricalBoundednessdefini}
%Let $(V,\omega,J)$ be an almost K\"{a}hler manifold. Let $g_J=\omega(-,J)$ be the induced metric on $V$. Then $(V,\omega,J)$ is said to be \emph{tame} if $(V,J,g_J)$ is tame with respect to symplectic forms $\omega_p=\omega\vert_{B(p,r_V)}$. 
%A family of almost K\"{a}hler structures $(V,\omega_s,J_s)$ parametrized on $S$ is said to be \emph{uniformly tame} if $(S\times V, \omega_S\oplus \omega,j\oplus J_s)$ is tame.
%\end{definition}
We can extend the notion of tameness to totally real submanifolds which are Lagrangian at infinity.
%\begin{definition}{\cite[Definition B.5]{groman2019wrapped}}
%\end{definition}
\begin{definition}\label{totallyrealgeombounddefini}({\cite[Definition 4.7.1]{Audin1994HolomorphicCI}})
Let $(V,J,g)$ be as in Definition \ref{GeometricalBoundednessgeneraldefini}.
Let $W\subset V$ be a properly embedded totally real submanifold of $V$ which is Lagrangian at infinity. Then $W$ is said to be \emph{tame} if there exists an $r_W>0, C_W>0$ such that the following holds. 
\begin{itemize}
\item For $x,y\in W$ with $d(x,y)_V<r_W$, we have
\[d(x,y)_W\leq C_W d(x,y)_V.\]
\item Each $B(r_W,p)\cap W$ is contractible. 
%\item For any chord $\gamma:[0,1]\to B_{r_W}(p)$ with endpoints on $W$, there is a half disc $D$ with $\partial D=\gamma\cup \tilde{\gamma}$ and $\tilde{\gamma}\subset W$, such that
%\[Area(D)\leq \ell (\gamma)^2.\]
\end{itemize}
Given a uniformly tame family $(V,\omega_s,J_s)$ over an almost K\"{a}hler surface $S$, we say that $W$ is \textit{uniformly tame} if $W$ is uniformly Lagrangian at infinity with respect to $\omega_s$ and $\partial S\times W$ is tame in $(S\times V,\omega_s\oplus \omega_S, J_s\oplus j_S,g_s\oplus g_S)$. 
\end{definition}

%\begin{lemma}[\cite{groman2019wrapped}]
%	Suppose $g,L$ are $(\delta,c)$-isoperimetric at some $p\in M$ and let $g'$ be such that for some $a_1,a_2>1$, we have
%	\[a_1^{-1}g(v,v)\leq g'(v,v)\leq a_2 g(v,v)\]
%	for all $x\in M,v\in T_x M$. Then $g',L$ are $(\delta/a_1,a_1^2a_2^2c)$-isoperimetric.
%\end{lemma}

The following well known proposition relates tameness with geometric boundedness.
\begin{proposition} 
Suppose $(V,\omega,J)$ is geometrically bounded, then it is tame. Furthermore, if $W$ is a geometrically bounded totally real submanifold of $V$ which is Lagrangian at infinity, then $W$ is also tame. 
\end{proposition}
\begin{proof}
Groman's estimate \cite[Lemma 4.10]{groman2021floer} gives control over the isoperimetricity constants in terms of the sectional curvature and the injectivity radius. Jean-Philippe Chass\'{e}'s estimate \cite[Lemma 1]{chasse2021convergence} shows tameness for geometrically bounded Lagrangian submanifolds. 
\end{proof}
\begin{remark}\label{controllingtheconstants}
Controlling the injectivity radius requires control over the sectional curvature and the \textit{volume comparison} between the Euclidean volume and the volume induced from $g_J$. This requires theorem \cite[Theorem 4.3]{CheegerJeffGromovinjectivityradius}. In particular, a uniform lower bound on the $g$-volume of the unit ball and an upper bound on the sectional curvature gives a uniform lower bound on the injectivity radius. Controlling the injectivity radius and the cut locus distance (the supremum of the radial radius of the embedded tubular neighbourhood) for Lagrangians was done in \cite[Lemma 2.12]{JboundedGroman} and \cite[Theorem 3.9]{JboundedGroman} respectively.
\end{remark}
The following proposition verifies geometrical boundedness of almost complex manifolds $(T^{\ast}M,\omega,J)$ for $J$ an admissible almost complex structure. This is a modification of \cite[Lemma 2.43]{GPSCV} and we follow their proof closely. 

\begin{proposition}\label{geoboundnonfamily}
Let $J$ be an admissible complex structure on $T^{\ast}M$. Let $g_J$ be the metric induced from $J$ and $\omega$. Then the almost K\"{a}hler manifold $(T^{\ast}M,J,\omega,g_J)$ is geometrically bounded. Furthermore, Lagrangians which are finite at infinity are also geometrically bounded.
\end{proposition}
\begin{proof}
Since admissible almost complex structures are cylindrical at infinity, we may assume that $J$ is cylindrical without loss of generality, over the positive cone over some fixed sphere bundle of $J_g$-radius $R>0$ which depends only on $J$. This $R$ only depends on the auxiliary function $\rho$.

Let $p=(q,r)$ be a point near vertical infinity where $q$ is the corresponding point on the sphere bundle. Take the reverse Liouville flow to bring it down to the point $q$. Since $J$ is invariant under the Liouville flow, we see that the geometry of $(T^{\ast}M,J,\omega)$ near $p$ is the same as the geometry of $(T^{\ast}M,J,r\omega)$ near $q$. But as pointed out in \cite[Lemma 2.43]{GPSCV}, the geometry $(T^{\ast}M,J,r\omega)$ uniformly converges in the Cheeger-Gromov sense to the linear K\"{a}hler geometry at $T_q(T^{\ast}M)$ induced by the triple $(\omega,J,g(0))$ as $r\to \infty$. Hence, it suffices to bound the geometry of the sphere bundle. By the flatness at infinity, the geodesic exponential map is locally an isometry outside some compact subset, with the radius of injectivity uniformly bounded below. From this observation, the boundedness of the geometry of the sphere bundle automatically follows. For $L$ a horizontally finite Lagrangian, it is conical at infinity and the horizontal support is contained in a compact subset of $M$, and so we are back in the situation of  \cite[Lemma 2.43]{GPSCV}.

%Let $r_g$ be as in Definition \ref{flatatinfinitydefinition}. Outside some compact subset of $K$, we can cover $M$ with countably many balls $B_{r_i}(x_i)$ such that $0<R_1'<r_i<r_g$ for some $R_1'>0$. Since the curvature vanishes, the exponential map $\exp_{x_i}:B_{r_i}(0)\to B_{r_i}(x_i)$ is a local isometry. Taking the pullback via the exponential map and using the covariance of $J_{con}$, we see that the unit sphere bundle is trivial and we are simply bounding the geometry of $S^{n-1}\times B_{r_i}(0)$ equipped with the standard metric scaled by $R^{1/2}$ (The $R^{1/2}$ factor only appears because we had taken the sphere near infinity where the almost complex structure becomes conical). This is automatic. 
%Scale via the Liouville flow to bring $p$ very close to scaling via Liouville flow rescales the geometry of points $p$ close to infinity to the geometry of points $q$ over the unit sphere bundle. Under this rescaling limit, the geometry just converges to the tangent space at $q$ with the induced linear almost K\"{a}hler structure.
Suppose now $L$ is a vertically finite Lagrangian. Let $K_L$ and $r_L$ be as in Definition $\ref{vertifin}$. By definition, there exists some compact subset $K_1\subset M$ such that $J=J_{con}$ on $T^{\ast}(B_{r_g}(M-K_1))$ and $K_L\cap B_{r_g}(M-K_1)=\emptyset$. For $x\in M-K_1$, the restriction of the exponential map $\exp_x:B_{r_g}(0)\to B_{r_g}(x)$ is an isometry since the sectional curvature vanishes identically on the image, by flatness at infinity. Consider the induced map 
\[(\exp_x,d(\exp^{-1}_x)^{\ast}):T^{\ast}B_{r_g}(0)\to T^{\ast}B_{r_g}(x).\]
Then
\[J\vert_{D_{3/2}T^{\ast}B_{r_g}(x)}=J_{con}\vert_{D_{3/2}^{\ast}B_{r_g}(x)}=J_g\vert_{D_{3/2}T^{\ast}B_{r_g}(x)}\]
by definition, that by the covariance of $J_g$, 
\[(\exp_x,d(\exp^{-1}_x)^{\ast})^{\ast}J\vert_{D_{3/2}^{\ast}B_{r_g}(x)}=(\exp_x,d(\exp^{-1}_x)^{\ast})^{\ast}J_g\vert_{D_{3/2}^{\ast}B_{r_g}(x)}=J_{g_{std}}\vert_{D_{3/2}^{\ast}B_{r_g}(0)}.\]
Here $g_{std}$ is the standard metric on $\mathbb{R}^{n}$. Of course, the metric induced from $J_{g_{std}}$ is the standard metric on $\mathbb{R}^{2n}$.
%Furthermore,
%\[J_{g_{std}}\vert_{D_{3/2}^{\ast}B_{r_g}(0)}= {J_0}\vert_{{D_{3/2}^{\ast}B_{r_g}(0)}}\]
%regarding ${D_{3/2}^{\ast}B_{r_g}(0)}$ as a subset of $\mathbb{C}^n$, and the induced metric on $D_{3/2}^{\ast}B_{r_g}(0)$ coincides with the standard flat metric on $\mathbb{C}^n=\mathbb{R}^{2n}$. 

Now any totally geodesic submanifold of $\mathbb{R}^{2n}$ equipped with the standard flat metric is a linear subplane of $\mathbb{R}^{2n}$. Furthermore, since $L\to M$ on $B_{r_g}(x)$ is a proper covering, and any covering on a contractible open set is trivial, $(\exp_x,d(\exp^{-1}_x)^{\ast})^{-1}(L)$ consists of finitely many disjoint Lagrangian subplanes of $T^{\ast}B_{r_g}(0)\subset \mathbb{R}^{2n}$. By the final condition in Definition \ref{vertifin}, setting $r'_L=\min\{r_L,1/4,\frac{1}{2}(r_g)\}$, for any $x'\in (\exp_x,d(\exp^{-1}_x)^{\ast})^{-1}(L)$, $B_{r'_L}(x')\cap L$ is connected and consists of a single Lagrangian plane. Furthermore, $(\exp_x,d(\exp^{-1}_x)^{\ast})^{\ast}\omega=\omega_{std}$. This finishes the proof of geometric boundedness of $L$. Note that we have derived tameness of $L$ directly in the proof as well. 
%let $r_{T^{\ast}M}$ be the uniform minimal injectivity radius for the geometrically bounded almost complex manifold $(T^{\ast}M,J,\omega)$  (Definition \ref{GeometricalBoundednessdefini}).  By definition, there exists some $K_1\subset M$ containing $K_L$ such that outside of $T^{\ast}(B_{r_{T^{\ast}M}}K_1)$, $J=J_{con}$. 
%But then $B(x,\epsilon_L)\cap L{\vert}_{(M-K)}$ is connected hence by the covering property, it consists of an open ball of radius $\epsilon_L$ in the affine plane component containing $x$. Since the pullback of $J_{con}$ and $\omega$ by $\exp_x$ just gives the standard almost complex structure and symplectic structure on $\mathbb{C}^n$, we can take these open balls of radius $\epsilon_L$ as the coordinate chart as in Definition \ref{actualgeometricboundednesscondition}. 
\end{proof}

We also need the following ``family" version of the geometric boundedness statement. This is a modification of \cite[Lemma 2.44]{GPSCV}. Recall the conventions in \cref{sec:convention}. 
\begin{lemma}\label{geoboundfamily}
Let $J:\overline{A_1}\to \mathcal{J}(T^{\ast}M)$ be a uniformly admissible family of almost complex structures over $\overline{A_1}$, then $(\overline{A_1}\times T^{\ast}M, j_{\overline{A_1}}\oplus J,\omega_{\overline{A_1}}\oplus \omega_{T^{\ast}M})$ is geometrically bounded. Furthermore, if a totally real submanifold $W\subset T^{\ast}M$ is Lagrangian at infinity, and finite at infinity, then $\partial \overline{A_1}\times W$ is geometrically bounded. 
\end{lemma}
\begin{proof}
Suppose $L$ is vertically finite. Since the family is uniformly admissible, there exists a compact subset $K\subset M$ such that over $T^{\ast}(M-K)$, $J$ agrees with the background almost complex structure $J_{con}^g$. Furthermore, $L$ is totally geodesic outside of some $T^{\ast}K'$ for some compact subset $K'\subset \tilde{C}$. Then the manifold $\partial A_1\times L\vert_{T^{\ast}((K\cup K')^c)}$ is geometrically bounded, and the statement for the compact part of $L$ also follows. For the case where $L$ is horizontally finite, repeat the argument in the proof of Proposition \ref{geoboundnonfamily}.
\end{proof}
\begin{remark}
Replacing a family of almost complex structures on $\overline{A_1}$ with an almost complex structure on $\overline{A_1}\times T^{\ast}M$ is called the \textit{Gromov Trick}. 
\end{remark}
We now focus our attention back to $\tilde{C}$. We first show flatness at infinity. 
\begin{proposition}{\label{Flat at infinity}}
Let $\phi$ be a complete GMN quadratic differential over $C$. 
Let $g$ be a Riemannian metric on $\tilde{C}$ that agrees with the singular metric $g^{\phi}$ outside a compact subset of $\tilde{C}$. Then $(\tilde{C},g)$ is flat at infinity.
\end{proposition}	

\begin{proof}
By definition, the metric $g$ is equal to $g^{\phi}$ in some neighbourhood of the poles. There exists a smaller neighbourhood $U$ of the poles such the points in $U$ are of distance $>1$ away from the zeros of $\phi$. Then for $p\in U$, the locally defined flat coordinate $W=\int \sqrt{\phi(z)}$ can be extended over to an open disc of some radius $\geq 1$. This shows that the minimal injectivity radius is positive. Hence $g$ is flat at infinity. 
\end{proof}	

\begin{corollary}\label{spectralcurvevertifin}
Let $\phi$ be a GMN complete quadratic differential. Let $(\tilde{C},g)$ be as above. Then the pair $T^{\ast}\tilde{C}$ and $\scurve$ equipped with an admissible almost complex structure is \textit{geometrically bounded}. Furthermore, the spectral curve $\scurve$ is vertically finite. 
\end{corollary}	
\begin{proof}
Note that on $D_1^{\ast}M$, $J_{con}=J_{g}$. Since $\scurve$ lies in $D_1 ^{\ast}M$, it suffices to show that outside $T^{\ast}K$ for some compact subset $K$ to $\tilde{C}$ containing the branch points, $\scurve$ is totally geodesic with respect to $J_g$ induced by the flat singular metric $g^{\phi}$. We now show that the spectral curve is vertically finite.

We now check that $\scurve$ satisfies the conditions in \cref{vertifin}. The first two conditions are obvious and so we have to check the flatness and the finiteness of vertical gap. Let $p^{x}$ and $p^y$ denote the dual coordinates in the $W$-coordinate system. Outside of a compact set in $\tilde{C}$, the metric on the base equals $dW^2$ and the spectral curve reads $\{p^x=\pm 1, p^y=0\}$ in the $W=\int \sqrt{\phi}$ coordinate system. Hence we may take the vertical sheet gap to be $r_{\scurve}=\frac{1}{2}$. 
\end{proof}

\subsubsection{Monotonicity techniques}\label{Monotonicity Techniques}

Now we introduce monotonicity techniques and apply them to find a priori restriction on the diameter of the Floer trajectories. We first start with the statement of the monotonicity lemma from \cite[Proposition 4.31, ii)]{Audin1994HolomorphicCI},\cite[Lemma B.7]{groman2019wrapped}.

\begin{lemma}[\textit{Monotonicity Lemma}] 
Suppose $(V,J,g,W)$ is tame, $0<\delta<r_W$ and let $C=1/4(C_1+1+C_W)$ (See \cref{totallyrealgeombounddefini}). Then any $J$-holomorphic curve $u$ passing through $p$ with boundary in $M-B_{\delta}(p)$ satisfies
\[Area(u;u^{-1}(B_{\delta}(p)))=\int_{u^{-1}(B_{\delta}(p))}\frac{1}{2}\abs{du}^2\geq \frac{\delta^2}{2C}.\]
If $p\in W$, the same holds if $\partial u \cap B_{\delta}(p)\subset W$. 
\end{lemma}
%\begin{rem}\label{tameisoperimetri}
%For tame manifolds, we can set $\delta=r_W$ and .
%\end{rem}

From the monotonicity lemma, one can derive the following $C^0$ \textit{estimates} on the image of the $J$-holomorphic curves. The first is the \textit{boundary estimate}, which is \cite[Proposition 4.7.2 iv)]{Audin1994HolomorphicCI}.

\begin{proposition}[Boundary Estimate] \label{totaldomainestimate}
Let $(V,J,\omega,g)$ be a tame manifold. Let $W$ be a totally real submanifold of $V$. Let $S$ be a connected Riemann surface with boundary and let $K$ be a compact subset of $V$, then there exists a constant $C_5(W,K,E)>0$ with the following property.

\begin{itemize}
\item []Let $u:S\to V$ be a $J$-holomorphic map with $Area(u)<E$ such that $u(S)\cap K\neq \emptyset$ and $u(\partial S)\subset K\cup W$, then $C\subset B_{C_5(W,K,E)}(K)$. Here $C_5(W,K,E)$ can be chosen to depend linearly on $E$ and $r_W^{-1}$.
\end{itemize}
\end{proposition}
\begin{rem}
In general, sending $r_W\to 0$ will make $C_5\to +\infty$. However, suppose we have a $C^{\infty}$-convergent family $W_{\epsilon}$ such that it is Lagrangian outside a fixed compact subset, $C_{W_{\epsilon}}$ as in \cref{GeometricalBoundednessgeneraldefini} is uniformly finite, and satisfies $r_{W_{\epsilon}}=O(\epsilon)$. Suppose furthermore we have a family of holomorphic curves $u_{\epsilon}$ with boundary on $W_{\epsilon}$ satisfying $E(u_{\epsilon})=O(\epsilon)$. Then $C_5(W,K,E)$ for $u_{\epsilon}$ can be chosen to be uniformly \textit{finite}. This situation occurs, for instance, when we study the degeneration of pseudo-holomorphic discs with boundary on the scalings of a multi-graph outside its caustics. This point will later reappear in \cref{{Adia Degen}}. 
\end{rem}
%prove by contradiction. Suppose there doesn't exist $x_1\in C$ such that $x_1\in \partial B_{x_0}(K)$. Then since $C\cap K\neq \emptyset$ we see that $u(C)\subset B_{x_0}(K)$ by connectedness. So we are done. Hence suppose such an $x_1$ exists. Now, suppose we do not have a point $x_2$ such that $u(x_2)
%$x_1\in K\cap \text{int}(S)$, $x_i\in  \partial B(K,2ir_0)$, $B_{r_W}(u(x_i))\cap u(\partial C)\subset W$ for $i=1,...,N$ and $d(u(x_i),d(u(x_j))>2r_W$ whenever $i\neq j$. 

%Then $C$ is contained in $(N+2)r_0$-neighbourhood of $K$. To see this, suppose $x\in S$ maps to $\partial B_{(N+1)r_W}(K)^c$. Then:
%\begin{enumerate}
%	\item Suppose $x$ satisfies $B(u(x),r_W)\cap u(\partial C)\subset W$. Hence $u(x)$ must lie in a $2r_W$ neighbourhood of some $x_j$.  contradicts the fact that $x\in B_{(N+2)r_W}(K)^c$.
%	\item Suppose it does not satisfy $B(u(x),r_W)\cap u(\partial C)\subset W$. Then we must have $B(u(x),r_W)\cap  u(\partial S)\subset K$. This is again, a contradiction. 
%\end{enumerate}

The second is the \textit{interior} estimate {\cite[Theorem B.8]{groman2019wrapped}} by Groman. For the proof, see \cite[Theorem 4.11]{groman2021floer}. 

%	\begin{rem}
%		Note that a vertically finite Lagrangian $L$ is uniformly isoperimetric for $\delta<\min(K,r)$ where $r$ is the radius of the sphere bundle that the Lagrangian lies in at infinity and $K$ is some finite constant that is determined from the geometry of $L$ over the compact part. Otherwise, the argument breaks down since we may have two points which lie on different sheets of $L$. In other words, we need a finite \textit{lower bound} on the minimal distance between two points lying on different sheets. 
%	\end{rem}

\begin{proposition}[Interior Estimate]\label{interiorestimate1}
Let $(V,J,\omega,g)$ be a tame almost K\"{a}hler manifold. Let $L$ be a tame Lagrangian submanifold of $V$. Let $E>0$ and let $K$ be a compact set of $V$. Then there exists an $R=R(V,L,E,K)>0$ such that the following holds.
\begin{itemize}
\item For any $J$-holomorphic map
\[u:A_1\to V\]
satisfying $Area(u;A_1)\leq E$ and  $u(A_{1/2})\cap K\neq \emptyset$, we have \[u(A_{1/2})\subset B_R(K).\]
\item For any $J$-holomorphic map 
\[u:(E_1,\partial E_1)\to (V,L)\]
satisfying 
$Area(u;E_1)\leq E$ and $u(E_{1/2})\cap K\neq \emptyset$, we have
\[u(E_{1/2})\subset B_R(K).\]
%where $R(E,K,L,J,\mathcal{K})$. 
\end{itemize}
\end{proposition}
We also  have the following "family" version of the interior estimate:
\begin{proposition}\label{interiorestimate2}
Let $(J_s,\omega,g_s), s\in \overline{A_1}$ be a family of uniformly tame compatible triples on $V$. Let $E>0$ and let $K$ be a compact set of $V$. Then there exists a compact subset $R(K,E,J_s,\omega,g_s)$ such that the following holds.
\begin{itemize}
\item[]  If $u:A_1\to V$ satisfies $(du)^{0,1}_{J_s}=0$, $Area(u;A_1)\leq E$ and $u(A_{1/2}^2)\cap K\neq \emptyset$, then 
\[u(A_{1/2})\subset R(K,E,J_s,\omega,g_s).\]
\end{itemize}
Furthermore, suppose $L$ is a Lagrangian submanifold of $V$ which is uniformly tame with respect to $(J_s,\omega,g_s)$. Then there exists a compact subset $R'(K,E,J_s,\omega,g_s,L)$ of $V$ such that the following holds.
\begin{itemize}
\item[]  If $u:(E_1,\partial E_1)\to (V,L)$ satisfies
$(du)^{0,1}_{J_s}=0$, $Area(u;E_1)\leq E$, and $u(E_{1/2})\cap K\neq \emptyset$, then 
\[u(E_{1/2})\subset R'(K,E,J_s,\omega,g_s,L).\]
\end{itemize}
\end{proposition}
\begin{proof}
Reduce to Proposition \ref{interiorestimate1} by taking $(\overline{A_1}\times V, j_{\overline{A_1}}\oplus J_s,\omega_{\overline{A_1}}\oplus \omega)$. 
\end{proof}
%\begin{itemize}
%	\item For any compact set $K\subset X$ and $E\in \mathbb{R}_+$, there exists an $R>0$ such that for any connected compact Riemann surface $(S,\partial S)$ with boundary, and any $J$-holomorphic map
%		\[u:(S,\partial S)\to (X,K)\]
%		satisfying $E(S;u)\leq E$, $u(S)\subset B_R(K)$. 
%	Now we show the main result of this subsection which gives $C^0$ control over the images in the \textit{thick} part of the domain:

%In \cite{groman2021floer}, Groman introduces the class of $i-$bounded almost complex structures and uniformly $i$-bounded familes of almost complex structures for open symplectic manifolds. In particular, geometrically bounded almost Kahler manifolds are $i$-bounded. However, transversality requires us to introduce domain-dependent almost complex structures which are $s$-independent in the strip-like ends. This requires $i$-boundedness in the Gromov metric on say, $S\times X$. Then Groman shows in \cite{groman2021floer}:
%\begin{theorem}[Groman]
%	$i$-bounded almost complex structures have domain-local confinement property. 
%\end{theorem}

\subsection{Floer operations}\label{Floer Moduli SpaceS}

We now utilize the estimates in Section \ref{Monotonicity Techniques}. In Section \ref{Compactness and Transversality}, we discuss the compactification and transversality of moduli space of pseudo-holomorphic strips. The key idea is to use geometric boundedness to $C^0$-bound their images. We use this to define the Floer chain complex. In Section \ref{continuationstripsss}, we define the notion of passive continuation strips. In Section \ref{energyformulasss}, we review the standard formulas for the geometric energy of the continuation strips. In Section \ref{confinementcontstripssss}, we show the $C^0$ confinement of the passive continuation strips.  In Section \ref{continuationmapsss}, we construct continuation chain maps and discuss their properties. Finally, in Section \ref{pathgroupoidrepsss}, we discuss the path groupoid representation of the Floer cohomology local system on the base $M$. From now on, we assume that all the Lagrangians are exact, with respect to $\canliouvile$. 

\subsubsection{Compactness and  transversality}\label{Compactness and Transversality}
\paragraph{Moduli Spaces}\label{Modulispaces}
We first start with the compactness and transversality properties of the Floer moduli spaces. We follow \cite[Section 3.2]{GPSCV} closely. Let $L_1$ and $L_2$ be a pair of transversely intersecting Lagrangians in $T^{\ast}M$ such that $L_1$ is finite at infinity, and $L_2$ is horizontally finite (see Definition \ref{vertifin}). Since $L_i$ is $\canliouvile$-exact, there are smooth functions $f_i:L_i\to \mathbb{R}$ on $L_i$ such that $df_i=\canliouvile\vert L_i$. Such functions $f_i$ are unique up to constants. We choose the primitives $f_1$ and $f_2$ once and for all for $L_1$ and $L_2$ respectively. We define the action of an intersection point $x\in L_1\pitchfork L_2$ by:
\begin{align}\label{actionofintersectionpoint}
a(x):=f_1(x)-f_2(x).
\end{align}
Given such a pair $(L_1,L_2)$, a \textit{Floer datum} is an admissible family $J=J(t)=J_{L_1,L_2}(t)$. For our purpose, it suffices to only consider the case where $J_{L_1,L_2}$ is given by a \textbf{compact} deformation of $J_{con}$, and we will assume so from now on. 
\begin{definition}\label{stripmodulidefinition}
For $x,y\in L_1\pitchfork L_2$, let $\mathcal{R}(L_0,L_1,J_{L_1,L_2})_{x\mapsto y}$ be the moduli space of unparametrized $J_{L_1,L_2}(t)$-holomorphic strips $u$ between $L_0$ and $L_1$ with $\lim_{s\to -\infty} u(s,t)=x$ and $\lim_{s\to +\infty} u(s,t)=y$. The space $\mathcal{R}(L_0,L_1,J_{L_1,L_2}(t))$ is the union of unparametrized moduli spaces $\mathcal{R}(L_0,L_1,J_{L_1,L_2}(t))_{x\mapsto y}$ for $x,y\in L_0\pitchfork L_1$. The space $\overline{\mathcal{R}}(L_0,L_1,J_{L_1,L_2}(t))$ is the union of $\overline{\mathcal{R}}(L_0,L_1,J_{L_1,L_2}(t))_{x\mapsto y}$ consisting of broken $J_{L_1,L_2}$-holomorphic strips. 
\end{definition}
More generally, we will consider a uniformly admissible family $J(s,t)$ of almost complex structures on the strip such that $J(s,t)$ becomes $s$-invariant outside $[-N,N]\times [0,1]$. We will call the set $[-N,N]\times [0,1]$ the \textit{thick part}, and its complement the \textit{thin part}. The moduli of such pseudo-holomorphic strips will be denoted $\mathcal{R}$. This moduli admits a partial Gromov-Floer compactification by adding broken strips. The resulting partially compactified moduli will be written $\overline{\mathcal{R}}$. 

\paragraph{Compactness}\label{CompactnessofFloermoduli}
We want to show that the moduli space $\bar{\mathcal{R}}$ is compact. First, we need the following lemma from Abouzaid-Seidel \cite[Lemma 7.2]{SeidelAbouzaid} and  Ganatra-Pardon-Shende \cite[Lemma 2.46]{GPSCV}. For related ideas, see \cite[Section 3.2]{AurouxAbouzaidHMS}.
\begin{lemma}[Vertical Confinement/Maximum Principle]\label{verticalconfinement} Let $(S,j)$ be a Riemann surface with boundary. Let $J$ be an $\omega$-compatible almost complex structure of general contact type (Definition \ref{generalcontact}) on $\{r>a\}$ for some $a>0$. Let $u:S\to T^{\ast}M$ be a $(j,J)$-holomorphic curve such that $u^{\ast}\canliouvile\vert_{\partial S}\leq 0$ on $u^{-1}(\{r>a\})$, then $u$ is locally constant over $u^{-1}(\{r>a\})$. By $u^{\ast}\canliouvile\vert_{\partial S}\leq 0$, we mean that the evaluation of $u^{\ast}\canliouvile$ on a positively oriented vector field along $T\partial S$ is negative. Here the orientation is given with respect to $j$.
\end{lemma}

In particular, the condition $u^{\ast}\canliouvile\leq 0$ in Lemma \ref{verticalconfinement} holds when connected components of $\partial S$ belong to Lagrangians that are cylindrical over $\{r>a\}$ since on the cylindrical part $u^{\ast}\canliouvile=0$. We now show that the moduli space is compact.
\begin{proposition}
The moduli space $\overline{\mathcal{R}}$ is compact.\label{compactnomoving}
\end{proposition}
\begin{proof}
In the case $L_1$ is vertically finite, let $K_{L_1}$ be as in Definition \ref{vertifin}. In the case $L_1$ is horizontally finite, let $K_{L_1}$ be a compact subset of $M$ such that $\pi(L_1)\subset K_{L_1}$. Let $K_{L_2}$ be a compact subset of $M$ such that $\pi(L_2)\subset K_{L_2}$. Let $R>0$ be such that: (i) the Lagrangians $L_i$s are either cylindrical or empty outside $D_R^{\ast}M$, and (ii) the almost complex structure $J(s,t)$ is cylindrical outside $D_R^{\ast}M$. We furthermore demand that the Legendrian submanifolds $\Lambda_i=L_i\cap S_R^{\ast}M$ are either compact or empty, and that they are disjoint. Let $K_0$ be a compact codimension $0$ submanifold with boundary of $M$ such that on $T^{\ast}K_0^c$, $J(s,t)=J_{con}$. We assume that $K_0$ is large enough so that it contains all the compact subsets $K_{L_i}$ of $M$. Here the radius of the codisc bundle is taken with respect to the metric $g$ on the base $M$.  

Since the Lagrangians $L_{1},L_{2}$ are all exact and the almost complex structures in the family are $\omega$-compatible, we have 
\[\frac{1}{2}\norm{du}^2_{L^2,J}=\int u^{\ast}\omega=a(y)-a(x),\]
and so the geometric energy admits a uniform finite upper bound since there are only a finite number of intersections.  Proposition \ref{compactnomoving} will be true by Gromov compactness if we can find a fixed compact subset $K$ of $M$ and $R_3>0$ such that if $u\in \overline{\mathcal{R}}$ then the image of $u$ lies in $D_{R_3}^{\ast}K$. To do this, we show that outside of some compact subset the Lagrangians are uniformly separated near infinity, and argue via monotonicity to control the images of the thick parts and the thin parts.  

We first need to show that the Lagrangians in question are uniformly separated at infinity outside $D_R^{\ast}K_0$, that is, we have a lower bound $C>0$ on the distance with respect to metric induced from $J(s,t)$ between the Lagrangians $L_i$s outside of $D_{R}^{\ast}K_0$. 
When $L_1$ is vertically finite, such a lower bound between $L_1$ and $L_2$ is obvious. Now, by the proof of \cite[Proposition 3.19]{GPSCV}, we see that  horizontally finite Lagrangians are also uniformly separated outside of $D_{R} (T^{\ast}K_0)$. Modifying $C$ in the case $L_1$ is vertically finite, we get our uniform lower bound $C$. 

Having separated the Lagrangians at infinity, we deal with the thick part. Lemma \ref{geoboundfamily} tells us that 
$(T^{\ast}M,L_i,J)$ is uniformly geometrically bounded over some disc $A_{l}$ (or half-disc $E_{l}$)  in the thick part of $S$, and the geometric boundedness constants only depend on $l$, the family $J$, and the Lagrangians $L_i$s. Lemma \ref{interiorestimate2} then tell us that if the image of $u$ restricted to $A_{l}$ (or  $E_{l}$) intersects a large enough compact subset $A\subset T^{\ast}M$ that separates the Lagrangians near infinity in $A_{l/2}$ (or $E_{l/2}$) then the image of $u$ restricted to $A_{l/2}$ (or $E_{l/2}$) is contained in some $R(A,E,J_s,\omega,g_s,l)$ (or $R'(A,E,J_s,\omega,g_s,L_s,l)$). 
% (or equivalently, find a priori control over the metric space \textit{diameter} of image $u$). 
%Since the family $\overline{S}_{k,1}\to \mathcal{J}(T^{\ast}M)$ is uniformly admissible, and $\mathcal{J}(T^{\ast}M)$ consists of almost complex structures which are of contact type at infinity, by Lemma \ref{verticalconfinement} we have a priori \textit{vertical} control over the images of $J$-holomorphic curves. So all stable $J$-holomorphic curves are contained in $D_{R} T^{\ast}M$. 

The thick part ($[-N,N]\times [0,1]$) can be covered by uniformly finite number of half-discs and discs of some uniform radius $l_1>0$ such that the shrinked discs of radius $l_1/2$ also cover the thick part. The boundary conditions corresponding to horizontally finite Lagrangians lie in a compact set, since their intersections with $D_{R}^{\ast}M$ are compact. So by repeatedly applying Lemma \ref{interiorestimate2}, we can use monotonicity to show that over the thick part, the $J$-holomorphic curves are a priori contained in $D_{R_1}^{\ast} K'$ for some compact subset $K'\subset M$ and $R_1>R>0$.

%Suppose $L_n$ is vertically finite. The argument does not change if $L_n$ is horizontally finite. We apply the monotonicity argument  for the thick part. This implies that we can a priori restrict the diameters on the thick parts of $u$. %Since $L_n$ is vertically finite, finding the Fraudenfelder metric (see \cite[Lemma 4.34]{McduffSalamon}) such that $L_n$ is totally geodesic everywhere only requires modification over a \textit{compact} subset since $L_n$ is totally geodesic at (horizontal) infinity. This can be done exactly as in \cite[Lemma 4.34]{McduffSalamon}. Note that flatness at infinity and the vertical finiteness give uniform isoperimetricity constants near the horizontal infinity. 
%We may assume that $K'$ is contained in some $D_{R_1}(T^{\ast}M)$ for $R_1>R$. 

So now it remains to show that we can compactly enlarge $K'$ and $R_1$ so that the whole image of $u$ is contained in the compact enlargement. We follow the strategy in the proof of \cite[Proposition 3.19]{GPSCV}.
On the thin-parts, we have $s$-invariant almost complex structures $J(t)$. We take some constants $R_2>R_1>0$ so that the image of $u$ restricted to the thick part is contained in $D_{R_2}^{\ast}M$ and outside of $D_{R_2}^{\ast}M$, $J(t)$ is cylindrical. Let $K_{thin}$ be a compact subset of $M$ such that outside of $T^{\ast}K_{thin}$, $J(t)=J_{con}$. Then we take a codimension $0$ submanifold-with-boundary $K_{base}$ of $M$ containing $K_{thin}$, $K'$ and $K_0$, such that the $g$-distance $d_{base}$ between $\partial K_{base}$ and $K_{L_i}$ is positive.

%From monotonicity, we have seen that there exists a compact subset $K'$ of $T^{\ast}M$ such that the images of the thick parts of $u$ are contained in $K'$. We can enlarge this $K'$ so that outside of $K'$, all the Lagrangians $L_{i_0},...,L_{i_k}$ are uniformly separated and outside of $\pi(K')$, $J=J_{con}$. Indeed, the horizontally finite Lagrangians are uniformly separated near infinity, and we can take $\pi(K')$ to include the horizontal supports of the $L_i$s, for $i\neq n$ and the $s$-invaraint almost complex structures $J(t)$ on the thin parts. Then 

Suppose now that the interval $[a,b]\times [0,1]$ in the thin part is mapped outside of the disc bundle $D_{R_2}^{\ast}K_{base}$. Then we have 
\begin{align}\label{strip-endcontrol}
E\geq \int_{[a,b]}\left(\int_0^1 \abs{\partial_t u}\right)^2 \geq C (b-a),
\end{align}
Therefore, taking $L=E/C$, we see that if $(a-b)>L$ then $u$ cannot map $[a,b]\times [0,1]$ outside of $D_{R_2}^{\ast}K_{base}$. So there exists some $\epsilon>0$ such that $[a,b]\times [0,1]$ is covered by uniformly finite half-discs and discs of radius $\epsilon$. Hence applying the interior estimate (Lemma \ref{interiorestimate2}) again, we can enlarge $D_{R_2}^{\ast}K_{base}$ to $D_{R_3}^{\ast}K$ so that the image of $u$ is wholly contained in $D_{R_3}^{\ast}K$. This finishes the proof.%Now suppose we have $X\neq 0$. However, we know that outside of some compact subset $K'$, the family of almost complex structures are all equal to the background a.c. Thus we reduce to the $X=0$ case and proceed on by exactly the same argument.  So we are done. 
\end{proof}
\paragraph{Transversality}
We now proceed on with the construction of the Floer chain complex. For details, see \cite[Part 2, Sections 9-11]{SeidelZurich}. We now assume that the Lagrangians are graded, and that they are spin. Given each pair $L_1, L_2$, choose an initial $s$-invariant uniformly admissible family of almost complex structure $J^{in}_{L_1,L_2}(t)$\footnote{For our purpose, it suffices to set $J^{in}_{L_1,L_2}(t)=J_{con}$}. By Proposition \ref{compactnomoving}, we can find some $H>0$ and a compact subset $K\subset M$ such that: (i)  $D_H^{\ast}K$ contains all the intersection points of $L_1$ and $L_2$, (ii) $L_i$ is horizontally finite then $K$ contains the horizontal support of $L_i$, (iii) $J^{in}_{L_1,L_2}$ is cylindrical outside $D_H^{\ast}M$,  and if $L_i$ is vertically finite, then $K$ contains the compact subsets $K_{L_i}$ in the sense of Definition \ref{vertifin}. and (iv) $J^{in}=J_{con}$ outside $T^{\ast}K$ and the image of $u\in \mathcal{R}(L_1,L_2,J_{L_1,L_2}^{in})$ are contained in the \textit{interior} of $D_H^{\ast}K$.  

Consider the following space of $\omega$-compatible almost complex structures 
\begin{align}\label{compactdeformationfamily}
\mathcal{J}(K,H):=\{J: J=J^{in}_{L_1,L_2}(t) \text{ outside }
{D_{H}^{\ast}(K)}.\} \end{align}
This space $\mathcal{J}(K,H)$, equipped with the $C^{k}$ topology (which is equivalent to the uniform $C^k$-norm topology induced from $g_{con}$) is a Banach manifold modelled on the space of $C^{k-1}$-infinitesimal deformations $\mathcal{Y}(U)$ that satisfies the conditions
\begin{align}
&YJ+JY=0 && \omega(Y\cdot,\cdot)+\omega(\cdot,Y\cdot)=0 && \supp(Y)\subset \overline{D_H^{\ast}K}.
\end{align} 
Note then under $J\to J\exp(-JY)$, the class of horizontally finite almost complex structures stays invariant. Similarly, we can equip the space $\mathcal{J}(K,H)$ with the $C^{\infty}$-topology which makes it a Fr\'echet manifold. Furthermore, applying the proof of Proposition \ref{compactnomoving}, we see that there exists a compact set $P$ containing every $u\in \mathcal{R}(L_1,L_2,J)$ for $J\in \mathcal{J}(K,H)$. Indeed, we can run the argument of Proposition \ref{compactnomoving} outside of ${D_H^{\ast}K}$ where $J=J_{con}$ where the uniform separation of the Lagrangians and tameness constants coincide for $J\in \mathcal{J}(K,H)$. 
\footnote{Given a general symplectic manifold $M$, the space $\mathcal{J}(M)$ of (tame, or compatible) almost complex structures is given the \textit{weak} $C^{\infty}$-topology. When the base $M$ is \textit{compact}, the space $\mathcal{J}^k(M)$ and the space $\mathcal{J}^{\infty}(M)$ of $C^k$ and smooth compatible almost complex structures are Banach and Fr\'echet manifold. However, when $M$ is not compact, endowing Banach/Fr\'echet structures on such spaces become much more involved, unless one specifies appropriate decay condition at infinity for maps in $W^{k,p}_{loc}$.}

%Then proposition \ref{compactnomoving} implies that there exists a compact subset $K$ of $M$, and some $R>0$ and $H>0$ such that $J^{in}_{L_1,L_2}(t)$ is cylindrical on $\{r\geq \frac{H}{2}\}$ and $U=D_{H}(T^{\ast}B_R(K))$ contains $\mathcal{R}(L_0,L_1,J_{L_1,L_2})_{x\mapsto y}$ for all $x,y\in L_1\pitchfork L_2$--- note that there are only finitely many intersection points. Then we 
%\begin{align}\label{compactdeformationfamily}
%\mathcal{J}(K,R,H):=\{J: J \text{ is cylindrical outside } \{r\geq H\} \text{ and } J=J_{con} \text{ on }
%{T^{\ast}(M-B_R(K))}.\} \end{align}
%From Remark \ref{controllingtheconstants}, we see that sufficiently small neighbourhood of $J_{L_0,L_1}^{in}$ inside $\mathcal{J}(K,R,H)$ in the $C^{\infty}$-topology allow us to control the compact set that the disks and their Gromov degenerations are confined in. 
%, monotonicity implies that we can confine the $C^0$ images of pseudo-holomorphic discs in the \textit{same} compact set. 

Since we are now in the situation covered in \cite[Section 8(i)]{SeidelZurich}, we can perturb the family $J_{L_0,L_1}(t)$ generically so that moduli spaces of holomorphic strips $\mathcal{R}(L_0,L_1,J_{L_1,L_2})_{x\mapsto y}$ are transversely cut out for all $x,y\in L_0\pitchfork L_1$. Indeed, note that given a $J^{in}_{L_1,L_2}$-holomorphic curve $u$, the set of injective points is dense, and the images of such points must necessarily lie in the interior of $D_H^{\ast}K$. 

To ensure smoothness of the $J$-holomorphic strips via elliptic regularity, we need to find a Baire dense subset in the $C^{\infty}$-topology whose associated moduli spaces of strips are transversely cut-out. This is done either using the Floer $C^{\infty}_{\epsilon}$-space or Taubes' trick. For details, see \cite[Chapter 3]{McduffSalamon}. Note that since there are only finitely many intersections, and since finite intersections of Baire dense subsets are Baire dense, we can find a generic $J(t)$ such that all the moduli spaces $\mathcal{R}(L_0,L_1,J_{L_1,L_2})_{x\mapsto y},x,y\in L_0\pitchfork L_1$ are transversely cut out.

%Then we compactify, obtaining the compactified moduli space $\overline{\mathcal{R}}(L_0,L_1,J_{L_1,L_2})_{x\mapsto y}$. 

\begin{definition}
A uniformly admissible family $J(t)$ of almost complex structures such that the moduli space of holomorphic strips $\mathcal{R}(L_0,L_1,J(t))$ is transversely cut out is called a \emph{regular Floer-datum} for the pair $L_0,L_1$. 
\end{definition}

Choose a regular Floer datum for each pair once and for all. We can regard each intersection point $x\in L_0\pitchfork L_1$ as a \textit{constant} holomorphic half-strip. The boundary conditions are given as follows. If $A$ is the induced Maslov grading on $LGr(T_x (T^{\ast}M))$ and $A_0$,$A_1$ are chosen grading functions on $L_0$ and $L_1$ respectively, we choose a path $L_s$ of Lagrangian subspaces of $T_x (T^{\ast}M)$ that begin at $T_x L_0$ and end at $T_x L_1$ such that $A(L_0)=A_0(x)$ and $A(L_1)=A_1(x)$.  We then choose a path of spin structures $P_s$ over $L_s$ and isomorphisms of $Spin$ torsors $(P_0)_x\simeq P_{L_0}(x)$ and $(P_1)_x\simeq P_{L_1}(x)$. Such choices define an orientation on the determinant line of the linearized Cauchy-Riemann operator associated to $x$. Any other choice of paths that satisfy the grading constraints give a canonically isomorphic real line, but different homotopy classes of path of spin structures may give different orientations. 

Writing $\mathfrak{o}_x$ for the corresponding abstract real line, we can regard the orientation as a choice of an element in  $(\mathfrak{o}_x-\{0\})/\mathbb{R}^{\ast}$. Then we write $\abs{\mathfrak{o}}_x=\mathbb{Z}((\mathfrak{o}_x-\{0\})/\mathbb{R}^{\ast})$ with $+1$ identified with the orientation of $\mathfrak{o}_x$. We then form the $\mathbb{Z}$-group as a direct sum:
\begin{align}\label{floercomplex}
CF(L_0,L_1;\mathbb{Z})=\oplus_{x\in L_0\pitchfork L_1} \abs{\mathfrak{o}}_x
\end{align}
Since this is standard in literature, we conclude that we can define a chain complex structure on $CF(L_0,L_1)$ by counting regular $J$-holomorphic strips. As usual, we call the cohomology $HF(L_0,L_1)$ of this chain complex the \textit{Floer cohomology}. 
For details, see \cite[Section 11]{SeidelZurich}.

\subsubsection{Continuation strips} \label{continuationstripsss}
We now pose the continuation strip moduli problem. Let $L_s$ be an exact Lagrangian isotopy of horizontally finite Lagrangians. We say that $L_s$ is \textit{uniformly horizontally supported} if there exists a compact subset $K$ of $M$ such that $\pi(L_s)\subset K$. Given such an exact Lagrangian isotopy, we can find some time-dependent horizontally finite Hamiltonian $H_s$ such that $L_s=\psi_s(L)$ with uniform horizontal support (Definition \ref{horizontalfinitenessdefini}). Here $\psi_s$ is the Hamiltonian flow associated to $H_s$. Let $V$ be a vertically finite Lagrangian submanifold. We say that an exact Lagrangian isotopy $V_s:V\times [0,1]\to T^{\ast}M$ of vertically finite Lagrangians is \textit{compactly supported} if there exists a compact subset $K$ of $V$ such that $V_s(v,s)=v$ for $v$ outside of $K$. Note that given a uniformly horizontally finite isotopy $\psi_s$, the isotopy $V_s=\psi_s^{-1}(V)$ is compactly supported.% vertically finite and the Hamiltonian deformation takes place inside a \textbf{compact} subset of $V$. 

Let $L_s$ be a uniformly horizontally supported isotopy of horizontally finite Lagrangians. Suppose $V$ is transverse to $L_0$ and $L_1$. As we explained in Section \ref{Compactness and Transversality}, we can compactly perturb the constant family $J_{con}=J_{con}(t)$ to find a regular Floer datum $J_0(t),J_1(t)$ for the pair $(V,L_0)$ and $(V,L_1)$, respectively, such that the moduli spaces $\mathcal{R}(V,L_0,J_0(t))$ and $\mathcal{R}(V,L_1,J_1(t))$ are transversely cut out. 

Choose a uniformly horizontally supported Hamiltonian isotopy $\psi_s$ generating $L_s=\psi_s(L_0)$. Given a pair $((L_0,V,J_0),(L_1,V,J_1))$, fix a uniformly admissible family $\tilde{J}(s,t)$ of almost complex structures on $\mathcal{Z}$  
such that $\tilde{J}(s,t)=J_0(t)$ for $s\leq -N$, $\tilde{J}(s,t)=J_1$ for $s\geq N$, and $\tilde{J}$ is given by a compact perturbation of the constant family $J_{con}$ on $[-N,N]$. \footnote{By this, we mean that $J(s,t)=J^{in}(s,t)$ outside some compact subset of $T^{\ast}M$.} Choose a smooth increasing elongation function $l:[0,\infty)\to [0,1]$ such that $l(s)=0$ for $s<-N$ and $l(s)=1$ for $s>N$.

We will introduce a slightly non-standard model for continuation strips. As illustrated in \cref{continuationstripfig}, we will move $V$ instead of $L$ by the inverse of $\psi_s$. The reason is that the movie of the family $V_{l(s)}=\psi_{l(s)}^{-1}(V)$, which is normally only totally real, will be Lagrangian outside a compact subset, and simplify the confinement problem. For instance, see \cref{semiadmissiblegeoboundfamily}. To make this idea precise, let 
\[J(s,t)=(\psi_{l(s)}^{\ast})\tilde{J}.\]
Note that $J$ satisfies $J(s,t)=(D\psi_1^{-1})_{\ast} J_1(t)(D\psi_1)_\ast=(\psi_1)^{\ast}J_1$ for $s\geq N$. 
\begin{figure}[t]
\centering
\includegraphics[width=0.6\textwidth, height=5cm]{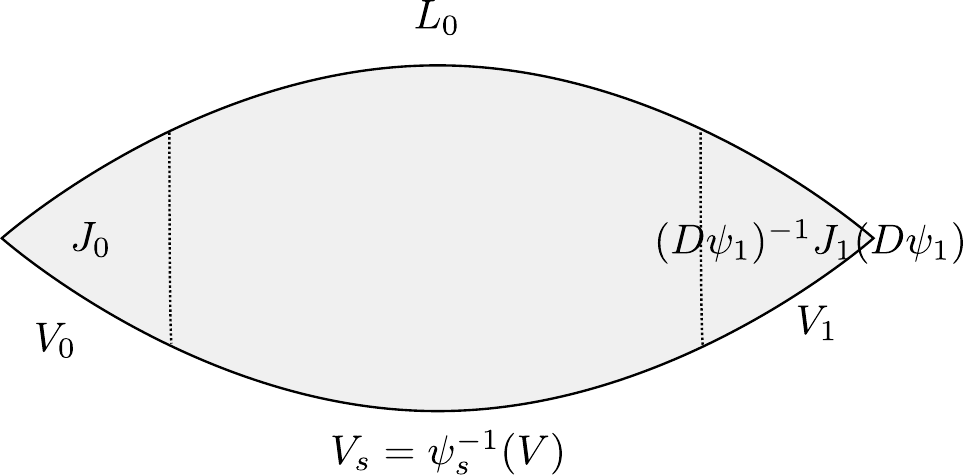}
\caption{Boundary Conditions For Passive Continuation Strips}
\label{continuationstripfig}
\end{figure}
We say that a map $u:\mathcal{Z}\to T^{\ast}M$ is a $J$-holomorphic strip with a \textit{passive moving Lagrangian boundary condition} if the following equation is satisfied.
\begin{align}\label{passivemovingboundaryeq}
\begin{cases}
\bar{\partial}_J u=0 &  \\ 
u(s,0)\subset V_{l(s)}&\\
u(s,1)\subset L_0&\\
\lim_{s\to \infty} u(s,t)\in L_0\cap V\\
\lim_{s\to -\infty} u(s,t)\in L_0\cap \psi_1^{-1}(V).
\end{cases}
\end{align}
We call the solutions \textit{passive continuation strips}. See Figure \ref{continuationstripfig}.

To convert \eqref{passivemovingboundaryeq} into something more familiar, we can simply take $\tilde{u}=\psi_{l(s)}\circ u$ and arrive at an equivalent family of equations
\begin{align}\label{perturbedpassive}
\begin{cases}
\frac{\partial \tilde{u}}{\partial s}+\tilde{J}(s,t)\frac{\partial \tilde{u}}{\partial t}-l'(s)X_{l(s)}(\tilde{u})=0 &  \\ 
\tilde{u}(s,0)\subset V&\\
\tilde{u}(s,1)\subset L_{l(s)}&\\
\lim_{s\to \infty} \tilde{u}(s,t)\in L_0\cap V\\
\lim_{s\to -\infty} {\tilde{u}}(s,t)\in L_1\cap V.
\end{cases}
\end{align}
The point is that a priori $C^0$-diameter estimate (Proposition \ref{compactyesmoving}) for the solutions of \eqref{passivemovingboundaryeq} is much simpler compared to the equivalent equation \eqref{perturbedpassive}. 

We now introduce the classes of \textit{homotopies} of families that we will use to show certain invariance properties of $HF$. Suppose we are given two uniformly admissible families of almost complex structures $\tilde{J}^0$ and $\tilde{J}^1$ on $\mathcal{Z}$ such that $\tilde{J}^i(s,t)=J_0$ for $s\leq -N$ and $\tilde{J}^i(s,t)=J_1$ for $s\geq N$ for some $N>0$. Then we say that a path $\tilde{J}^{\tau}$ of family of almost complex structures between $\tilde{J}^0$ and $\tilde{J}^1$ over $\mathcal{Z}$ is a \textit{uniformly admissible homotopy} if (i) there exists some $N'>0$ such that $\tilde{J}^{\tau}(s,t)=J_0$ for $s\leq -N'$ and $\tilde{J}^{\tau}(s,t)=J_1$ for $s\geq N'$, and (ii) there exists a compact set $K\subset M$ and some $R>0$ such that outside $T^{\ast}K$, $\tilde{J}^{\tau}=J_{con}$ and $\tilde{J}^{\tau}(s,t)$ is cylindrical for all $s,t\in \mathcal{Z}$ and $\tau\in [0,1]$ outside $D_R^{\ast}M$. The Hamiltonian counterpart is as follows; suppose we are given a family of time-dependent Hamiltonians $H^{\tau}_s$ and suppose there exists an $R>0$ and a compact subset $K\subset M$ such that $H_s^{\tau}$ is cylindrical outside $D_R^{\ast}M$ and $\pi(\supp H_s^{\tau})\subset K$ for all $s$ and $t$. Then we say that such a family is a \textit{uniformly cylindrical and horizontal}. 

We now state the result, whose proof we postpone to Section  \ref{continuationmapsss}.
\begin{proposition}\label{propertiesofactivepassivecontinatuinmaps}
Let $V$ be a vertically finite Lagrangian and let $L$ be a horizontally finite Lagrangian. For uniformly horizontally supported isotopies $L_s$ such that $L_s\pitchfork V,s=0,1$ there exists a chain map called \emph{passive continuation map}
\begin{align}\label{passivecontinuationmap}
c^{passive}&=c_{(L_0,J_0)\to(L_1,J_1)}:CF(V,L_0,J_0)\to CF(V,L_1,J_1).
\end{align}
The passive continuation map has the following properties.
\begin{itemize}
\item A uniformly horizontally finite generic homotopy $(L_{s}^{\tau},\tilde{J}^{\tau})$ relative to end points generated by a uniformly cylindrical and horizontally supported family of Hamiltonians $H_s^{\tau}$ induces a chain homotopy map 
\[H:CF^{\ast}(V,L_0)\to CF^{\ast+1}(V,L_1)\]
for the passive continuation maps. 
\item There is a chain homotopy between $c_{L_0\to L_1}\circ c_{L_1\to L_2}$ and the continuation map $c_{L_0\to L_2}$ associated to concatenation of isotopies. Hence the continuation maps are well-defined up to isomorphisms in cohomology. 
\item For constant maps $L_s=L$, the induced continuation map is the identity. 
\item For any uniformly horizontally finite isotopy, the passive continuation maps \eqref{passivecontinuationmap} are quasi-isomorphisms. 
\end{itemize}
\end{proposition}

Then an argument as in the end of Section \ref{Compactness and Transversality} tells us that for generic $J$, the moduli spaces of solutions of \eqref{passivemovingboundaryeq} are transversely cut out. This will show the existence of a chain map
\[\hat{c}:CF(V_0,L_0,J_0)\to CF(V_1,L_0,(D{\psi_1}^{-1})_{\ast}J_1 (D\psi_1)_{\ast}).\]

Making use of the trivial isomorphism 
\[CF(V_1,L_0,(D\psi^{-1})_{\ast}J_1 (D\psi)_{\ast})\simeq CF(V,L_1,J_1)\]
induced by the \textbf{global} Hamiltonian isotopy $\psi^{-1}_1:T^{\ast}M\to T^{\ast}M$,
we get the passive continuation map
\[\tilde{c}:CF(V,L_0,J_0)\to CF(V,L_1,J_1)\]
defined by the following commutative diagram. 
\[\begin{tikzcd}
{CF(V,L_0,J_0)} & {CF(V,L_0,J_0)} \\
{CF(V,L_1,J_1)} & {CF((\psi_1)^{-1}V,L_0,(D\psi_1^{-1})_{\ast}J_1(D\psi_1)_{\ast})}.
\arrow["Id", from=1-1, to=1-2]
\arrow["{\hat{c}}", from=1-2, to=2-2]
\arrow["Id"', from=2-1, to=2-2]
\arrow["{\tilde{c}}"', from=1-1, to=2-1]
\end{tikzcd}\]

\subsubsection{Energy Formula}\label{energyformulasss}
We now state some standard energy formulas for the solutions of (\ref{passivemovingboundaryeq}). We first start with the following formula from Oh \cite[Proposition 3.4.8]{Ohsymplectictopology}.
\begin{lemma}\label{exactisotopyprimitive}
Let $(X,d\alpha)$ be an exact symplectic manifold and let $L\subset X$ be an exact Lagrangian. Let $\psi_s$ be a Hamiltonian isotopy on $X$ and let $F$ be the primitive of $\alpha$ on $L$. Let $L_s=\psi_s(L)$, then
\[F_s= F+\int_0^s (-H_t\circ i_t+\alpha(X_{H_t})\circ i_t)dt\]
satisfies $dF_s=i_s^{\ast}\alpha$. Here $i_s=\psi_s\circ i$ where $i:L\to X$ is the inclusion map. 
\end{lemma}
The following formula follows from Lemma \ref{exactisotopyprimitive} by direct computation. 
\begin{lemma}
Suppose we have $(X,d\alpha,L,\psi_s)$ as above. Suppose $\gamma:[0,1]\to X$ is a curve such that $\gamma(s)\in L_s$. Then we have
\[\int \gamma^{\ast}\alpha= F_1(\gamma(1))-F_0(\gamma(0))+\int_0^1 H_s(\gamma(s)) ds.\] 
\label{movingboundarycomputation}
\end{lemma}

Combining the two lemmas, we arrive at the following expression for the energy of discs with moving boundary conditions

\begin{lemma}
Suppose $S$ is a disc with $k+1$ marked points $x_0,...,x_{k+1}$. Identify each of the anticlockwise ordered boundary segments $\partial S_1,...,\partial S_{k+1},\partial S_0$ with $[0,1]$.  Suppose we have moving Lagrangian labels $L_0,...,L_{k+1}$, $L^s_i=\psi_s(L_i)$ as above, with $L_i^0=L_i, L_i^1=L_{i+1}$. Suppose the Lagrangians $L_j,j=0,...,k$s are mutually transverse. Let $u:S\to T^{\ast}M$ be a continuation disc with moving Lagrangian labels with respect to Hamiltonians $H_s:S\to C^{\infty}(T^{\ast} M,\mathbb{R})$. Choose the primitives of $L_i^s$ as in Proposition \ref{exactisotopyprimitive}. Then the geometric energy of the solutions of \eqref{passivemovingboundaryeq} satisfy:
\begin{align}\label{geometricenergymovingboundary}
\int_S \frac{1}{2}\abs{du}^2_J=\int_S u^{\ast}\omega=\sum_i a^{+}(x_i)-\sum_i a(x_i^{-})+ \int_{\partial S} H_s(u)ds.
\end{align} 
%Furthermore, the formula is independent of the choices of the primitives.
In particular, if the isotopies $H_s$ are compactly supported on $L_i$s, then the geometric energy is bounded by a constant which depends only on the original action of the intersection points and $H_s$. %Similarly, if the isotopies $H_s$ are all non-negative, then the energy is a priori bounded by a constant which depends only on the action of the intersection points and $H_s$.
\end{lemma}
\begin{proof}
There exists some $N\gg 1$ such that $l(s)$ is locally constant outside of $[-N,N]$. Fix such an $N$. Then we can regard $L_{l(s)}$ as a family on $[-N,N]$. Define the primitives of $\canliouvile$ of $L_{l(s)}$ with respect to this family on $[-N,N]$ using Lemma \ref{exactisotopyprimitive}. Then the proof follows from Lemmas \ref{exactisotopyprimitive} and \ref{movingboundarycomputation}. 
\end{proof}
\subsubsection{Confinement of continuation Strips}\label{confinementcontstripssss}
We now show the $C^0$ confinement of passive continuation strips and construct continuation chain homomorphisms. Firstly, given a moving Lagrangian boundary condition $L_s$, let $\mathcal{L}:=\{(s,p):s\in [0,1], p\in L_s\}$. Then $\mathcal{L}$ is a totally real submanifold of $\overline{A_1}\times T^{\ast}M$.

We have the following analogue of Lemma \ref{geoboundfamily}.
\begin{lemma}\label{semiadmissiblegeoboundfamily}
Let $J:\overline{A_1}\to \mathcal{J}(T^{\ast}M)$ be a uniformly admissible family of almost complex structures over $\overline{A_1}$, then $(\overline{A_1}\times T^{\ast}M, j_{\overline{A_1}}\oplus J,\omega_{\overline{A_1}}\oplus \omega_{T^{\ast}M})$ is geometrically bounded. Furthermore, if $L\subset T^{\ast}M$ is finite at infinity, then $\partial \overline{A_1}\times L$ is geometrically bounded. If the submanifold $W\subset (\overline{A_1}\times T^{\ast}M,j_{\overline{A_1}}\oplus J,\omega_{\overline{A_1}}\oplus \omega_{T^{\ast}M})$ is totally real, and coincides with some $\partial \overline{A_1}\times L$ outside a compact subset for $L$ a Lagrangian finite at infinity, then $W$ must be tame.
\end{lemma}
\begin{proof}
The proof is as before; the Lagrangian submanifold $\partial \overline{A_1}\times L$ is geometrically bounded by Lemma \ref{geoboundfamily}. So since $W$ agrees with $\partial \overline{A_1}\times L$ outside a compact subset, $W$ must be geometrically bounded as well. 
\end{proof}

From now on, we will not distinguish between horizontally finite Hamiltonian isotopies with uniform horizontal support, and exact Lagrangian isotopies with uniform horizontal support. From Lemma \ref{semiadmissiblegeoboundfamily}, we arrive at the following corollary:
\begin{corollary}\label{geometricallyboundedtotallyrealfiniteLag}
Suppose $L_s$ is an exact Lagrangian isotopy of horizontally finite Lagrangians with uniform horizontal support. Given an $R>0$, let
\[\mathcal{L}_{\abs{p}\leq R}:=\{(s,p):s \in [0,1], p\in L_s, \abs{p}\leq R \}.\]
Then $\mathcal{L}_{\abs{p}\leq R}$ is tame. 

Alternatively, suppose $K_s$ is an exact Lagrangian isotopy of vertically finite Lagrangians with uniform horizontal support. The totally real submanifold
\[\mathcal{K}=\{(s,p):s\in [0,1], p\in K_s\}\]
is tame.
\end{corollary}
\begin{proof}
The first case follows immediately since $\mathcal{L}_{\abs{p}\leq R}$ is compact. The second case satisfies the hypothesis of Lemma \ref{semiadmissiblegeoboundfamily} so we are done. 
\end{proof}

We have the following analogue of Proposition \ref{compactnomoving} for $J$-holomorphic curves with moving boundary conditions:
\begin{proposition}\label{compactyesmoving}
Let $L_s,s\in [0,1]$, be an exact Lagrangian isotopy of horizontally finite Lagrangians with uniform horizontal support and let $\psi_s$ be the horizontally finite Hamiltonian isotopy generating $L_s$. Let $V$ be a vertically finite Lagrangian. Then the following holds.
\begin{itemize}
\item[] There exists a compact set $K=K(J(s,t),L_s,V,l)\subset T^{\ast}M$ such that the solutions of (\ref{passivemovingboundaryeq}) are contained in $K$. 
\end{itemize}
\end{proposition}
\begin{proof}
We modify the proof of \cite[Proposition 3.23]{GPSCV}. The boundary conditions are fixed for $s>>0$ and $s<<0$, so the moving boundary conditions appear only on the compact part $S_N:=[-(N+1),(N+1)]\times [0,1]$ for some $N\gg0$. We can split the strip $\mathcal{Z}$ into the thin part $(-\infty,-N-1)\times [0,1]\cup (N+1,\infty)\times [0,1]$ and the thick part $S_N$.  

%We deal with the thin part first. First, note that both $L_i\cap\{r\leq R\},i=0,1$ are compact. We can enlarge this subset so that the Lagrangians $L_0,L_1$ and $V$ are uniformly separated. So as in the proof of \ref{compactnomoving}, we see that the thin parts must map into a compact subset of $T^{\ast}M$. In particular, we see that the arcs $ \{\pm N\}\times [0,1]$ map into a compact subset $K'$ of $T^{\ast}M$ which only depends on $L_0,L_1,V$ and $J$.

%But then we can regard the thick part $[-(N+1),N+1]\times [0,1]$ as the half-disk $E_1$ where the boundary component $\partial E_1$ is mapped to $[-N+1,N+1]\times \{0\}$ and the rest of the boundary is mapped to $\{\pm (N+1)\}\times [0,1]\cup [-(N+1),N+1]\times \{1\}$. Composing this with the map $u$, we get a map $v:(E_1,\partial E_1)\to (X,V\cup K')$ with the image coinciding with $u(S_{N})$. By the boundary estimate (Proposition $\ref{totaldomainestimate}$), we see that the diameters of such discs are also a priori bounded. Note that here we are utilising Corollary \ref{geometricallyboundedtotallyrealfiniteLag}. 
We control the thick part using tameness for totally real submanifolds. Consider the compatible triple $(S_N\times T^{\ast}M,j\oplus J,\omega_{\mathbb{C}}\oplus \omega_{T^{\ast}M})$. Note that the manifold 
\[\mathcal{V}:=\{(s,p):s\in[0,1],p\in V_s\}\] is totally real with respect to $j\oplus J$. Furthermore, it is $(\omega_{S_N}\oplus \omega_{T^{\ast}M})$-Lagrangian outside some compact subset of $(V_{s(l)})$. So by Lemma \ref{semiadmissiblegeoboundfamily} and Corollary \ref{geometricallyboundedtotallyrealfiniteLag}, the manifold $\mathcal{V}$ must be actually tame with respect to $(j\oplus J)$. Then from the a priori bound on the geometric energy and tameness, we see that the image of the thick part $S_N$ must be a priori confined by Proposition \ref{interiorestimate2}. The analysis for the thin-part is unchanged. This finishes the proof.   
\end{proof}

Note that the proof does not extend to the case where $L$ and $K$ are both vertically finite. 
\subsubsection{Continuation maps}\label{continuationmapsss}
Recall the set-up in \ref{continuationstripsss}. From the discussion in Section \ref{Compactness and Transversality}, we had chosen a generic compact perturbation $J_i(t),i=0,1$ of the constant family $J_{con}$ as regular Floer datum for pairs $(V,L_0)$ and $(V,L_1)$. Then we further chose an initial family of uniformly admissible almost complex structures $\tilde{J}^{in}$ on $(-\infty,\infty)\times [0,1]$ for (\ref{passivemovingboundaryeq}) such that $\tilde{J}^{in}(s,t)=J_0(t)$ for $s\ll 0$ and $\tilde{J}(s,t)=J_1(t)$ for $s\gg 0$. Then we set $J^{in}(s,t)=(\psi_{l(s)})^{\ast}\tilde{J}^{in}$.

We see from Proposition \ref{compactyesmoving} that for such a family $J^{in}$, the solutions of (\ref{passivemovingboundaryeq}) are compactly confined. Just as we did in the end of Section \ref{Compactness and Transversality}, (See expression (\ref{compactdeformationfamily})), we perturb the family $J^{in}$ to $J$ over this compact set such that the solutions of (\ref{passivemovingboundaryeq}) are transversely cut out. Then for this perturbed $J$, $\tilde{J}=(\psi_s^{-1})^{\ast}J$ is called a regular perturbation datum for $((V,L_0),(V,L_1),\psi_s,\tilde{J}^{in})$. Furthermore, the 1 dimensional part of this moduli space is compactified as usual. 

We briefly explain how to orient the moduli space of passive continuation strips. We adopt the notions from \cite{SeidelZurich}. Suppose we have Lagrangian branes $(V,A_V,P_V)$ and $(L_{s},A_{s},P_{L_s})$. For $x\in V\pitchfork L_0$ and $y\in V\pitchfork L_1$, choose a path of Lagrangian subspaces $L_T(x), T\in [0,1]$ from $T_x V$ to $T_x L_0$, and a path $L_T(y)$ from $T_y V$ to $T_y L_1$, that satisfy the grading constraint. Choose a path 
of spin structures $(P_T)_x$ over $L_T(x)$ and isomorphisms of $Spin$ torsors $(P_0)_x\simeq P_V(x)$ and $(P_1)_x\simeq P_{L_0}(x)$. Define the half-plane operators and their orientation lines as before. Then given a regular $u$ satisfying \eqref{passivemovingboundaryeq} we glue the constant half-strips $x$ and $\psi_1^{-1}(y)$ to the strip-like ends of $u$. Then given the glued disc $x\sharp u\sharp \psi_1^{-1}(y)$, the induced spin structure on the boundary is given by glueing $(P_T)_x$, $\psi_s^{\ast}P_V(u)$, $\psi_1^{\ast}(P_T)_y$ and $\psi_s^{\ast}P_{L_s}(u)$ along the boundary of $x\sharp (\psi_s\circ u)\sharp y$, and then pulling back by $\psi_s$. We then orient the orientation line of $u$ using the induced spin structure--- for details, see \cite[Chapter II, Section 11]{SeidelZurich}. 

Observe that we would have obtained the same answer if we used the solutions of $\eqref{perturbedpassive}$ instead. Indeed, observe that the linearization of \eqref{passivemovingboundaryeq} at $u$ is $D_1=\nabla_{s}(-)+J\nabla_t (-)+(\nabla J)(-)(\frac{\partial u}{\partial t})$ for some torsion-free connection $\nabla$, and the linearization of $\eqref{perturbedpassive}$ is $D_2=\nabla_{s}(-)+\tilde{J}\nabla_t (-)+(\nabla {\tilde J})(-)(\frac{\partial \tilde{u}}{\partial t})-\nabla_{(-)}(l'(s)X_{l(s)})$. For such $D_1$ and $D_2$, we have $D{\psi_s}_{\ast}\circ D_1=D_2 \circ D{\psi_s}_{\ast}$, so the resulting Fredholm theories are equivalent. This means we can orient the moduli space of passive continuation strips using \eqref{perturbedpassive}. One may say that Floer theoretically, the actual equation we look at is \eqref{perturbedpassive}, though in terms of confinement and controlling the geometry of the continuation strips, \eqref{passivemovingboundaryeq} is better suited. Either way, just as we suggested, by counting the $0$th dimensional parts, we get an induced chain map
\begin{align}\label{continuationmap}
c^{passive}&=c_{(L_0,J_0)\to(L_1,J_1)}:CF(V,L_0,J_0)\to CF(V,L_1,J_0)
\end{align}
which we call the passive \textit{continuation map}. 
Two passive continuation maps are concatenated as indicated by the commutative diagram \eqref{E: concatenationdiagram}. We are now ready to show Proposition  \ref{propertiesofactivepassivecontinatuinmaps}.

\begin{figure}[t]
\centering
\begin{tikzcd}
{CF(V,L_0,J_0)} & {CF(V,L_0,J_0)} \\
{CF(V,L_1,J_1)} & {CF(\psi_1^{-1}(V),L_0,{\psi_1}^{\ast}J_1)} \\
{CF(\psi_2^{-1}(V),L_1,(\psi_2^{-1})^{\ast}J_2)} & {CF((\psi_1^{-1}\circ \psi_2^{-1})(V),L_0,(\psi_2\circ \psi_1)^{\ast}J_2)} & {} \\
{CF(V,L_2,J_2)} & {CF(V,L_2,J_2)} & {}
\arrow["Id", from=3-1, to=4-1]
\arrow["Id"', tail reversed, from=2-1, to=2-2]
\arrow["{\hat{c_{12}}}", from=2-1, to=3-1]
\arrow["{\tilde{c_{01}}}", from=1-1, to=2-1]
\arrow["{\hat{c_{01}}}", from=1-2, to=2-2]
\arrow["Id", tail reversed, from=1-1, to=1-2]
\arrow["Id", tail reversed, from=4-1, to=4-2]
\arrow["Id", tail reversed, from=3-1, to=3-2]
\arrow["{\hat{c_{12}}'}", from=2-2, to=3-2]
\arrow["Id"', from=3-2, to=4-2]
\end{tikzcd}
\caption{Composition Diagram}
\label{E: concatenationdiagram}
\end{figure}	

\begin{proof}
We only sketch the proof. See \cite[Section (8k)]{SeidelZurich} and the construction in \cite[Section 3.6]{Zapolsky} for details. We first show the first assertion. Suppose we are given a homotopy of Lagrangian isotopies $L^{\tau}_s=\psi_s^{\tau}(L)$, that is fixed at the endpoints $s=0,s=1$ generated by a homotopy $\psi_s^{\tau}$ of uniformly cylindrical and horizontally supported Hamiltonian isotopies. Set $V^{\tau}_{s}=(\psi_s^{\tau})^{-1}(V)$.  

Recall that we had chosen a uniformly admissible family $\tilde{J}^{in}$ for the pair of triples $((V,L_0,J_0),(V,L_1,J_1))$ such that $\tilde{J}^{in}(s,t)=J_0$ for $s\ll 0$, $\tilde{J}^{in}(s,t)=J_1$ for $s\gg 0$. Suppose $\tilde{J}^0$ is a regular perturbation datum for $((V,L_0),(V,L_1),\psi^{0}_s,\tilde{J}^{in})$ and $\tilde{J}^1$ is a regular perturbation datum for $((V,L_0),(V,L_1),\psi^{1}_s,\tilde{J}^{in})$. Suppose furthermore there exists an initial uniformly admissible homotopy of $\omega$-compatible almost complex structures $\tilde{J}^{\tau},\tau\in [0,1]$ extending $\tilde{J}^0$ and $\tilde{J}^1$ such that each $\tilde{J}^{\tau}$ is given by compactly perturbing $\tilde{J}^{in}(s,t)$ for $s\in [-2,2]$. Set $J^{\tau}(s,t)=(\psi^{\tau}_{l(s)})^{\ast}\tilde{J}^{\tau}$. The corresponding family of passive continuation strip equations is given by:
\begin{align}\label{passivemovingboundaryeqfamily}
\begin{cases}
\bar{\partial}_{J^{\tau}} u=0 &  \\ 
u(s,0)\subset V^{\tau}_{l(s)}&\\
u(s,1)\subset L_0&\\
\lim_{s\to \infty} u(s,t)\in L_0\cap V\\
\lim_{s\to -\infty} u(s,t)\in L_0\cap {(\psi^{\tau}_1)}^{-1}(V).
\end{cases}
\end{align}
%for almost complex structure $J^{\tau}(s,t),\tau=0,1$ such that $J^{\tau}(s,t)=(\psi^{\tau})^{\ast}\tilde{J}^{in}(s,t),\tau=0,1$ outside a compact subset $K_1$. 
Observe that in general, ${(\psi_1^{\tau})}^{-1}(V)$ will depend on $\tau$ and so we are looking at a family of continuation strips with different boundary conditions. By the properness of the map $(\psi^{\tau}_s)^{-1}:[0,1]\times [0,1]\times  T^{\ast}M\to T^{\ast}M$, and application of the arguments in the proof of Proposition \ref{compactyesmoving}, we may enlarge $K$, and $R>0$ such that: (i) the horizontal support of $\psi^{\tau}_s$ is contained in $K$, (ii) $\psi^{\tau}_s$ is cylindrical outside $D_R^{\ast}$, (iii) $J^{\tau}=J_{con}$ outside $T^{\ast}K$, and (iv) solutions of \eqref{passivemovingboundaryeqfamily} are contained in $D_R^{\ast}K$ for $t\in[0,1]$. In particular, condition (i) implies that the set $T^{\ast}K$ is \textit{invariant} under $\psi^{\tau}_{s}$. Let $R_1> R$ be such that $\psi^{\tau}_s(D_R^{\ast}K)\subset D_{R_1}^{\ast}K$. 

Take the gauge transformation as in \eqref{perturbedpassive}. The Hamiltonian perturbation datum in the sense of Seidel \cite[Section (8f)]{SeidelZurich} is given by the Hamiltonian valued 1-form $B(s,t)=l'(s)H^{\tau}_{l(s)}ds$. Indeed, the corresponding Hamiltonian vector field valued 1-form is $Y^{\tau}=l'(s)X^{\tau}_{l(s)}ds$ and \eqref{perturbedpassive} just reads $(d\tilde{u}-Y^{\tau})^{0,1}=0$ as usual. Let $\mathcal{J}(K,R_1,J^0,J^1)$ be the space of homotopy of uniformly admissible almost complex structures $\hat{J}^{\tau}$ rel endpoints such that $\hat{J}^{\tau}=\tilde{J}^{\tau}$ outside $D_{R_1} ^{\ast}K$. Let $\mathcal{H}(R_1,K)$ be the space of Hamiltonians supported inside $D_{R_1}^{\ast}K$. 

Now further perturb the equation \eqref{perturbedpassive} by replacing $\tilde{J}^{\tau}$ with $\hat{J}^{\tau}$ in $\mathcal{J}(K,R_1,J^0,J^1)$ and $B^{\tau}(s,t)$ with $\hat{B}^{\tau}(s,t)=B^{\tau}(s,t)+Q^{\tau}(s,t)$ where $Q^{\tau}(s,t)$ is a family of Hamiltonian valued 1-forms taking values in $\mathcal{H}(R_1,K)$. We may assume that the $1$-form vanishes on the boundary. Let $\hat{Q}^{\tau}$ be the vector field valued 1-form obtained from $Q(s,t)^{\tau}$ and $\hat{Y}^{\tau}=Y^{\tau}+\hat{Q}^{\tau}$. Consider the following equation
\begin{align}\label{additionalperturbation}
\begin{cases}
(d\tilde{u}-\hat{Y}^{\tau})_{\hat{J}^{\tau}}^{0,1}=0 &  \\ 
\tilde{u}(s,0)\subset V&\\
\tilde{u}(s,1)\subset L_{l(s)}&\\
\lim_{s\to \infty} \tilde{u}(s,t)\in L_0\cap V\\
\lim_{s\to -\infty} {\tilde{u}}(s,t)\in L_1\cap V.
\end{cases}
\end{align}
Let $R_2>R_1$ such that the image of $(\psi_s^{\tau})^{-1}(D_{R_1}^{\ast}K)$ is contained in $D_{R_2}^{\ast}K$. The most important feature of \eqref{additionalperturbation} is that unlike \cref{{passivemovingboundaryeqfamily}}, we are now looking at pseudo-holomorphic strips with the same Lagrangian moving boundary conditions, that we may apply the standard homotopy method. Abusing notation, let $J^{\tau}$ be the pullback of $\hat{J}^{\tau}$ via $\psi^{\tau}_{l(s)}$. The pullback equation $(\psi_s^{\tau})^{-1}(\tilde{u})$ solves $(du-Z^{\tau})^{0,1}_{J^{\tau}}=0$ for  Hamiltonian vector field valued 1-forms $Z^{\tau}$ coming from the Hamiltonian-valued 1-form $Q(s,t)\circ \psi^{\tau}_s$.  Indeed, as before, \[(d\psi_s^{\tau})^{-1}(d\tilde{u}-Y^{\tau})_{\hat{J}^{\tau}}^{0,1}=(du)^{0,1}_{J^{\tau}}\] and so \[(d\psi_s^{\tau})^{-1}(d\tilde{u}-\hat{Y}^{\tau})^{0,1}_{\hat{J}^{\tau}}=(du-Z^{\tau})^{0,1}_{J^{\tau}}.\]
Note that $Z^{\tau}$ is supported on $D_{R_2}^{\ast}K$. In particular, the geometric energy $\int \abs{du-Z^{\tau}}^2$ is bounded above in terms of $u^{\ast}\omega$ and the curvature integrand (See \cite[(8(g))]{SeidelZurich}). \footnote{Actually, the curvature integrand vanishes in this situation since $Z^{\tau}$ vanishes on the boundary and \[\omega(\partial_s u-Z,J(\partial_s u-Z))=\omega(\partial_s u-Z,\partial_t u)=u^{\ast}\omega-dH(\partial_t u).\]
} The boundary conditions are the same as in \eqref{perturbedpassive} and outside $D_{R_2}^{\ast}K$, solutions of $(du-Z^{\tau})^{0,1}_{J^{\tau}}$ solves \eqref{passivecontinuationmap} that the solutions of \eqref{additionalperturbation} are still compactly confined on, say, $D_{R_3}^{\ast} K_1$, for any $\hat{J}^{\tau}$ and with respect to the bound on $\sup_{\tau}\abs{\nabla Q^{\tau}}$. Note that the bound on $\sup_{\tau}\abs{\nabla Q^{\tau}}$ will however, depend on $\hat{J}^{\tau}$. 

Then we may use the solutions of (\ref{perturbedpassive}) to construct the desired chain homotopy $H$. To achieve transversality, we use the Banach manifold $\mathcal{J}(K,R_1,J^0,J^1)$ and $\mathcal{H}(R_1,K)$ and run the standard transversality argument, say as in the proof of \cite[Lemma 8.7]{SeidelAbouzaid}). This is essentially the same strategy as in \cite[Section 4.5]{chriswendlholomorphiccurve}. Then we count the zero dimensional component of the moduli space of solutions of \eqref{additionalperturbation} for generic $\hat{J}^{\tau}$ and $\hat{Y}$; the $1$-dimensional component has boundary either that induced from strip breaking or the solutions of the equation \eqref{additionalperturbation} for $\tau=0,1$ so we get the desired chain homotopy relation. %The linearised equation is then given by \cite[Section (9k), 9.26]{SeidelZurich}
%\[(\delta K,\delta J,X)\to (\delta Y)^{0,1}+\delta J\circ \frac{1}{2}(du-Y)\circ j_{\mathcal{R}}+D\bar{\partial}_u(X)
%which can be shown to be surjective, by choosing $(\partial K,\partial J)$ such that $\partial Y$ sufficiently approximates the Dirac-delta at some injective point. 

We now discuss the second bullet point. Note that showing that the following commutative diagram holds up to chain homotopy, 
\[\begin{tikzcd}
& {{CF(\psi_1^{-1}(V),L_0,{\psi_1}^{\ast}J_1)}} \\
{{CF(V,L_0,J_0)}} && {{CF((\psi_1^{-1}\circ \psi_2^{-1})(V),L_0,(\psi_2\circ \psi_1)^{\ast}J_2)}}
\arrow["{\hat{c_{01}}}", from=2-1, to=1-2]
\arrow["{\hat{c_{12}}'}", from=1-2, to=2-3]
\arrow["{\hat{c}_{\psi_2 \circ \psi_1}}"', from=2-1, to=2-3]
\end{tikzcd}\]
reduces to the standard case discussed in \cite[Section 3.6]{Zapolsky} since the endpoint conditions match.  

The last point on passive continuation maps being quasi-isomorphisms follows because uniformly horizontally finite isotopies are compactly supported on $V$, and so are their inverses. Therefore, the "inverse movie" $\mathcal{V}^{-}:=\{(s,v):p\in V_{-s}\}$ is still tame. Hence the same argument applies and we can explicitly construct the chain inverse map. 
\end{proof}

Suppose now that $V$ is a vertically finite Lagrangian in $T^{\ast}M$ and suppose that the set 
\[V(F):=\{m\in M: T_m^{\ast}M \text{ is transverse to } V\}\]
is dense. Given two points $m,m'\in V(F)$ and a path homotopy class $\alpha$ between $m$ and $m'$, we can find a piecewise smooth representative of $\alpha$ such that i) each of the smooth components $\alpha_i$ are embedded curves in $M$, and ii) the endpoints are contained in the set $V(F)$. We call the induced passive continuation map the \textit{parallel transport map} associated to $\alpha$. From Proposition  \ref{propertiesofactivepassivecontinatuinmaps}, we readily obtain:

\begin{proposition}\label{propertiesofparalleltransportmaps}
A relative path homotopy class $\alpha$ between $m,m\in V(F)$ as above induces a \textit{parallel transport map}
\[\Gamma(\alpha): HF(V,F_{\alpha(0)},J_0)\to HF(V,F_{\alpha(1)},J_1)\]
with the following properties:
\begin{itemize}
\item parallel transport maps are isomorphisms,
\item parallel transport maps are compatible with respect to concatenation of paths,
\item parallel transport maps only depend on the path homotopy classes.
\end{itemize}
In particular, the assignment
\[z\mapsto HF(V,F_z)\]
equipped with the parallel transport maps defines a \textit{local system} on $M$.  
\end{proposition}
%\begin{rem}
%We expect that the active continuation maps should allow us to build a "flat at infinity" wrapped Fukaya category of the cotangent bundle $T^{\ast}M$, using the construction in \cite[Section 3]{GPSCV}. Perhaps it is worthwhile to point out the main conceptual difference with \cite[Proposition 3.23]{GPSCV}; GPS uses almost complex structures that are of contact type for sufficiently big enough compact subsets of $\partial_{\infty}X$ to show confinement properties for active continuation strips (which they simply call the continuation strips). For a Liouville sector $X$, and the  projection map $Nbh^Z \partial X$
%\end{rem}
\subsubsection{Path groupoid representation}\label{pathgroupoidrepsss}

We now relate everything we discussed to path groupoid representations of the Floer cohomology local system. Suppose $V$ is a vertically finite Lagrangian simple branched cover of $M$. Choose a finite set of points $\mathcal{P}_M(V)$ on $M$ such that  $F_b$ and $V$ are transverse for $b\in \mathcal{P}_M(V)$. Choose a grading on $T^{\ast}M$, and suppose furthermore that $V$ is spin and Maslov $0$ graded. We will assume that near the intersections $F_b\pitchfork V, b\in \mathcal{P}_M(V)$, the local configurations are trivial and isomorphic to 
\[\mathbb{C}^2_{(u=x+ip^x,v=y+ip^y)}, V:=\mathbb{R}_{(x,y)}^2+c_i, F_z:=i\mathbb{R}_{(p^x,p^y)}^2, J=J_{std}, \omega=\omega_{std}.\]
where $c_1,...,c_n$ are some distinct constant covectors. 

In order to simplify some sign computations, we will describe an alternative but equivalent approach for defining signs in Lagrangian Floer theory. For this, we need to pass to local systems on the sphere bundle rather than local systems on the actual surface. We'll explain this non-standard approach as follows. Given a relative homotopy class $w:(D,\partial D)\to (T^{\ast}M,V)$, the choice of the trivialization of ${w}^{\ast}(TV)$ along the boundary defines an orientation on the determinant line $\det D_{w}\bar{\partial}$ of the induced Cauchy-Riemann operator on the disk. This depends on the homotopy class of the (stable) trivializations. 

By \cite[Theorem 1A]{spinstructuresquadraticforms}, a smooth, simple closed curve lifts to a class in the sphere bundle $P_V$ modulo the class $H$ where $H$ is the distinguished class that winds around the fibre once. Since $V$ is a Riemann surface, there is a one-to-one correspondence between homotopy classes of trivializations ${w}^{\ast}TV$ and lifts of homology classes $w\vert_{\partial D}$ to the sphere bundle $P_{V}$. Now, according to \cite[Proposition 8.17]{FOOO2}, twisting by the class $H$ reverses the orientation. Thus, given a sphere bundle relative homology class modulo $2H$, one gets a distinguished orientation on the determinant line bundle. In other words, spin structures give distinguished lifts to sphere bundle homology classes modulo $2H$. As we will see in \cref{Wall-Crossing Analysis Section}, using sphere bundle homology classes greatly simplify some computations.

One can easily extend this discussion to the case of half-planes. Given a path $V_s$ of Lagrangians over $s\in [0,1]$, we get a Cauchy-Riemann operator on the \textit{upper half plane} with the boundary conditions given by $V_s$. This again depends on the homotopy class of the (stable) trivializations over $w(s,0)^{\ast}(TV_s)$. Write $q_{\theta}=\cos \theta \frac{d}{dx}+\sin \theta \frac{d}{dy}$, and $p_{\theta}=\cos \theta \frac{d}{dp^x}+\sin \theta \frac{d}{dp^y}$. Given the above local model, $\theta\in [0,2\pi)$, and a lift $\bar{m}$ of $m \in \mathcal{P}(V)$ to $V$, we choose the basis $\left \langle q_{\theta},q_{\theta+\frac{\pi}{2}} \right \rangle$ for $T_{\bar{m}}V$, and $\left \langle p_{\theta},p_{\theta+\frac{\pi}{2}} \right \rangle$ for $T_{\bar{m}}T^{\ast}_m M$. We will use the path of Lagrangian subspaces connecting $T_{\bar{m}}V$ and $T_{\bar{m}}V$, given by the path of their basis
\begin{align}\label{eq:Lagrangianpath} 
&\cos (-\frac{1}{2}\pi T)q_{\theta}+\sin (-\frac{1}{2}\pi T)p_{\theta} \nonumber \\
&\cos (-\frac{1}{2}\pi T)q_{\theta+\frac{\pi}{2}}+
\sin (-\frac{1}{2}\pi T)p_{\theta+\frac{\pi}{2}}
%&\cos (-\frac{1}{2}\pi T)\left(\cos \theta \frac{d}{dx}+\sin \theta\frac{d}{dy}\right)+\sin (-\frac{1}{2}\pi T)\left(\cos \theta \frac{d}{dp^x}+\sin\theta \frac{d}{dp^y}\right) \nonumber \\
%&\cos (-\frac{1}{2}\pi T)\left(-\sin \theta \frac{d}{dx}+\cos \theta \frac{d}{dy}\right)+
%\sin (-\frac{1}{2}\pi T)\left(-\sin \theta \frac{d}{dp^x}+\cos \theta \frac{d}{dp^y}\right)
\end{align}
for $T\in [0,1]$ which has grading $e^{-2\pi T}$. 

Following \cite[Section 10.1]{GNMSN}, we will say that a rank $1$ local system $\tloc$ on the sphere bundle of a Riemann surface is a \textit{twisted local system} if it has holonomy equal to $-1$ around each fibre. Observe that the choice of a spin structure $\mathfrak{s}$ is the same as a fibre-wise double covering of $P_M$. So after choosing a spin structure on $M$, we get a one-to-one correspondence between twisted local systems and genuine rank $1$ local systems on $M$, since pulling back by the spin structure, we get a local system $\loc$ on the sphere bundle with a trivial monodromy along the fibres. %In such a case, we will say that the twisted local system $\tloc$ and the induced genuine local system $\widetilde{\tloc}$ on $V$ are spin-equivalent.

It turns out that Floer-theroetic twisted local systems are easier to compute compared to computing the induced local systems on the base. So we will now sketch a method to obtain a Floer-theoretic twisted local system on $P_M$ given $V$ and $\tloc$. Given two base points $b, c$ for the path groupoid $\mathcal{P}_M$, let $\tilde{b}$ and $\tilde{c}$ be lifts of $b$ and $c$ to the sphere bundle $P_M$, $\tilde{\alpha}$ a smooth path connecting $\tilde{b}$ and $\tilde{c}$, and $\hat{b}$ and $\hat{c}$ intersection points in $T_b M\pitchfork L$ and $T_c M\pitchfork L$, respectively. We take a small perturbation without changing the path homotopy class, and break $\tilde{\alpha}$ into shorter segments so that each of these segments projects down to a smooth embedded path. So working over each segment at a time, we will assume that the projection is indeed embedded.

As observed in equation \eqref{perturbedpassive}, our continuation strip equation is gauge equivalent to a continuation strip equation with $V$ fixed and the fibres varying along the projection $\pi(\alpha(s))$. We can regard the path $\tilde{\alpha}$ as a choice of a trivialization of $T^{\ast}M_{\pi(\alpha(s))}$, where $\alpha=\pi(\tilde{\alpha})$. Now, equip our Lagrangian $V$ with a twisted local system. Let $\psi_{l(s)}\circ u$ be the (gauge-transformed) continuation strip. Given a sphere bundle lift $\tilde{\gamma}$ of the relative path homotopy class $\gamma$ given by $(\psi_{l(s)}\circ u)(s,0)$,
we can concatenate at both ends by the shortest paths in the circle fibre to obtain a path homotopy class in $P_V$ connecting the lifts of $\tilde{b}$ and $\tilde{c}$, modulo $2H$. As a result, we get a sphere bundle lift of the boundary of the disk obtained by glueing the half-strip operators at the strip-like ends. Twisting the count by the twisted flat connection $\tloc$, we see that our twisted count is independent of the choice of the sphere bundle lift because the $H$-contribution (which reverses the orientation) is cancelled out by the monodromy of $\tloc$. Thus, we obtain our Floer-theoretic parallel transport maps and a twisted family Floer local system $\widetilde{HF}=\widetilde{HF}(V,\tloc)$ on $M$.

Now choosing a spin structure $\mathfrak{s}_M$ on $M$, we can pullback the induced twisted local system on $M$ to a genuine local system on $M$. However, after choosing a spin structure $\mathfrak{s}_V$ on $V$,  we can construct family Floer cohomology as usual, using $\mathfrak{s}_M$, $\mathfrak{s}_V$, and the induced spin structures $\mathfrak{f}_b$ on the cotangent fibres. The only additional choice left is the path of Lagrangians and spin structures at each intersection point; for this, we use the path and the basis as indicated in \eqref{eq:Lagrangianpath} to trivialize the path $TV_T$, and use the trivial $Spin(2)$ bundle $P_T$ over it. Notice then that $(P_T)_0$ is identified with ${\mathfrak{s}_V}_{\hat{b}}$ and  $(P_T)_1$ is identified with $\mathfrak{f}_{\hat{b}}$.

We now proceed as usual. As before, such a choice defines an orientation on the determinant line of the linearized Cauchy-Riemann operator. We now twist $CF(V,F_b)$ with $\loc$, the induced rank $1$ local system on $V$ by
\begin{align}\label{eq:holonomymap}
\Phi^{\loc}(\partial (\psi_s\circ u)\vert_{(-\infty,\infty)\times \{1\}}): \loc_{\hat{b}}\to\loc_{\hat{c}}.
\end{align}
The previous discussion gives a (path groupoid representation of a) local system on $M$. One can then check directly that the family Floer cohomology local system $HF$ is equivalent to the pullback of the induced twisted local system  $\widetilde{HF}$ on $M$.

Suppose, furthermore, we have a compactly supported exact Lagrangian isotopy $V\sim V'$ such that the support lies outside $\pi^{-1}(\mathcal{P}_M(V))$. Then $CF(V',F_b)$ can also be made a graded chain complex over $\mathbb{Z}$ in a compatible manner. In particular, the quasi-isomorphisms $CF(V,F_b)\to CF(V',F_b)$ for $b\in P_M(V)$ commute with parallel transport maps.

The following proposition then summarizes our discussion.
\begin{proposition}\label{pathgroupoidrepresentation} Let $M,V,\mathcal{P}_M(V)$ as above, and fix a twisted rank $1$ local system $\tloc$ on $V$. The following data forms a path groupoid representation of a $GL(\mathbb{C};k)$-local system.
\begin{itemize}
\item The free $\mathbb{C}$-module $\dim HF(V,F_b, \loc)$. 
\item Parallel transport maps
\[\Gamma(\alpha):HF(V,F_b, \loc)\to HF(V,F_{c}, \loc)\]
defined as in \ref{propertiesofparalleltransportmaps}
\end{itemize}
Furthermore, let $\Gamma'(\alpha)$ denote the parallel transport maps associated to the $\mathbb{Z}$-modules $HF(V',F_b,\loc)$ for ${b\in\mathcal{P}_M(V)}$. Then the two path groupoid representations\\ $(\mathcal{P}_M(V),HF(V,F_b,\loc),\Gamma(\alpha))$ and  $(\mathcal{P}_M(V),HF(V',F_b,\loc),\Gamma'(\alpha))$ are equivalent. Finally, there exists an induced Floer theoretic twisted local system  $\widetilde{HF}(V,\tloc)$ on $M$, and its spin pullback is equivalent to the family Floer cohomology local system  $HF(V,F_b, \loc)$. 
\end{proposition}

We may instead think of the globally defined local system $z\mapsto HF(V,F_z)$ as the local system induced from the path groupoid representation $(\mathcal{P}_M(V),HF(V,F_b),\Gamma(\alpha))$. We will switch between these two conceptual pictures depending on whichever is more convenient. 	

\section{Desingularization and real-exact spectral curves}\label{Spectral Curves}
In this section we discuss the geometry and topology of real-exact spectral curves. In Section \ref{Desingularization of metric}, given a small deformation parameter $\delta>0$, we deform the singular metric $g^{\phi}$ on $C^{\circ}$ to a K\"{a}hler metric $g^{\phi}_{\delta}$ on $\tilde{C}$ as we discussed briefly in Section \ref{Proofguidess}. In Section \ref{moduliproblemss}, we discuss Ekholm's conformal models on disks with a single positive puncture, and define BPS discs. In Section \ref{toycase subsec}, we look at the toy case of $\phi=zdz^2$ on $\mathbb{C}=\mathbb{C}P^1-\{\infty\}$, whose spectral curve is isomorphic to $\scurve=\{(p^z)^2-z=0\}$ on $\mathbb{C}^2=T^{\ast}\mathbb{C}$. Then we discuss how the associated spectral network is related to BPS discs.

In Section \ref{Domain Decomposition of C},  we discuss a domain decomposition of $C$ induced from a \textit{complete saddle-free GMN quadratic differential} $\phi$. In Section \ref{Real-Exact Spectral Curves ss}, we discuss the geometry of real-exact spectral curves. Like we said in Section \ref{Proofguidess}, we show that given an energy cut-off $E\gg 1$ we can deform $C-\snetwork(0)$ to a bounded open subdomain $C(\delta;E)$ of $C$ such that horizontal trajectories passing through $z\in C(\delta;E)$ never enter sufficiently small neighbourhoods of the zeroes of $\phi$. Furthermore, we show in Proposition \ref{horizontaldomaindevision} that outside of $\snetwork(\pi/2)$, we have a canonical $\pm$ ordering on the lifts of the points in $z\in \tilde{C}$ with respect to the projection $\pi:\scurve\to \tilde{C}$. 

%Then like we said in \ref{introFloerbifur}, we will then find an "initial" deformation retract $C(\delta;\infty)$ of $C-S(0)$, which is traced out by some family of horizontal trajectories. In particular, $C(\delta;\infty)$ will be such that small vertical thickenings of $C(\delta;\infty)$ will completely avoid $U((2+\eta)\delta)$ for some $0<\eta\ll 1$. 

\subsection{Desingularization}\label{Desingularization of metric} 

We provide a way of deforming the singular $\phi$-metric to a smooth metric on $\tilde{C}$. We first start with the case of $\phi=zdz^2$. This deformation depends on some auxiliary choices but all the deformed metrics are conformally equivalent and they agree near infinity. The singular flat metric $g^{\phi}=\abs{z}\abs{dz}^2$ in polar coordinates reads 
\begin{align}
	g^{\phi}=r\big(dr^2+r^2 d\theta^2\big).
\end{align}
Choose a $\delta>0$ and a smooth strictly increasing positive function $\psi_\delta:[0,\infty)\to [1,\infty)$ such that $\psi_{\delta}(r)=r$ for $r<\delta$ and $\psi_{\delta}(r)=1$ for $r>\frac{3}{2}\delta$. The metric
\begin{align}\label{deformedradialmetric}
	g_{\delta}^{\phi}=\frac{r}{\psi_{\delta}(r)}\big(dr^2+r^2 d\theta^2\big)
\end{align}
is now globally defined on $\mathbb{C}$, and conformal hence invariant with respect to the standard complex structure on $\mathbb{C}$. So $g_{\delta}^{\phi}$ is actually a K\"{a}hler metric since we are in complex dimension $1$, though it is not real analytic.

Recall that we call a quadratic differential \textit{complete} if it does not admit poles of order one. Let $\phi$ be a complete GMN quadratic differential and let $b_1,...,b_n$ be the zeroes of $\phi$. Recall from Proposition \ref{zeroofphiprop} that near a simple zero of $\phi$, there exists a neighbourhood $U_b$ of $b$, an open set $D$ of $\mathbb{C}$ containing zero, and a biholomorphism $\xi=\xi_b:(D,0)\to (U_b,b)$ such that $\phi(\xi)d\xi^2=\xi d\xi^2$.  Let $U_i=U_{b_i}$ and $\xi_i=\xi_{b_i}$. By shrinking if necessary, we may assume that the open sets $U_i$ are disjoint and that $\xi_i^{-1}(U_i)=D(r_i)$ for some $r_i>0$. Having made these choices, we define:
\begin{definition}\label{zeroneighbourhood}
	Let $0<r<\min\{r_1,...,r_n\}$. Let $b_i$, $U_i$, $\xi_i$ and $D(r_i)$ be as above.   Let $U_i(r)=\xi_i(D(r))$. We define 
	\begin{align}
		U(r)=\bigcup_{i=1}^n U_i(r). 
	\end{align}
\end{definition}

By choosing $\delta<\frac{1}{2}\min\{r_1,...,r_n\}$, we use the local form \eqref{deformedradialmetric} near each branch point to conformally deform the flat metric $g^{\phi}$ to obtain a global smooth metric on $\tilde{C}$ which we still denote as $g_{\delta}^{\phi}$. Note that for any other choice of $\delta$ and $\psi_{\delta'}$ gives a metric which is conformally equivalent to $g_{\delta}^{\phi}$. Furthermore, the conformal factor is a smooth positive function which is equal to $1$ except on some small annular regions near each of the zeroes of $\phi$.  We call the metrics obtained by this general method \textit{(conformally) desingularized metrics}.

%Let $U(r)$ denote the set of points $z$ such that $\abs{z}<\delta$. 
%which interpolates between the standard flat metric near the origin and the flat $\phi$-metric outside of  $2\delta>r$. 
\subsection{Conformal structures}\label{moduliproblemss}
We discuss the conformal model $\triangle_m$ of the closed unit disc with $m-1$ out-going punctures and one in-going puncture on the boundary, which was constructed by Ekholm in \cite[Section 2.1]{Morseflowtree}. Given points $c=(c_1,...,c_{m-2})\in \mathbb{R}^{m-2}$, we consider the subdomain of $(-\infty,\infty)\times[0,m]$ given by removing $m-2$ horizontal slits in the direction of $+\infty$, of width $0<\slit\ll 1$, starting from the points $(c_j,j)$ for $j=1,...,m-2$. A boundary component $I$ with both of its ends at $+\infty$ is called a \textit{slit boundary component}. Given a slit boundary component $I$, the \textit{boundary minimum} of $I$ is the unique point with the smallest real part along $I$. We can regard each of these subdomains as giving conformal structures on $\triangle_m$ induced by $z=s+it$. 
\begin{figure}
	\centering 
	\includegraphics[]{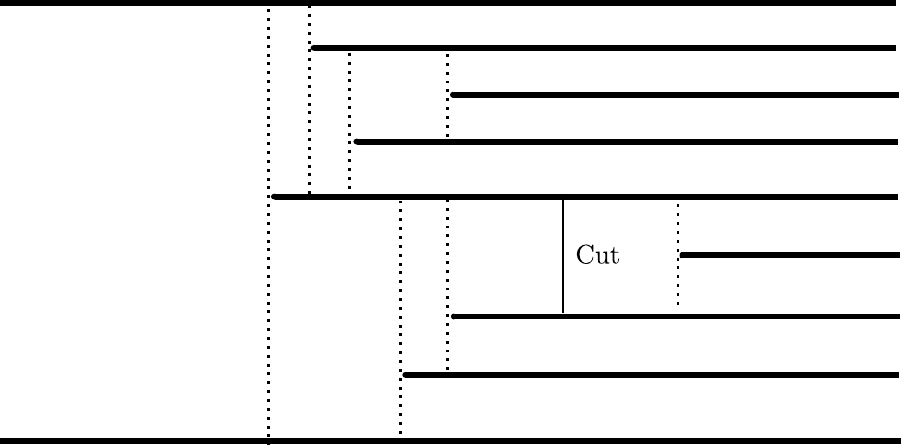}
	\caption{An example of $\triangle_m$ with the vertical rays passing through the boundary minima, together with a possible vertical cut. }
\end{figure}
\begin{figure}
	\usetikzlibrary{decorations.pathmorphing}
\usetikzlibrary{decorations.markings}
% Set the overall layout of the tree
\tikzstyle{level 1}=[level distance=3.5cm, sibling distance=3.5cm]
\tikzstyle{level 2}=[level distance=3.5cm, sibling distance=2cm]

% Define styles for bags and leafs
%\tikzstyle{bag} = [text width=4em, text centered]
%\tikzstyle{end} = [circle, minimum width=3pt,fill, inner sep=0pt]
\tikzset{
   red/.style={draw=red,
   , postaction={decorate},
        decoration={markings}},
    blue/.style={draw=blue, 
    postaction={decorate},
        decoration={markings}}, 
            null/.style={draw=blue, 
    postaction={decorate},
        decoration={markings}}, 
}

% The sloped option gives rotated edge labels. Personally
% I find sloped labels a bit difficult to read. Remove the sloped options
% to get horizontal labels. 
\begin{tikzpicture}[
        thick,
        % Set the overall layout of the tree
        level/.style={level distance=2cm},
        level 2/.style={sibling distance=2.6cm},
        level 3/.style={sibling distance=1cm}
    ]
    \coordinate
        child[grow=left]{
           edge from parent
        }
        % I have to insert a dummy child to get the tree to grow
        % correctly to the right.
        child[grow=right, level distance=0pt] {
       child{
       child{
       node {}
       edge from parent 
       }
       child{
       child{
       node {}
       edge from parent 
       }
       child{
       child{
       node {}
       edge from parent 
       }
       child{
       node {}
       edge from parent 
       }
       edge from parent
       node (n1) [above] {}
        node (n2) [below] {}
        }
       edge from parent 
       }
       edge from parent 
       }
       child{
       child{
       child{
       node {}
       edge from parent 
       }
       child{
       child{
       node {}
       edge from parent 
       }
       child{
       node {}
       edge from parent 
       }
       edge from parent 
       }
       edge from parent 
       }
        child{
       node {}
       edge from parent 
       }
       edge from parent 
       }
    };
\draw[dashed, shorten <=-0.5mm, shorten >=-0.5mm] 
    ([xshift=5mm,yshift=0mm] n1.north)       -- ([xshift=5mm, yshift=-0.5mm] n2.south);
\end{tikzpicture}
	\caption{The tree $T$ obtained from $\triangle_m$ together with the vertical cut.}
	\label{fig:tree0}
\end{figure}

Note that translating by $(a,...,a)$ on $\mathbb{R}^{m-2}$ for $a\in \mathbb{R}$ gives a biholomorphism of this subdomain and hence a conformal equivalence between two different conformal structures on $\triangle_m$. Quotienting $\mathbb{R}^{m-2}$ by the $\mathbb{R}-$action gives $\mathbb{R}^{m-3}$. In \cite[Section 2.1]{Morseflowtree}, Ekholm shows that there is a diffeomorphism between $\mathbb{R}^{m-2}/\mathbb{R}$ and the space of conformal structures on $\triangle_m$. In particular, we recover the unique conformal structure on $\triangle_3$. 

Given such a conformal model, we obtain a stable $m-1$-leaved tree $T$ as follows. At each of the boundary minima of $\triangle_m$, introduce a vertical ray in $\triangle_m$ passing through the boundary minimum, connecting a boundary point to a boundary point, and consider the resulting subdivison of $\triangle_r$.
Then we let $T$ be the tree with edges the connected components of the complement of the vertical rays, which share a common vertex if the corresponding components are adjacent along a vertical ray. For instance, when $m=3$, we obtain the unique $2$-leaved stable tree. The tree is directed \textit{away} from the root. See Figure \ref{fig:tree0} for an example.

Now suppose we are given a $\triangle_m$ and we want to cut it open along a vertical segment $l$ whose boundary end-points strictly lie on the horizontal boundaries of $\triangle_m$. In this case, we get a single component in the complement that lies on the right-hand side of $l$. This right-hand side component is of the form $\triangle_{m'}\cap (a,+\infty)\times [0,m']$ for some $a,-\infty<a<\infty$ lying strictly left to the boundary minima of $\triangle_{m'}$. From now on, we will call such a component the \textit{right-adjacent component.} 

We now define $\epsilon$-BPS discs ending at $z$ as follows. Let $\mathcal{Z}=(-\infty,\infty)\times [0,1]$ be the infinite strip. As before, let $\phi$ be a complete GMN quadratic differential, $g^{\phi}$ the induced singular flat metric on $C$, and $g^{\phi}_{\delta}$ a desingularization of $g^{\phi}$ that we constructed in Section \ref{Desingularization of metric}. We do not require $\phi$ to be saddle-free. Here $J=J_{\phi}$ is the induced almost complex structure on $T^{\ast}\tilde{C}$ with respect to $g_{\delta}^{\phi}$, and $J_{con}$ its conical deformation. Then
\begin{definition}\label{BPSdiskmodels}
	%A map $u:\mathcal{R}\to T^{\ast}\tilde{C}$  in the half-strip model if it satisfies the following equation:
	%\begin{align}\label{moduliproblemeq}
	%	\begin{cases}
		%		\bar{\partial}_{J}u=0\\
		%		u(s,0)\subset \scurve\\ 
		%		u(s,1)\subset \scurve\\
		%		\lim_{s\to -\infty}u(s,t)\in \scurve\\
		%		u(0,t)\subset F_z\\
		%		\lim_{(s,t)\to (0,1)} u(s,t)\neq \lim_{(s,t)\to (0,-1)} u(s,t).
		%	\end{cases}
	%\end{align}
	Let $0\leq \epsilon\leq 1$. A map $u:\mathcal{Z}\to T^{\ast}\tilde{C}$ is an $\epsilon$-\textit{BPS disc} ending at $z$ in the infinite strip model if it satisfies the following equation:
	\begin{align}\label{moduliproblemeq3}
		\begin{cases} 
			\bar{\partial_{J}}u=0\\
			u((-\infty,\infty)\times \{0\})\subset \epsilon\scurve\\
			u((-\infty,\infty)\times \{1\})\subset F_z\\
			\lim_{s\to \pm \infty}u(s,\tau)\in F_z\cap \epsilon\scurve\\
			\lim_{s\to -\infty}u(s,\tau)\neq \lim_{s\to \infty}u(s,t).
		\end{cases}
	\end{align}
\end{definition}
\subsection{\texorpdfstring{The toy case $\phi(z)=zdz^2$}{The toy case}}\label{toycase subsec}

Now we discuss the case of $\phi=z dz^2$ to illustrate how the spectral network relates to the existence of BPS discs. We remark that for general complete GMN quadratic differentials, one needs the adiabatic degeneration argument in Section \ref{Adia Degen}. 

We can identify $\mathbb{C}^2\simeq T^{\ast}\mathbb{C}$ and $\scurve$ with $\{{(p^z)}^2-z=0\}$ in $\mathbb{C}^2$. Recall that $g^S$ is the metric induced on $\mathbb{C}^2$ and $\Omega$ is the canonical holomorphic symplectic form on $\mathbb{C}^2$ (Defined in \ref{introquadss}). Let $\tilde{I}$ be the horizontal lift of $I$.  In conformal normal K\"{a}hler coordinates, we have:
\begin{align}\label{eq:hpkrotation}
	&g^S= \abs{dz}^2+\abs{d(p^z)}^2 
	&\Omega=dp^z\wedge dz 
	&&J=\begin{bmatrix}
		0 & Id\\ 
		-Id & 0
	\end{bmatrix}
	&&\tilde{I}=\begin{bmatrix}i & 0\\ 
		0 & -i \end{bmatrix}.\end{align}
Note that $g^S$ is $\tilde{I}$ and $J$ invariant. Let $K=\tilde{I}J$ and  $\omega_I=g^S(\tilde{I}-,-)$. Then the imaginary part $\omega_{\pi/2}=\frac{1}{2i}(\Omega-\bar{\Omega})$ of $\Omega$ is given by $g^S(\tilde{K}-,-)$. Furthermore, since
\[\omega_K(v,Jv)=g^S(Kv,Jv)=-g^S(v,KJv)=-g^S(Iv,v)=\omega_I(v,v)=0,\]
the imaginary part of $\Omega$ vanishes on the interior of a $J$-holomorphic disc. 

The spectral curve $\scurve$ is exact with respect to the holomorphic Liouville form $\lambda$. We choose the primitive ${\primitive}=\frac{2{(p^z)}^3}{3}$ for $\lambda$ on $\scurve$. Given a complex number $z\in \mathbb{C}$, write 
\[z_{\theta}= \frac{e^{-i\theta}z+e^{i\theta}\bar{z}}{2}.\]
For the quadratic differential $\phi=zdz^2$, the spectral network $\snetwork(\theta)$ consists of three positive rays of phases $e^{i\frac{2\theta+2\pi k}{3}},k=0,1,2$ emanating from the origin. So we see that we have the following alternative characterization of the spectral network $\snetwork(\theta)$ in terms of the holomorphic primitive $W$:

\begin{proposition}\label{snetworkchar1}
	The spectral network $\snetwork(\theta)$ is the locus of points $z$ on $\mathbb{C}$ such that
	\begin{align}\label{holomorphiccharacterizationsnetwork}
		W(\pi^{-1}(z))=W(\pm \sqrt{z},z)_{\theta+\frac{\pi}{2}}=(\pm z\sqrt{z})_{\theta+\frac{\pi}{2}}=0.
	\end{align}
\end{proposition}
For $a\in \mathbb{C}-\{0\}$, let $\{a^0,a^1\}$ be the set of the lifts of $a$ on $\scurve$. Since $W(w,z)=\pm\frac{2}{3}  z\sqrt{z}$ on $(z,w)\in \scurve$, we see that $\snetwork(0)$ is the locus of points $a\in \mathbb{C}$ such that the imaginary value of $W(a^i)$ is equal to zero for $i=0,1$. 
We now give an ordering to the pair provided that $\text{Re}(W(a^0))\neq \text{Re}({\primitive}(a^1))$. The equality happens if and only if the real part vanishes, which then implies that $a$ is on $\snetwork(\pi/2)$. Based on this fact, for $a\notin \snetwork(\pi/2)$, we order the two lifts of $a$ by $a^{\pm}$ with respect to the relation
\[\text{Re}({\primitive}(a^+))>\text{Re}({\primitive}(a^{-})).\]
We will construct a similar ordering in Section \ref{Real-Exact Spectral Curves ss}. 

We now provide a Floer theoretic reformulation of the characterisation of the spectral network for $\{{(p^z)}^2-z=0\}$. From now on we fix the phase $\theta=0$.

\begin{proposition}\label{S-Networkw2z}
	Let $\phi=zdz^2$ on $\mathbb{C}$. The spectral network $\snetwork(0)$ is locus of the points $z$ on $\mathbb{C}$ such that there exists a BPS disc ending at $z$.
\end{proposition} 
\begin{proof}
	Recall here that we are using $J$ as in \eqref{eq:hpkrotation}. We utilise the exactness of the holomorphic Liouville form. Since
	\[\int u^{\ast}\Omega={\primitive}(z^+)-{\primitive}(z^{-})=\frac{4z^{3/2}}{3},\] and $\omega_{\pi/2}$ vanishes in the interior of any $J$-holomorphic disc, there can be no BPS disc ending at $z$ for $z\notin \snetwork(0)$. To construct explicitly some BPS disc ending at $z\in \snetwork(0)$, observe that the intersection of the $J$-holomorphic plane $\{(e^{\frac{2\pi i k}{3}}s, e^{\frac{-2\pi i k}{3}}t), s,t\in \mathbb{R}\}$ in $\mathbb{C}^2=\{(x+ip^x,y+ip^y)\}$ intersects $\scurve$ along the curve $y^2=x$ and $F_{z}$ for $z=te^{\frac{-2\pi i k}{3}}$ along the line $x=t$, for $t>0$. These two curves bound a disc. %But since these $J$-discs lie in the unit disc bundle, we see that they are in fact $J_{con}$-discs. 
\end{proof}

Notice that for the case $\phi=zdz^2$, Proposition \ref{S-Networkw2z} is much stronger than Theorem \ref{maintheorem}. However, our argument rests on exactness with respect to the holomorphic Liouville form $\holliouvile$, so it cannot be extended for general spectral curves $\scurve$. \\
\subsection{Domain decomposition}\label{Domain Decomposition of C}

We now discuss the domain decomposition that comes from the spectral network $\snetwork(0)$ associated to a saddle-free GMN quadratic differential $\phi$. We assume that $\phi$ is GMN and complete. We use the conventions introduced in the $\phi$-metric part in Section \ref{introquadss}.

\begin{definition}
	Given a class $[\gamma]\in H_1(\scurve;\mathbb{Z})$, its \emph{charge} $Z(\gamma)$ is defined by the integral
	\[Z(\gamma)=\int_{\gamma}\lambda\]
	where $\gamma$ is a smooth representative of $[\gamma]$. The induced $\mathbb{Z}$-additive homomorphism 
	\[Z:H_1(\scurve;\mathbb{Z})\to \mathbb{C}\]
	is called the \emph{charge homomorphism}.
\end{definition}

Given a saddle trajectory $\gamma$ of phase $\theta$ we can join the two lifts of $\gamma$ so that the charge of the corresponding class in $H_1(\scurve,\mathbb{Z})$ is of phase $e^{-i\theta}$. Furthermore, by rotating the quadratic differential $\phi$ to $e^{i2\theta}\phi$ for generic $\theta$, we can make the image of $Z$ avoid $\mathbb{R}_{>0}\cup \mathbb{R}_{<0}$. This means that by rotating the quadratic differential by a generic phase, we can always obtain a saddle-free quadratic differential (See \cite[Lemma  4.11]{bridgeland2014quadratic}).

We have the following result on the conformal equivalence classes of the connected components (which we called the chambers) of $C-\snetwork(0)$ for saddle-free, complete quadratic differentials $\phi$. For the proof, see Chapters 6 and 9-11 of \cite{quaddiff}, and Sections 3.4-3.5 and Lemma 3.1 of \cite{bridgeland2014quadratic}. 

\begin{proposition} \label{domaindecomposition}
	Let $\phi$ be a complete, saddle-free quadratic differential. Then the connected components of $\tilde{C}-\snetwork(0)$) are conformally equivalent to one of the following.
	\begin{itemize}
		\item Vertically finite horizontal strips 
		\[\mathcal{Z}(a,b)=\{z\in \mathbb{C}:a<Im(z)<b\}\] for some $-\infty<a,b<\infty$. The boundary of $\mathcal{Z}(a,b)$ consist of separating horizontal trajectories given by extending the biholomorphism to the lines $\{Im(z)=a, Re(z)> a_0\}, \{Im(z)=a, Re(z)< a_0\}, \{Im(z)=b, Re(z)< b_0\}, \{Im(z)=b, Re(z)> b_0\}$ for some $a_0,b_0\in \mathbb{R}$. In other words, the biholomorphism extends to a continuous map $\overline{\mathcal{Z}(a,b)}\to \mathbb{C}$ which is a surjection onto the closure of the corresponding horizontal chamber component, such that the points $a_0+ia$ and $b_0+ib$ are mapped to zeroes of $\phi$. 
		\item The open upper half-plane \[\mathcal{H}:=\{z\in \mathbb{C}:im(z)>0\}.\] Again, there exists some $x_0\in \mathbb{R}$ such that the biholomorphism extends to a continuous map $\overline{\mathcal{H}}\to \tilde{C}$ which is a surjection onto the closure of the corresponding horizontal chamber component, where the point $x_0+i\cdot 0$ is mapped to a zero of $\phi$, and the lines $\{im(z)=0, re(z)>x_0\}$ and $\{im(z)=0, re(z)<x_0\}$ are mapped to separating horizontal trajectories. 
	\end{itemize}
	In both cases, the pullback of $\phi$ under the conformal equivalence is equal to $dz^2$. In fact, these domains are given by maximal analytic continuations of $\int \sqrt{\phi(z)}$ along open neighbourhoods of generic horizontal trajectories. Both of these domains are traced out by generic horizontal trajectories. 
\end{proposition}

From now on, we will not distinguish between the horizontal chambers of $\phi$ (which are open conformal subdomains of $\tilde{C}$) and their conformally equivalent counterparts $\mathcal{Z}(a,b)$ and $\mathcal{H}$ (which are open conformal subdomains of $\mathbb{C}$). From the proposition, we see that given a $\delta>0$ we have an $\epsilon(\delta)>0,\:h(\delta)>0$ and $\eta(\delta)>0$ such that the $h(\delta)$-neighborhoods of horizontal trajectories which trace out the horizontal subdomains
\begin{align}
	\mathcal{Z}(\delta;a,b)&=\mathcal{Z}(a+\epsilon(\delta),b-\epsilon(\delta))\subset \mathcal{Z}(a,b)\\
	\mathcal{H}(\delta)&=\mathcal{H}\cap \{\im{y}>\epsilon(\delta)\}
\end{align}
\textit{never} enter the (slightly thickened) neighbourhood $U((2+\eta)\delta)$. For later purposes, we demand that $\eta>0$ is small enough so that outside $U((2-\eta)\delta)$, $g_{\delta}^{\phi}=g^{\phi}$. We sometimes call the latter neighbourhood the \textit{desingularization region}.

Note that $\mathcal{Z}(\delta;a,b)$ and $\mathcal{H}(\delta)$ are naturally deformation retracts of the horizontal chambers of $\phi$. Taking the union of the horizontal subdomains $\mathcal{Z}(\delta;a,b)$ and $\mathcal{H}(\delta)$ inside each of the horizontal chambers, we obtain our domain $C(\delta;\infty)$.

\begin{proposition}
	There exists a conformal subdomain $C(\delta;\infty)\subset \tilde{C}$ which is a disjoint union of deformation retracts of connected components of $\tilde{C}-\snetwork(0)$ which satisfies the following.
	\begin{itemize}
		\item[] There exists an $h=h(\delta;E)>0$ and an $\eta(\delta)>0$ such that if $\gamma$ is a horizontal trajectory passing through $z\in C(\delta;\infty)$, then the $h(\delta)$-neighborhood of $\gamma$ lies strictly outside the desingularization region $U((2+\eta)\delta)$. 
	\end{itemize}
\end{proposition}

\subsection{Real-exact spectral curves}\label{Real-Exact Spectral Curves ss}

We now look at \textit{real-exact} quadratic differentials $\phi$, which, as stated in the introduction, is the main object of our interest. Recall that (Section \ref{introquadss}) we have the identification of the real cotangent bundle and the holomorphic cotangent bundle via
\[dx\to dz, dy\to -idz.\]
Recall that a complete GMN quadratic differential $\phi$ is called \textit{real-exact} if the spectral curve $\scurve$ associated to $\phi$ is sent to a $\lambda_{re}$-exact Lagrangian. Equivalently, this means that $\scurve$ is exact with respect to the real part of the holomorphic Liouville form:
\[\lambda_{\theta=0}:=\frac{\lambda+\bar{\lambda}}{2}.\] 

We discuss when saddle-free GMN quadratic differentials give real exact spectral curves. Then for $\phi$ real-exact, we find an open subdomain $C(\delta;E)\subset \tilde{C}$ which is a deformation retract of $C(\delta;\infty)$, such that the energy of an $\epsilon$-BPS disc ending at $z\in C(\delta;E)$ is a priori bounded above by $2\epsilon E$. Furthermore,  we also construct a vertical neighbourhood $\mathcal{V}$ of the ``truncated" spectral network (see Definition \ref{truncatedspectralnetworkdef}) such that we have a preferred ordering $z^+,z^{-}$ of the lifts $\pi^{-1}(z)$, for which the geometric energy of an $\epsilon$-BPS disc ending at $z$, that travels from $\epsilon z^{+}$ to $\epsilon z^{-}$, is strictly negative; hence we show the \textit{non-existence} of such $J$-discs. 
\subsubsection{Criterion for real exactness}\label{criterionforrealexactnessss}

Given a horizontal strip $(\mathcal{Z}(a,b),\phi=dz^2)$, consider the saddle trajectory given by connecting the two zeroes of $\phi$ on the horizontal boundary segments of $\mathcal{Z}(a,b)$. Such saddle trajectories are called \textit{standard saddle trajectories}. The corresponding homology classes in $H_1(\scurve;\mathbb{Z})$ given by joining the two lifts of the straight line are called \textit{standard saddle  classes}. From standard saddle classes, we obtain the following criterion for real-exactness.
\begin{proposition}
	The Lagrangian $\scurve$ with respect to the canonical symplectic form $\omega$ of the real cotangent bundle is real-exact if and only if the standard saddle trajectories all have purely imaginary charge.
\end{proposition}\label{realexactnesscrit}
\begin{proof}
	The natural involution on the spectral curve induces a $\mathbb{Z}_2$-action on the homology group $H_1(\scurve;\mathbb{Z})$. Define the hat-homology group $\widehat{H_1({\phi})}$ to be the $\mathbb{Z}_2$ anti-invariant part of $H_1(\scurve;\mathbb{Z})$. Then \cite[Lemma 3.2]{bridgeland2014quadratic} shows that the hat-homology group $\widehat{H_1({\phi})}$ is generated by the standard saddle classes of $\phi$.
	
	Since $\lambda$ is $\mathbb{Z}_2$-anti-invariant, the $\mathbb{Z}_2$-invariant part of $H_1(\scurve;\mathbb{R})$ lies in $\ker Z$. Hence the charge homomorphism on $H_1(\scurve;\mathbb{R})$ factors through $\widehat{H_1({\phi})}\otimes \mathbb{R}$. But $\widehat{H_1({\phi})}\otimes \mathbb{R}$ is spanned by the standard saddle classes of $\phi$. So it follows that $\scurve$ is real-exact if and only if the image of $Z$ over $\widehat{H_1({\phi})}\otimes \mathbb{R}$ lies on the imaginary axis which is if and only if the image of $Z$ over the standard saddle classes are all imaginary. %Now, the image of a standard saddle class under $Z$ is equal to its $\phi$-phase $\pm e^{i\theta}$ times its $\phi$-length.% So we see that $\scurve$ is real-exact if and only if the phases associated to the standard saddle classes are all imaginary. 
\end{proof}

\begin{rem}
	The standard saddle classes give a $\mathbb{Z}$-basis in the lattice $\widehat{H_1({\phi})}$. Hence we can identify it with $\mathbb{Z}^{\oplus n}$. Following Bridgeland and Smith \cite[Section 2.5]{bridgeland2014quadratic}, let $Quad_{free}(g,m=(\{m_1,p_1\},...,\{m_k,p_k\})$ be the space of pairs $(C,\phi)$, where $C$ is a genus $g$ closed Riemann surface and $\phi$ is a quadratic differential over $C$, such that the poles of $\phi$ are the points $p_i$ with the order $m_i$. We identify the pairs $(C,\phi)$ and $(C',\phi')$ up to conformal equivalence. Then in \cite[Proposition 4.9]{bridgeland2014quadratic}, Bridgeland and Smith show that $Quad_{free}(g,m)$ is locally isomorphic to the space of $\mathbb{Z}$-homomorphisms from $\mathbb{Z}^n$ to $\mathbb{C}$, which implies that it is a complex manifold of dimension $\dim \widehat{H_1(\phi)}$. Restricting to the homomorphisms which map entirely to $i\mathbb{R}\subset \mathbb{C}$, we see that the real-exact quadratic differentials form a totally real submanifold of $\text{Quad}_{free}(g,m)$. \\
\end{rem}
\subsubsection{Energy and horizontal distance}\label{energyandhorizontaldistancessss}
Let $\primitive$ be the primitive of $\lambda_{Re}$ over $\scurve$. We now relate $\primitive$ to the $\phi$-length. Given an arc-length parametrized $\phi$-geodesic $\alpha:[0,l]\to C^{\circ}$ of phase $\theta$, let $\tilde{\alpha}$ be a lift of $\alpha$ onto $\scurve$. Then
\begin{lemma}\label{alphalength}
	We have 
	\begin{align}
		\int_{\tilde{\alpha}} \lambda=\pm e^{-i\theta}l.
	\end{align}
\end{lemma}
\begin{proof}
	Take a flat local conformal coordinate $\int \sqrt{\phi(z)}dz$, sending $\alpha(0)$ to $0$, over which $\alpha$ reads $e^{-i\theta}t$, and $\scurve:=\{p^x=\pm 1, p^y=0\}$. So the value of the holomorphic Liouville form is just $\pm 1\cdot e^{-i\theta}$. Integrating this from $t=0$ to $t=l$ gives (\ref{alphalength}). 
\end{proof}

\begin{definition}\label{horizontaldistancedef}
	Let $\mathcal{Z}^h$ be a horizontal chamber. Suppose $z,z'$ are two points on the closure of $\mathcal{Z}^h$ in $\tilde{C}$. Choose the shortest straight line segment $l$ in $\overline{\mathcal{Z}}^h$ connecting $z$ and $z'$ and a lift $\tilde{l}$ of $l$ to $\scurve$. Then the horizontal distance $d_{hor}(z,z')$ is defined by:
	\[d_{hor}(z,z'):=\abs{\int_{\tilde{l}}\lambda_{re}}.\]
\end{definition}
Since any other choice of $\tilde{l}$ reverses the sign of $\int_{\tilde{l}}\lambda_{re}$ by $-1$, the horizontal distance is well-defined. Since $\phi$ is real-exact, all the standard saddle trajectories are vertical. By translating if necessary, we may assume that the standard saddle trajectory on $\mathcal{Z}^h$ lies on $x=0$. The following lemma is a direct computation  (See Figure \ref{horidistance}).

\begin{figure}[t]
	\includegraphics[width=0.8\textwidth]{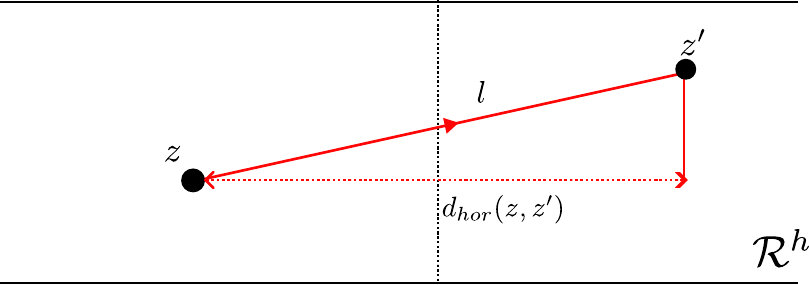}
	\centering
	\caption{A horizontal chamber $\mathcal{Z}^h$. The dotted line segment denotes the vertical saddle trajectory in $\mathcal{Z}^h$. $l$ is the shortest line segment between $z$ and $z'$ in $\mathcal{Z}^h$. }
	\label{horidistance}
\end{figure}
\begin{lemma}\label{horizontallengthcontrol}
	Let $\mathcal{Z}^h$ be a horizontal chamber and let $z,z'$ be two points on the closure of $\mathcal{Z}^h$. Suppose $z=x+iy, z'=x'+iy'$ under some $\phi$-flat conformal equivalence $(\mathcal{Z}^h,\phi)\simeq (\mathcal{H},dz^2)$ or $(\mathcal{Z}^h,\phi)\simeq (\mathcal{Z}(a,b),dz^2)$. Then 
	\[d_{hor}(z,z')=\abs{x-x'}.\]
	Furthermore, let $d=d_{hor}(z,b)$ for some zero $b$ of $\phi$ on the boundary of $\overline{\mathcal{Z}^h}$. Then $d$ only depends on $z$ and not on $b$. Finally, if $z^0$ and $z^1$ are the two lifts of $z$, then
	\begin{align}
		\abs{\primitive(z^0)-\primitive(z^1)}=2d.\label{superpotentialexact}
	\end{align}
\end{lemma}
Independence on $b$ comes from the fact that the standard saddle trajectory in $\mathcal{Z}^h$ has imaginary phase, and that there are no saddle trajectories on the boundary. Therefore, the two zeroes on the boundary lie over a vertical line segment. Given a point $z\in \tilde{C}-\snetwork(\pi/2)$, we can now order the two lifts of $z$ to $\scurve$ by the condition
\[W(z^+)>W(z^{-}).\]
Furthermore, we have the following corollary:
\begin{corollary}
	Let $z$ be a point in $\tilde{C}-\snetwork(\pi/2)$. Connect $z$ to a point $\tilde{z}$ on the wall $\gamma$ of $\snetwork(0)$ by a vertical trajectory. Then
	\[\primitive(z^{+})-\primitive(z^{-})=\primitive(\tilde{z}^{+})-\primitive({\tilde{z}}^{-})=2l>0,\]
	where $l$ is the $\phi$ distance of $\tilde{z}$ from the branch point end of the wall $\gamma$. $l$ does not depend on the choice of $\tilde{z}$.  
\end{corollary}

\subsubsection{Chamber deformations}\label{theregionCdeltaEsss}
We now construct the region $C(\delta;E)$ and the bridge region $\mathcal{V}(\delta;E)$.
\paragraph{Constructing $C(\delta;E)$.}

The conformal subdomain $C(\delta;E)$ is a deformation retract of $C(\delta;\infty)$ such that we have a bound on $W(z^+)-W(z^{-})$ for $z\in C(\delta;E)$. Again, since all the standard saddle trajectories are vertical, we can translate the vertically finite horizontal strip domains and half-plane domains as in Proposition \ref{domaindecomposition} so that all the branch points lie over $x=0$. Recall
\begin{align*}\mathcal{Z}(\delta;a,b)&=\mathcal{Z}(a+\epsilon(\delta),b-\epsilon(\delta))\subset \mathcal{Z}(a,b)\\
	\mathcal{H}(\delta)&=\mathcal{H}\cap \{\im{y}>\epsilon(\delta).
\end{align*}
Then for $E>0$, set
\begin{align}
	\mathcal{Z}(a,b;E)&:=\{z\in \mathcal{Z}(a,b):\ \abs{Re(z)}<E\}\\
	\mathcal{Z}(\delta;a,b;E)&:=\mathcal{Z}(a,b;E)\cap \mathcal{Z}(\delta;a,b)\\
	\mathcal{H}(E)&:=\{z\in \mathcal{H}:\ \abs{Re(z)},\abs{Im(z)}<E.\}\\
	\mathcal{H}(\delta;E)&:=\mathcal{H}(\delta)\cap \mathcal{H}(E).
\end{align}

\begin{definition}
	We define
	\[C(\delta;E):=\bigcup \mathcal{Z}(a,b;E) \cup \bigcup \mathcal{H}(\delta;E) \]
	where we take the union over all the horizontal chambers of $C$. \\
\end{definition}
Note that $\tilde{C}-\snetwork(0)$ deformation retracts to $C(\delta;E)$ and $\abs{\primitive(z^+)-\primitive(z^{-})}<2E$ by Proposition \ref{horizontallengthcontrol}.
\paragraph{Constructing $\mathcal{V}(\delta;E)$.}

We now construct the bridge region $\mathcal{V}(\delta;E)$. We start with a definition.
\begin{definition}\label{truncatedspectralnetworkdef}
	Suppose $\gamma:[0,\infty)\to \tilde{C}$ is a wall on the spectral network $\snetwork(0)$, arc-length parametrized with respect to $\phi$. Then for $T>0$, the $T$-truncated wall $\gamma$ is the restriction of $\gamma$ to the interval $[T,\infty)$. The $T$-truncated spectral network (or the truncated spectral network for short)  $\snetwork(0)_{T}$ is the union of the images of the $T$-truncated walls.
\end{definition}

The following definition will be useful:
\begin{definition}
	Let $\gamma$ be an open geodesic arc in $C^{\circ}$. Then we say that a neighbourhood $V$ of $\gamma$ is a \emph{vertical neighbourhood} if $V$ is traced out by open vertical segments that passes through $\gamma$. 
\end{definition}

Let $h_v\ll\min_{\mathcal{Z}(a,b)\subset C-\snetwork(0)} {\frac{\abs{b-a}}{2}}$ and let $\mathcal{V}(h_v)$ be the set of points in $C$ that are connected to points on $\snetwork(0)_{T}$ by a vertical geodesic of length less than $h_v$. Each component $\mathcal{V}$ of $\mathcal{V}(h_v)$ is a vertical neighbourhood of a unique truncated wall $\gamma\vert_{[T,\infty]}$ which we call the \textit{core} of $\mathcal{V}$. By taking $\delta\ll 1$, we can ensure that $\mathcal{V}$ intersects all the horizontal chambers that are adjacent to the core wall $\gamma$. 
%have that the component of $\mathcal{V}$ containing $\gamma$ as its core intersect all the $C(\delta;E)$-part of the chambers that contain $\gamma$ on their boundaries
%This is the union of small disjoint vertical thickening of the walls $\gamma\vert_{(T,\infty)}$ of $S(0)_{T}$ which we call the \textit{core} of $\mathcal{V}$.

If $z$ is a point on $\mathcal{V}$, then $\primitive(z^{+})-\primitive(z^{-})$ only depends on the core horizontal geodesic. For small enough $\delta>0$, $\mathcal{V}$ serves as a ``connecting bridge" between the connected components of $C(\delta;\infty)$ for $T=D((2+\eta)\delta)$. Here $D$ is some continuous function that only depends on $\phi$, and the choice of identifications $U_i\simeq D(r_i)\subset \mathbb{C}$ made in Section \ref{Desingularization of metric}. Note that $\mathcal{V}$ now lies outside $\snetwork(\pi/2)$ and $U(2+\eta)\delta$. 

\begin{figure}[t]
	\includegraphics[width=0.8\textwidth]{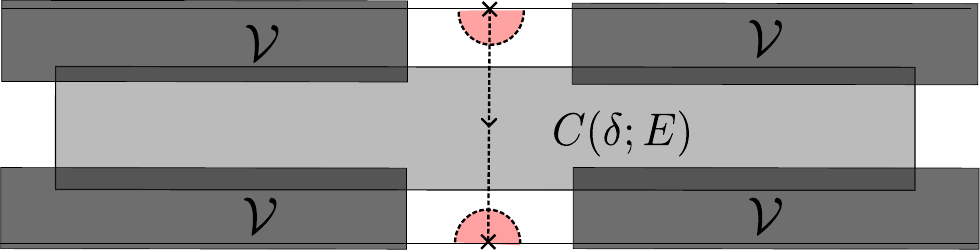}
	\centering
	\caption{A horizontal chamber $\mathcal{Z}^h$. The rectangular region is $C(\delta;E)\cap \mathcal{Z}^h$. $E$ is the width of the region $C(\delta;E)$. The half-disc region is $U((2+\eta)\delta)\cap \mathcal{Z}^h$. The dotted line segment denotes the vertical saddle trajectory in $\mathcal{Z}^h$. The dark gray regions indicate the components of $\mathcal{V}$.}
	\label{inideformfig}
\end{figure}
We summarize the discussion. 

\begin{proposition}\label{horizontaldomaindevision}
	Let $\delta\ll1$ be a small deformation parameter and let $E\gg1$ be an energy cut-off. Then there are precompact open conformal subdomains $C(\delta;E)$ contained outside $U(2+\eta)\delta$ and $\mathcal{V}(\delta;E)$ contained outside both $\snetwork(\pi/2)$ and $U(2+\eta)\delta$ with the following properties.
	\begin{itemize}
		\item There exists an $h(\delta)>0$ and an $\eta(\delta)>0$ such that if $\gamma$ is a generic horizontal trajectory passing through a point $z$ contained in $C(\delta;E)$ then $\gamma$ never enters $U((2+\eta)\delta)$. Furthermore, for $z\in C(\delta;E)$, we have \[\abs{\primitive(z^+)-\primitive(z^{-})}<2E.\] 
		\item Given a connected component $\mathcal{V}$ of $\mathcal{V}(\delta;E)$, there exists a unique wall $\gamma:(0,\infty)\to C$ called the \emph{core} of $\mathcal{V}$ and a truncated portion $\gamma\vert_{(\delta;\infty)}$ lying in $\mathcal{V}$, such that the component of $\mathcal{V}$ is given by some vertical thickening of $\gamma\vert_{(\delta,\infty)}$.  Furthermore, for $z\in \mathcal{V}$, we can order the lifts $z^{+},z^{-}$ of $z$ on $\scurve$ such that
		\[\primitive(z^+)-\primitive(z^-)>0.\]
		Finally, the connected component $\mathcal{V}$ overlaps with all the components of $\mathcal{C}(\delta;E)$ adjacent to its core wall $\gamma$. %In addition to this, the component of $\mathcal{V}$ containing $\gamma$ as its core intersect all the $C(\delta;E)$-part of the chambers that contain $\gamma$ on their boundaries.
	\end{itemize}
\end{proposition}
\begin{corollary}\label{uppertriangular}
	Let $J$ be a compatible almost complex structure on $T^{\ast}\tilde{C}$. Let $z\in \mathcal{V}(\delta;E)$. Then there are no non-constant BPS disks ending at $z$ going from $z^{+}$ to $z^{-}$. 
	
	%Let $\psi:X\to T^{\ast}\tilde{C}$ be a fibre-preserving diffeomorphism such that $J_{\psi}=\psi^{\ast}J$ is uniformly tame with respect to $\omega$, meaning, with respect to the $g_J$-metric, there exists a $C>0$ such that $g(X,X)\leq C\omega(X,J_{\psi}X)$. We also require that $\psi$ is the identity over $C(\delta;E)\cup \mathcal{V}$. Then there are no $J$-discs bound between $\psi^{-1}(\scurve)$ and $F_z$ going from $z^{-}$ to $z^{+}$. The same holds replacing $\scurve$ by a Hamiltonian equivalent exact Lagrangian. 
\end{corollary}
\begin{proof}
	%	The $J$-discs $u$ bound between $F_z$ and $\psi(\scurve)$ are the same as $J_{\psi}$-discs $\psi^{-1}\circ u$ bound between $(\scurve)$ and $F_z$. 
	From Stokes' theorem, and $\omega$-compatibility, 
	\[Area(u)= \int u^{\ast}\omega=\int u^{\ast}\omega=\primitive(z^{-})-\primitive(z^{+})<0.\]
	This is a contradiction since $u$ must have a positive $L^2$ norm. This finishes the proof.
\end{proof}

\begin{rem}
	Corollary \ref{uppertriangular} says nothing about the discs going from $z^{-}$ and $z^{+}$. However, it is important because otherwise we do not know a priori that the parallel transport across the connected components of $\mathcal{V}$ is \textit{upper}-triangular. 
\end{rem}
\section{Adiabatic degeneration}\label{Adia Degen}

We now study the adiabatic degeneration of $\epsilon$-BPS discs ending at $z$ as $\epsilon\to 0$. Recall that we had constructed a domain decomposition of $C$ and a deformation retract $C(\delta;E)$ of its horizontal chambers, with respect to a parameter $\delta>0$ and an energy cut-off $E\gg 1$. The region $C(\delta;E)$ has the property that the maximal horizontal trajectory passing through $z\in C(\delta;E)$ \textit{never} enters the region $U((2+\eta)\delta)$. In Section \ref{Flowlines}, we define the notion of holomorphic flow lines for (slight generalizations of) spectral curves and describe how they relate to $\phi$-trajectories. In Section \ref{the energy estimate and the diameter control}, we find an a priori energy and boundary length estimate for $\epsilon$-BPS discs ending at $z$ for $z\in C(\delta;E)$. In Section \ref{subsection:gradientestimate}, we establish some gradient estimates. 

In Section \ref{subsection domain subdivision}, we follow \cite[Section 5.2]{Morseflowtree} closely and introduce an $\epsilon$-uniformly finite number of punctures on the boundary of $\mathcal{Z}$ with the $\epsilon \scurve$-labelling for each $\epsilon$ to obtain a conformal domain $\stdm$ with the conformal structure defined as in Section \ref{moduliproblemss}. On this new conformal domain $\stdm$ of the map $u_\epsilon$, we construct a domain subdivision $D_{0}(\epsilon)\cup D_{1}(\epsilon)$ with a uniformly finite number of components with the following two properties.
\begin{itemize}
	\item The discs $u_\epsilon$ map $D_{0}(\epsilon)$ outside of $T^{\ast}U(2\delta)$ and map $D_{1}(\epsilon)$ into $T^{\ast}(U(2+\eta)\delta)$.
	\item The size of the derivatives of $u_\epsilon$ over $D_{0}(\epsilon)$ is $O(\epsilon)$. We show this by utilizing the gradient estimates in Section \ref{subsection:gradientestimate}.
\end{itemize}
In Section \ref{subsect convergence to gradient flows} and \ref{conv flowline subsec}, we study the limiting behaviour of $u_\epsilon$ restricted to $D_0(\epsilon)$ as $\epsilon\to 0$. We introduce uniformly finite number of auxiliary subdomains $W_{0}(\epsilon)$ of $D_0(\epsilon)$ so that $D_{0}(\epsilon)-W_{0}(\epsilon)$ consist of uniformly finitely many strip-like domains, and $u_\epsilon\vert_{W_{0}(\epsilon)}$ converges to points on $\tilde{C}$ (Lemma \ref{Wjttopoints}). The components of $D_{0}(\epsilon)-W_{0}(\epsilon)$ are classified as follows.
\begin{itemize}
	\item A \textit{$0$-special} domain is a strip domain that contains a horizontal boundary component with an $F_z$ label. By Lemma \ref{zerospecialdomainmappoint}, a $0$-special domain uniformly converges to $z$. 
	\item A non-$0$-special strip domain is either a vertex region or a flowline component. By Proposition \ref{vertex}, a vertex region is mapped very close to a point in $\tilde{C}$ after taking a subsequence.
	\item By Proposition \ref{non-vertex} a flowline component is mapped very close to a gradient flowline after taking a subsequence. Since $D_0(\epsilon)$ is mapped on the region where $g=g^{\phi}$, this flowline will turn out to be a unit-speed horizontal geodesic. 
\end{itemize}
We now give a "high-level" proof of \cref{maintheorem}. This is an argument that only applies to the rank $2$ situation. We will later give a different proof that provides a far stronger understanding of the adiabatic degeneration phenomenon with holomorphic cusp singularities, which will be needed for the higher-rank situation. Still, the argument below is sufficient for the main result of this paper.

\begin{proof}(Theorem \ref{maintheorem})
	
	We argue by contradiction. Let $u_\epsilon:\mathcal{Z}\to T^{\ast}\tilde{C}$ be a sequence of $\epsilon$-BPS discs ending at $z_\epsilon$ with $z_\epsilon\in C(\delta;E)$ such that $z_\epsilon\to z$ and $\epsilon\to 0$. Let $D_0(\epsilon), W_0(\epsilon)$ and $D_1(\epsilon)$  be as in Sections \ref{subsection domain subdivision}, \ref{subsect convergence to gradient flows} and \ref{conv flowline subsec}. 
	
	From the domain subdivision $\stdm=D_0(\epsilon)\cup D_1(\epsilon)$ and $D_0(\epsilon)-W_0(\epsilon)$, we obtain a graph $\Gamma$ as follows (we will later show that it is a tree). The vertices are the components of $W_0(\epsilon)$; the edges are the components of $D_0(\epsilon)-W_0(\epsilon)$ and $D_1(\epsilon)$, and two edges are connected if the corresponding components overlap or the components are adjacent over some common $W_0(\epsilon)$-components. Since there is a uniformly finite number of components, after taking a subsequence, we may assume that the graph $\Gamma$ is constant. 
	
	We claim that given any path $\mathcal{P}=e_1...e_n$ in $\Gamma$ such that $e_1,...,e_n$ correspond to the components in $D_0(\epsilon)-W_0(\epsilon)$, with $e_1$ corresponding to one of the $0$-special components, the image of the union must get mapped inside some arbitrarily small neighbourhood of the horizontal trajectory $\gamma$ passing through $z$. This is a consequence of Lemmas \ref{Wjttopoints}-\ref{non-vertex} which will be proven below. Indeed, all the flowline components must degenerate to a geodesic contained in $\gamma$, and the vertex and the $W_0(\epsilon)$-components must degenerate to points. So, the claim follows.
	
	Now, suppose for small enough $\epsilon$, there were some edges that got mapped into $T^{\ast}U((2+\eta)\delta)$. Consider a path $\mathcal{P}=e_1...e_n$ contained in $\Gamma$ such that the edges $e_1,...,e_{n-1}$ correspond to $\Theta_1,...,\Theta_{n-1}$ flowline components in $D_0(\epsilon)-W_0(\epsilon)$, and the edge $e_n$ corresponds to a component in $D_1(\epsilon)$. One such path must exist because $\tilde{T}$ is connected. Now, since $C(\delta;E)$ was chosen such that the image of the trajectory passing through $z$ lies entirely outside $T^{\ast}U((2+\eta)\delta)$, the component $\Theta_{n-1}$ cannot possibly overlap with $D_1(\epsilon)$. So the union of this component with $W_0(\epsilon)$ will map outside $T^{\ast}U((2+\eta)\delta)$, a contradiction. Therefore, the image must be contained in a neighbourhood of $\gamma$. In particular, such a neighbourhood can be chosen to be strictly contained in one of the horizontal chambers over which the covering is trivial. We conclude that no such $u_{\epsilon}$ could have existed for small enough $\epsilon$ since its boundary would not have closed up. This shows the non-existence of non-constant strips for $z\in C(\delta;E)$.
\end{proof}

%In Section \ref{ProofSubsect}, we will prove the main analytic Theorem \ref{maintheorem} by combining the results in Sections \ref{subsection:gradientestimate}-\ref{conv flowline subsec}. 
\subsection{Flow lines}\label{Flowlines}

We adapt the notion of flow lines introduced in \cite[Section 2]{Morseflowtree} to the holomorphic setting. Let $\tilde{C}$ be a Riemann surface and let $(T^{\ast}_{\mathbb{C}})^{1,0}\tilde{C}$ be the holomorphic cotangent bundle. Let $Y$ be a codimension $1$ holomorphic submanifold of $(T^{\ast}_{\mathbb{C}})^{1,0}\tilde{C}$ such that the holomorphic projection $\pi:Y\to \tilde{C}$ is a simple branched covering of $\tilde{C}$. Let $n$ be the degree of the branched covering $Y\to \tilde{C}$. Suppose $z\in \tilde{C}$ is a regular value of $\pi$. Then there exist an open neighbourhood $U$ of $z$, and locally defined holomorphic functions $f_1,...,f_{n}$ such that $Y\cap \pi^{-1}(U)$ locally reads as $\Gamma_{df_1}\sqcup...\sqcup\Gamma_{df_n}$ over $U$. Suppose now that $z$ is a branch point. Then the germ of $Y$ near the ramification point over $z$ is isomorphic to the germ of $(0,0)$ for the zero set of $\{(p^z)^2-z)=0\}$. The other smooth sheets of $Y$ over $z$ are given by the disjoint union $\Gamma_{dg_1}\sqcup...\sqcup \Gamma_{dg_{n-2}}$ for some locally defined holomorphic functions $g_1,...,g_{n-2}$. 

Equip $\tilde{C}$ with a K\"{a}hler metric $h=h_{z\bar{z}} dz d\bar{z}$. We can regard $h$ as an isomorphism $h:\overline{T^{1,0}_{\mathbb{C}}\tilde{C}}\to {T^{\ast}_{\mathbb{C}}}^{1,0}\tilde{C}$. Following \cite[Definition 14.5]{alginfraredpolytope} and \cite[Section 2.2.2]{Morseflowtree}

\begin{definition}
	Let $W$ be a locally defined holomorphic function on $\tilde{C}$. Then the \textit{$h$-gradient} of $W$ is defined by
	\begin{align}\label{h-gradient}
		\nabla_h(W)=\overline{h^{-1}(dW)}.
	\end{align} 
	Given a curve $z:[0,1]\to \tilde{C}$, a cotangent lift of $z$ is an ordered pair $\{z_1,z_2\}$ of lifts $z_i:[0,1]\to Y$ such that $z_1\neq z_2$, or their common value is a branch point of the projection $Y\to \tilde{C}$. 
	
	A curve $z:I\to \tilde{C}$, defined over an open interval $I$, with an ordered cotangent lift $(z_1,z_2)$, is called a \textit{holomorphic (Morse) flow line} of phase $\theta$ associated to $Y$ if the following equation is satisfied:
	\begin{align}
		\frac{dz}{dt}=-e^{+i\theta} \nabla_h(W_1-W_2),\label{BPSflowline}
	\end{align}
	whenever the local holomorphic functions $W_1$, $W_2$ associated to the lifts $z_1$ and $z_2$ are defined. 
\end{definition}

Now we restrict to the case $Y=\scurve$. Given a point $z\in {C}^{\circ}$, the quadratic differential $\phi$ determines local functions $W^{\pm}(z)$ such that
\begin{align*}
	&W^{\pm}(0)=0,\partial_z W^{\pm}=\pm \sqrt{\phi}(z).
\end{align*}
Furthermore, a GMN quadratic differential admits a conformal coordinate near a zero, which pulls back $\phi$ to $zdz^2$.

Recall that the GMN equation is the ODE
\begin{align}\label{GMNequation}
	Im(e^{-2i\theta}\phi(\gamma'))=0.
\end{align}

\begin{proposition}\label{prop:GMNholMorse}
	Holomorphic	flow lines associated to the spectral curve satisfy the GMN equation \eqref{GMNequation}. In particular, holomorphic Morse flow lines of phase $0$ lie on a horizontal trajectory. 
\end{proposition}
\begin{proof}
	This follows from \[e^{-2i\theta}\phi(\gamma')=e^{-2i\theta} \phi(z)\Big(\frac{dz}{dt}\Big)^2=e^{-2i\theta} \phi(z)\Big(h^{-1}e^{i\theta} \overline{\sqrt{\phi(z)}}\Big)^2= h^{-2} \phi(z) \overline{\phi(z)}.\]
\end{proof}

\subsection{The energy and boundary length estimate}\label{the energy estimate and the diameter control}

We prove the crucial energy and boundary length estimate. 

\begin{proposition}
	\label{uniformenergy}
	Suppose $u:\mathcal{Z}\to T^{\ast}\tilde{C}$ is an $\epsilon$-BPS disc ending at $z$ for $z\in C(\delta;E)$.  Then
	\[Area(u)\leq 2E\epsilon.\]
\end{proposition}
\begin{proof}
	Let $W$ be the primitive of $\lambda_{re}$ on $\scurve$. By Stokes' theorem,
	\begin{align}
		Area_{J}(u)=\int u^{\ast}\omega=\epsilon\abs{W(z^{0})-W(z^{1})}\leq 2E\epsilon
	\end{align}
	where the last inequality follows from Proposition \ref{horizontaldomaindevision}. 
	This finishes the proof. 
\end{proof}

We now need the estimate on the length of the boundary of $u_\epsilon$ on $\epsilon\scurve$ lying outside $T^{\ast}U(2\delta)\cap \epsilon\scurve$. We will use the following standard fact from measure theory. 
\begin{lemma}\label{lemma:monotoneabsolutecontinuitiy}
	Let $0\leq ...f_n\leq f_{n+1}...$ be a uniformly bounded monotone sequence of absolutely continuous functions defined on $[a,b]$ such that their derivatives are also non-negative almost everywhere. Then the pointwise limit $f=\lim_{n\to \infty}f_n$ is also absolutely continuous.
\end{lemma}
\begin{proof}
	By the monotone convergence theorem, 
	\[f(r)=\lim f_n(r)= \lim f_n(a)+\lim \int_0^r f'_n(r)dr=f(a)+\int_0^r \lim f'_n(r)dr,\]
	for all $r$. Now, apply the fundamental theorem of Lebesgue calculus.
\end{proof}

\begin{lem}
	There exists some $c=c(\delta,\eta)>0$ such that for all sufficiently small $0<\epsilon\leq 1$, the following holds. Let $u$ be an $\epsilon$-BPS disc ending at $z$. The length of $u({\partial\mathcal{Z}})$ outside $\big(T^{\ast}U(2\delta)\cup T^{\ast}B_{\frac{\eta}{2}}(z)\big)\cap \epsilon\scurve$ is bounded above by $c$.  \label{lengthestimate}
\end{lem}
\begin{figure}[t]
	\includegraphics[height=7cm]{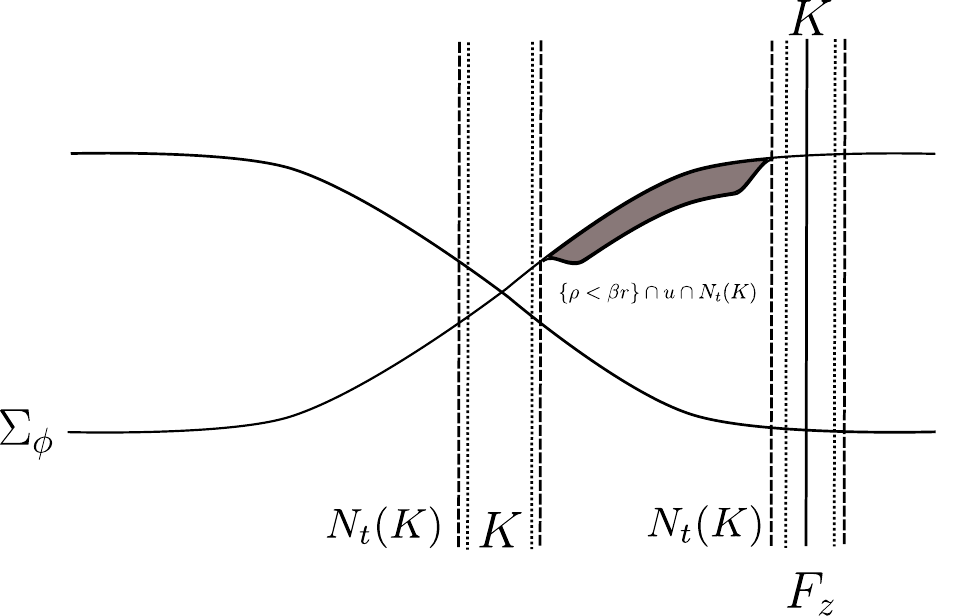}
	\centering
	\caption{The dashed lines indicate $\partial N_t(K)$ and the dotted lines indicate $\partial K$ in the proof of Lemma \ref{lengthestimate}. We need to slightly thicken $K$ to $N_t(K)$ in order to ensure that $u$ and $\partial u$ are transverse to the boundary of $N_t(K)$, so that the length of $u\cap \partial N_t(K)$ stay uniformly bounded with respect to $t$. Only its \textit{finiteness} will be important for the proof. The region inside is $K$. The region $\{\rho\leq \beta r\}\cap u\cap N_t(K)^c$ is colored. }
	\label{figure:compactsetremovalproof}
\end{figure}
\begin{proof}
	We split the proof of the lemma into several steps. Let $l$ be the length of $u({\partial\mathcal{Z}})$ on the region outside $\big(T^{\ast}U(2\delta)\cup T^{\ast}B_{\frac{\eta}{2}}(z)\big)\cap \epsilon\scurve$. Recall that we had chosen an $\eta>0$ such that over $U((2-\eta)\delta)^c$, $g_{\delta}^{\phi}=g^{\phi}$. Let $\rho$ be the distance function from $\epsilon\scurve$. The idea of the proof is to define a certain "truncated area function" $a(r)$ that measures the area of $u$ inside the dark grey domain illustrated in \cref{figure:compactsetremovalproof} and study its variations. 
	\paragraph{Step 1}
	Let $K=\big(T^{\ast}U((2-\frac{\eta}{2})\delta)\cup T^{\ast}B_{\frac{\eta}{3}}(z)\big)$. Outside $K$, the normal injectivity radius\footnote{The maximal radius of the normal disc bundle embedded via the radial exponential map.} of $\epsilon\scurve$ is $r'_0\epsilon$ for some $r'_0>0$ independent of $\epsilon$, which follows either the description of $\scurve$ in the $g^{\phi}$-flat $\int \sqrt{\phi}dz$-coordinate as parallel linear planes $\{p^z=\pm 1\}$, or the more general fact that the sheets of $\epsilon \scurve$ converge in the $C^{\infty}$ sense to the zero section outside the branch points. We take $r_0=\min(r_0',\frac{\eta}{8},\delta)$.
	
	Our goal is to approximate the length of $\partial u$ outside $K$ in terms of its energy. For technical reasons, we will need to slightly thicken $K$ so that that its boundary becomes transverse to $\partial u$ and $u$. We assure the reader that the constants in the bound won't depend on the particular choice of thickening. To do this, choose some small $0<\eta'\ll r_0$ such that the distance function from $K$ is smooth in $N_{\eta'}K-K$. This is fine since the normal bundle of $\partial K$ is necessarily trivial, and for small enough size of the tubular neighbourhood, the positive part of the normal bundle $N\partial K\simeq \partial K\times [-1,1]=\partial K\times [-1,0)\cup \partial K\times [0,1]$ is going to map strictly outside $\text{int}(K)$ under the exponential map. Then by Sard's theorem, we know that there exists some $0<t\ll \frac{1}{4}\eta'$ such that the maps $u$ and $\partial u$ are transverse to $\partial N_{t}(K)$. Furthermore, since $u^{-1}(\text{int}(K)^c)$ is compact, we see that $u^{-1}(\partial N_t(K))$ is a finite union of $1$-manifolds with boundary on $\partial S$. This condition is stable for small perturbations of $t$, and we can choose these perturbations so that the lengths of the preimages stay uniformly bounded. Therefore, there exists some $r_0\gg \mu>0$ such that $u$ and $\partial u$ are transverse to $\partial N_{t'}(K)$ for $t'\in (t-\mu,t+\mu)$ and $\text{Length}(u^{-1}(\partial N_{t'}(K)))$ is uniformly bounded with respect to the induced metric on $S$. Note that because the critical points of $u$ ($J$-holomorphicity implies that $du$ must entirely vanish at the critical points!) are discrete, and because the preimages of $\partial N_{t'}(K)$ avoid the critical points by regularity of $d(K,)$ at $t'$, this length is well-defined. 
	
	Here, we record the details since this standard fact will be an important part of the argument. Restrict $u$ to the open set $u^{-1}(N_{\eta'}K-K)$, then we can apply Sard's theorem for the smooth function $d(K,)$ for $u$ and $\partial u$. So for generic $t$, the level sets are $1$-manifolds (resp, $0$-manifolds) in $S$ (resp $\partial S$), and because $u^{-1}(\text{int}(K)^c)$ is compact, there are only finitely many points on $\partial S$ that map to $\partial N_t(K)$. Suppose $u^{-1}(\text{int}(K)^c)$ had infinitely many components. Choosing a point for each component, we see that we cannot have an accumulation point in the interior of $S$ because that will contradict the statement that $u^{-1}(\partial N_t(K))$ is a manifold. Therefore, the accumulation point must lie on the boundary, but there are finitely many points in $(\partial u)^{-1}(\partial N_t(K))$. In particular, this implies that $u^{-1}(\partial N_t(K))$ must be a $1$-manifold with boundary, and because there are finitely many boundary points, and there are no internal accumulation points, it must be compact. It is impossible for two different components of $u^{-1}(\partial N_t(K))$ to form a cusp at the boundary and share a tangency at their common boundary point because that will happen only when the gradient vector field vanishes on the boundary point (the two strands of the cusp will have normal vectors that are in \textit{opposite directions}), which violates the regularity of $d(K,)$ at $t$. Taking the gradient flow, we get a family of finitely many disjoint $1$-manifolds with boundary that we can uniformly bound their lengths.

	In the flat $\int \sqrt{\phi}$-coordinate on $U((2-\eta)\delta)^c$,  we have $J=J_{std}, \epsilon\scurve=\{y_1=\pm \epsilon, y_2=0\}$ and $\omega=\omega_{std}$. Translating the chosen sheet to $\{y_1=y_2=0\}$, we see the following.
	
	\begin{itemize}
		\item 
		In the neighbourhood of the boundary $\partial N_t K\cap \epsilon\scurve$, there are charts of radius $r_0 \epsilon$ contained in the complement of $\big(T^{\ast}U((2-\eta)\delta)\cup T^{\ast}B_{\frac{\eta}{4}}(z)\big)$ such that: $J=J_{std}$, $g=g_{std}$, $\epsilon\scurve=\mathbb{R}^2\subset \mathbb{C}^2$, and $N_tK=\overline{D}\times i\mathbb{R}^2$. Here $\mathbb{C}^2$ is given by the coordinates $(x_1,x_2,y_1,y_2)$ and $D$ is some open subdomain of $\mathbb{R}^2$.
		\item  Each point of $\epsilon\scurve\cap (N_{r_0}K)^c$ admits a chart of radius $\frac{r_0\epsilon}{2}$ in the complement of $\big(T^{\ast}U((2-\eta)\delta)\cup T^{\ast}B_{\frac{\eta}{4}}(z)\big)$ such that: $J=J_{std}$, $g=g_{std}$, and $\epsilon\scurve=\mathbb{R}^2\subset \mathbb{C}^2$.  
	\end{itemize}
	
	We choose a non-negative support function $\beta:T^{\ast}\tilde{C}\to\mathbb{R}_{\geq 0}$ such that $\beta=\beta(x_1,x_2)$ on each standard open chart chosen above, $\beta$ is positive on $(N_t K)^c$, $\beta$ vanishes on $N_tK$, and $\beta$ is equal to $1$ on $\big(T^{\ast}U(2\delta)\cup T^{\ast}B_{\frac{\eta}{2}}(z)\big)^c$. By making $t$ small, such a $\beta$ can be chosen in a way that $\xi=\sup \abs{\nabla \beta}$ depends only on $\delta$ and $\eta$. In the above local charts, $\rho=\abs{y}=\sqrt(y_1^2+y_2^2)$. Observe that $\rho$ is $J$-plurisubharmonic in the region $\{\rho\leq \beta r\}\cap K^c$
	\paragraph{Step 2}
	We now define our area and length functions. Choose a decreasing sequence of positive real numbers $\mu_n\to 0$ such that $\mu_n<\mu$. Write $K_n=N_{t+\mu_n}(K)$, and observe that $K_n\subset K_{n+1}\subset N_t(K)^c$. Let $r\leq r_0\epsilon$. We define the functions $a(r),a_n(r),a_n^{\beta}(r)$, and $l_n^{\beta}(r)$ as follows.
	
	\begin{align}
		&	a(r)=\int_{\{\rho\leq \beta r\}\cap u\cap {N_t(K)}^c} dA
		&&a_n(r)=\int_{\{\rho\leq \beta r\}\cap u\cap K_n^c} dA\\
		&a_n^{\beta}(r)=\int_{\{\rho\leq \beta r\}\cap u\cap K_n^c} \beta dA,&&l_n^{\beta}(r)=\int_{\{\rho=\beta r\}\cap u\cap K_n^c}\beta dl.
	\end{align}
	
	Observe that the functions $l_n^{\beta}(r)$ are only well-defined for almost every $r$ where $\{\rho\leq \beta r\}\cap u\cap int(K_n)^c$ is a $1$-manifold with boundary for all $n$. We are using Sard's theorem plus the fact that the countable union of measure $0$ sets have measure $0$. Observe that the integrands are all non-negative, and the domains of the integral become larger as $n$ increases. So, by the monotone convergence theorem, we can swap $\lim_{n\to \infty}$ with the integral, and we will freely do so unless it is unclear that the integrand is indeed non-negative.
	
	Our goal is to obtain a differential inequality of the form 
	\begin{align}\label{eq:differentialinequality1}
		ra'(r)\geq \frac{a(r)}{1+r \xi},
	\end{align}
	where $\xi$ is a constant that only depends on the geometry of $(T^{\ast}\tilde{C},J)$ and the relative geometry of $\epsilon \scurve$ outside $K$. 
	
	\paragraph{Step 3}
	As a first step, we show that $a(r)$ is absolutely continuous and obtain the following lower bound for its derivative in terms of $l^\beta_n$:
	\begin{align}\label{coarea2}
		a(r)\geq \int_0^{r} \lim_{n\to \infty}\frac{l_n^{\beta}(\tau)}{1+\tau\xi}d\tau\quad\text{and}\quad \frac{d}{dr}a(r)\geq \lim_{n\to \infty} \frac{l_n^{\beta}(r)}{1+r\xi} \text{ a.e} . 
	\end{align}
	Since $\beta>0$ on $K^c$, it follows that $\frac{\rho}{\beta}$ is Lipschitz on $N_{s}{K}^c$ for $s>t$. Hence, applying the coarea formula with respect to the induced metric, we get
	\begin{align}\label{coarea1}
		\int_{\{\rho\leq \beta r\}\cap u\cap N_s{K}^c} dA&=\int_{\{\rho\leq \beta r\}\cap u\cap N_s{K}^c} \frac{1}{\abs{\nabla (\rho/\beta)}}\cdot {\abs{\nabla (\rho/\beta)}}dA  \nonumber \\
		&=\int_{0}^r \int_{\{\rho=\beta \tau\}\cap u\cap N_s{K}^c}  \frac{1}{\abs{\nabla (\rho/\beta)}} dl d\tau.
	\end{align}
	By the monotone convergence theorem, we have \begin{align}\label{eq:absolutecontinuitiy} a(r)=\lim_{n\to \infty} a_n(r)&=\lim_{n\to \infty}\int_{0}^r \int_{\{\rho=\beta \tau\}\cap u\cap {K_n}^c}  \frac{1}{\abs{\nabla (\rho/\beta)}} dl d\tau\nonumber\\&=\int_{0}^r \lim_{n\to \infty}\int_{\{\rho=\beta \tau\}\cap u\cap {K_n}^c}  \frac{1}{\abs{\nabla (\rho/\beta)}} dl d\tau,
	\end{align}
	for all $r$	which proves the absolute continuity of $a(r)$.
	
	Now, by triangle inequality, $\abs{\nabla \rho}\leq 1$, and so on $\rho=\tau\beta$,
	\begin{align}\label{bound1}
		\abs{\nabla(\frac{\rho}{\beta})}\leq \frac{1}{\beta}(\abs{\nabla\rho}+\tau|\beta'|)\leq \frac{1+\tau \xi}{\beta}.
	\end{align}
	For \eqref{bound1}, we are taking the gradient by regarding it as a function on $T^{\ast}\tilde{C}$. However, the gradient of a smooth function restricted to a submanifold is given by the orthogonal projection of the original vector field. So, we get from \eqref{coarea1}-\eqref{bound1} that
	\begin{align}\label{coarea3}
		a_n(r)\geq \int_0^{r} \frac{l_n^{\beta}(\tau)}{1+\tau\xi}d\tau\quad\text{and}\quad \frac{d}{dr}a_n(r)\geq \frac{l_n^{\beta}(r)}{1+r\xi} \text{ a.e}. 
	\end{align}
	But $a(r)=\lim_{n\to \infty} a_n(r)$ and $a(r)$ is absolutely continuous almost everywhere. So by the monotone convergence theorem again, we obtain \eqref{coarea2}.
	\begin{align*}
		a(r)\geq \int_0^{r} \lim_{n\to \infty}\frac{l_n^{\beta}(\tau)}{1+\tau\xi}d\tau\quad\text{and}\quad \frac{d}{dr}a(r)\geq \lim_{n\to \infty} \frac{l_n^{\beta}(r)}{1+r\xi} \text{ a.e} . 
	\end{align*}
	\paragraph{Step 4}
	
	Having established the lower bound for the derivative of $a(r)$, we now show $\lim_{n\to \infty}rl^n_{\beta}(r)\geq a(r)$.
	For almost every $r$, we have the following:
	\begin{align}
		rl_n^{\beta}(r)&=\int_{\{\rho=\beta r\}\cap u\cap K_n^c} r\beta dl\geq \int_{\{\rho=\beta r\}\cap u\cap K_n^c} \frac{1}{2} \inner{\nabla h}{\nu}dl\label{eq:length-area1line1} \\
		&=\int_{\{\rho= \beta r\}\cap u\cap K_n^c} \frac{1}{2} d^c h=\frac{1}{2}\int_{\{\rho\leq \beta r\}\cap u\cap K_n^c} dd^c h-	
		\int_{\{\rho\leq \beta r\}\cap u\cap \partial K_n} d^c h. \label{eq:length-area1line2}
	\end{align}
	Here we have parametrized the oriented smooth curve $\{\rho=\beta r\}\cap u\cap K_n^c$ via $l(t)$ and took its unit normal $\nu(t)=-J(du\circ l(t))\abs{du\circ l'(t)}^{-1}$. Here we are using the fact that the critical points of $u$ do not lie on $l(t)$ for generic $r$. For the inequality in \eqref{eq:length-area1line1}, we use the standard fact that $\abs{\nabla \rho}=1$, and to pass from $\inner{\nabla h}{\nu}dl$ to $d^c h$, we use 
	\begin{align*}
		\inner{\nabla h}{\nu}dl=\inner{\nabla h}{\nu}\abs{du(l'(t))}l'(t)dt=-l'(t)dh\circ (J\circ du\circ \nu(t))=d^c h.\end{align*} The critical points of $u$ do not contribute because they are discrete, which is measure $0$. To arrive at the first equality in \eqref{eq:length-area1line2} is a bit more involved. We first use Stokes' theorem to obtain
	\begin{align}\label{eq:boundaryparts}
		\int_{\{\rho\leq \beta r\}\cap u\cap K_n^c} dd^c h&=
		\int_{\{\rho\leq \beta r\}\cap u\cap \partial K_n} d^c h+\int_{\{\rho\leq\beta r\}\cap \partial u\cap K_n^c}d^c h\nonumber \\
		&+\int_{\{\rho=\beta r\}\cap u\cap K_n^c}d^c h.
	\end{align}
	Here, $d^c h=0$ on $\{\rho\leq\beta r\}\cap \partial u$  since this set is contained in $L$. This explains why the second term in \eqref{eq:boundaryparts} vanishes. 
	
	We need to show that the term 	$\int_{\{\rho\leq \beta r\}\cap u\cap \partial K_n} d^c h$ converges to zero as $n\to \infty$. To see this, observe that the set is contained in $u^{-1}(\partial K_n)=u^{-1}(\partial N_{t+\mu_n}(K))$, whose length is uniformly bounded for all $n$. Therefore, we have
	\begin{align}
		\int_{\{\rho\leq \beta r\}\cap u\cap \partial K_n} d^c h&= \int_{\{\rho\leq \beta r\}\cap u\cap \partial K_n} \frac{1}{2} \inner{\nabla h}{\nu}dl\nonumber\\&\leq \sup_{\partial{N_{t+\mu_n}(K)}} \abs{\beta} r_0 \cdot \text{Length}(u^{-1}(\partial N_{t+\mu_n}(K))). 
	\end{align}
	Here $v$ means the unit normal vector field along $ u\cap \partial K_n$ defined similarly. We pass from $\inner{\nabla h}{\nu}dl$ to $d^c h$ by the same argument. To pass from the first line to the second line, we use that $\inner{\nabla h}{v}dl=\rho \inner{\nabla \rho}{v}dl\leq \rho dl\leq \beta r dl\leq r_0 \sup \beta dl$. 
	
	Since $\sup_{\partial_{N_{t+\mu_n}(K)}} \abs{\beta}$ converges to zero as $n\to \infty$, this term uniformly converges to zero as $n\to \infty$. The remaining term $\frac{1}{2}\int_{\{\rho\leq \beta r\}\cap u\cap K_n^c} dd^c h$ in $\eqref{eq:length-area1line2}$ converges to $\frac{1}{2}\int_{\{\rho\leq \beta r\}\cap u\cap {N_t(K)}^c}dd^ch$ since $dd^c h$ is $J$-plurisubharmonic that the integrand is necessarily non-negative. So we see that \begin{align}
		\lim_{n\to \infty} rl_n^{\beta}(r)\geq \frac{1}{2}\int_{\{\rho\leq \beta r\}\cap u\cap N_t(K)^c} dd^c h.
	\end{align}
	\paragraph{Step 5}
	We now derive the inequality and finish the proof. This is the only place where we use that the geometry is flat outside $K$. Since $h=\sum y_i^2$, we have $\frac{1}{2}dd^c h=\omega_{std}$. Therefore, combining \eqref{coarea2} and \eqref{eq:length-area1line1}--\eqref{eq:length-area1line2}, we see that
	\[ ra'(r)\geq \frac{a(r)}{1+r \xi}.\]
	Hence, we get the differential inequality
	\[\frac{d}{dr} \log \left(a(r)\cdot \frac{\xi r+1}{r}\right)\geq 0,\]
	which implies that the function
	\[r\to a(r)\cdot {\frac{\xi r+1}{r}}\]
	is non-decreasing.
	
	Now, if $r<\xi^{-1}$, then we get
	\begin{align}
		2\frac{a(r)}{r}\geq \lim_{s\to 0} \frac{Area(u;\big(T^{\ast}U(2\delta)\cup T^{\ast}B_{\frac{\eta}{2}}(z)\big)^c\cap \{\rho\leq s\})}{s}\Rightarrow 2\frac{a(r)}{r}\geq l.
	\end{align}
	Here we are using monotonicity of the function $a(r)\cdot {\frac{\xi r+1}{r}}$ and $(\xi r+1)<2$ for $r<\xi^{-1}$.  To see that the first inequality holds, observe that $(T^{\ast}U(2\delta)\cup T^{\ast}B_{\frac{\eta}{2}})^c\cap\{\rho\leq s\}$ is contained in the domain of integral of $a(s)$ which is $\{\rho\leq \beta s\}\cap u\cap N_t(K)^c$ since on $(T^{\ast}U(2\delta)\cup T^{\ast}B_{\frac{\eta}{2}})^c$, $\beta=1$ by definition. The second equality is the consequence of the right-hand side converging to the length outside $K$ as $s\to 0$. 
	
	The total energy of $u$ is bounded above by $<2E\epsilon$ by Proposition \ref{uniformenergy}. Setting $r=r_0 \epsilon$, it follows that we have
	\[Er_0^{-1}>l.\]
	Set $c=Er_0^{-1}$. This finishes the proof. 
\end{proof}

\begin{proposition}
	There exists a compact subset $K=K(\delta,\phi,E)\subset \tilde{C}$ containing $C(\delta;E)$ such that if $u$ is an $\epsilon$-BPS disc ending at $z$ for $z\in C(\delta;E)$, then $u$ lies in $P=D_1^{\ast}K^{\circ}$ for all small enough $\epsilon$. \label{boundaryestimate}
\end{proposition}
\begin{proof}
	The rescaled spectral curves $\epsilon\scurve$ for $0<\epsilon\leq 1$ lie inside the unit disc bundle $D_1^{\ast}\tilde{C}$. By the integrated maximum principle, the disc $u$ must lie in the unit disc bundle $D^{\ast}\tilde{C}$. Let $V$ be a sufficiently small neighbourhood of the poles of $\phi$ lying outside the region $C(\delta;E)\cup U(2\delta)$ such that $g\vert_V=g^{\phi}$. Let $K_1$ be the complement of $V$. For sufficiently small $\epsilon$, there exist some constants $G,H>0$ such that outside $T^{\ast}K_1$, the spectral curve $\epsilon\scurve$ has tameness constants (\ref{totallyrealgeombounddefini}) $C_{\epsilon \scurve}=G>0$ independent of $\epsilon$ and $r_{\epsilon \scurve}=H\epsilon$. Furthermore, by Proposition \ref{uniformenergy}, the total energy of $u$ is bounded above by $2E\epsilon$. So we can apply the proof of Proposition \ref{totaldomainestimate} to see that the discs cannot leave some neighbourhood of $D_1^{\ast}K$ by some precompact open subset $K$ containing $K_1$. Set $P=D_1^{\ast}K^{\circ}$.  
\end{proof}

\subsection{Gradient estimate} \label{subsection:gradientestimate}

We now follow {\cite[Section 5.1.3]{Morseflowtree} to prove the gradient estimates, which will be needed for the rest of the Section. We will only consider those fibres $F_z$ for $z\in C(\delta;E)$. From Proposition \ref{boundaryestimate}, we see that the discs of our interest are contained in a precompact neighbourhood $P$ of $C(\delta;E)$ in $T^{\ast}\tilde{C}$. For this reason, from now on, we only consider smooth functions that map into $P$.
	
	We start with the following gradient estimate:
	\begin{lemma}
		\label{Mainestimate} \cite[Lemma 4.3.1]{McduffSalamon}
		There exists some $\hbar>0$ such that for all $0<\epsilon\leq 1$, the following inequalities hold.
		\begin{itemize}
			\item
			If $u:A_\halfradius\to T^{\ast}\tilde{C}$ is a $J$-holomorphic disc, then 
			\[Area(u)<\hbar\Rightarrow \abs{du(0)}^2\leq \frac{8}{\pi {\halfradius}^2}\int_{A_\halfradius}\abs{du}^2.\]
			\item If $u:E_{2\halfradius}\to (T^{\ast}\tilde{C},\epsilon\scurve)$ is a $J$-holomorphic half-disc with $u(\partial E_{2\halfradius})\subset T^{\ast}U(2\delta)^c$ then
			\[Area(u)<\hbar\Rightarrow \sup_{E_\halfradius}\abs{du}^2\leq \frac{8}{\pi {\halfradius}^2}\int_{E_{2\halfradius}}\abs{du}^2.\]
			The same statement holds replacing $\epsilon\scurve$ with $F_z$ for $z\in C(\delta;E)$.
		\end{itemize}
	\end{lemma}	
	\begin{proof}
		The Sasaki almost complex structure $J$ already satisfies the conditions in \cite[Lemma 4.3.4]{McduffSalamon} that \textit{outside} $T^{\ast}U(2\delta)$, $\epsilon\scurve$ is totally geodesic, $JT(\epsilon\scurve)$ is orthogonal to $T(\epsilon\scurve)$ and $J$ is skew-adjoint with respect to $g^S$. Then by  \cite[Remark 4.3.2]{McduffSalamon}, there exists some $\hbar=\hbar(g^{\phi}_{\delta},\eta)>0$ such that the statement of Lemma  \ref{Mainestimate} holds. The same argument applies for $F_z,z\in C(\delta;E)$
	\end{proof}

	Fix now some $\epsilon>0$. Suppose we have an $\epsilon$-BPS disc $u$ ending at $z$ and suppose $u$ admits a subdomain $(E_\halfradius,\partial E_\halfradius)\subset (\mathcal{Z},\partial \mathcal{Z})$ such that $u\vert_{\partial E_{\halfradius}}$ maps outside $T^{\ast}U(2\delta)$. Suppose $\epsilon$ is small enough so that $2E\epsilon<\hbar$. By Proposition \ref{uniformenergy}, the total energy of $u$ is bounded above by $2E\epsilon$, so $u\vert_{E_\halfradius}$ satisfies the conditions in Lemma \ref{Mainestimate}. From this, we see that $\sup_{E_{\frac{1}{2}\halfradius}}\abs{du}$ must be bounded above by $\frac{8}{\halfradius}\sqrt{\frac{E}{\pi}}\epsilon^{1/2}$. 
	
	The following estimate by Ekholm improves the above $O(\epsilon^{1/2})$-estimate to an $O(\epsilon)$-estimate, given that the image of the disc lies in some disc bundle of radius $O(\epsilon)$. Observe that for $u=(q,p)$, we get $\abs{p}\leq \epsilon$ from the integrated maximum principle (see also \cite[Lemma 5.5]{Morseflowtree}). 
	\begin{lemma}{\cite[Lemma 5.7]{Morseflowtree}}\label{Otestimte} Fix some positive constants $\halfradius,C_1,C_2>0$, then for sufficiently small $\epsilon>0$, the following holds.
		\begin{itemize}
			\item Let $u:A_{8\halfradius}\to D_{C_1\epsilon}^{\ast} \tilde{C}$ be a $J$-holomorphic disc such that $Area(u)<C_2\epsilon$. Then there exists a constant $k(\halfradius,\delta,\eta,\phi,C_1,C_2)>0$ such that
			\begin{align}
				\sup_{A_{\halfradius}}	\abs{Du}\leq k\epsilon. \label{interiorOtestimate}
			\end{align}
			\item Let $u:E_{8\halfradius}\to D_{C_1\epsilon}^{\ast} \tilde{C}$ be a $J$-holomorphic half-disc such that $Area(u)<C_2\epsilon$ and $u(\partial E_{8\halfradius})$ lies on either $\epsilon\scurve$ outside $T^{\ast}U(2\delta)$, or on $F_z$ for $z\in C(\delta;E)$. Then there exists a constant $k(\halfradius,\delta,\eta,\phi,C_1,C_2)>0$ such that
			\begin{align}
				\sup_{E_{\halfradius} }\abs{Du}\leq k\epsilon.
				\label{boundaryOtestimate}
			\end{align}
		\end{itemize}
	\end{lemma}
	\begin{proof}
		Take a small enough $\epsilon$ so that $C_2 \epsilon<\hbar$.
		The idea is to show that the geometric energy of $u$ restricted to $E_{2\halfradius}(p)$ is actually of the size $O(\epsilon^2)$.  Hence applying Lemma \ref{Mainestimate}, we see that $\norm{Du}$ on $E_{\halfradius}$ is of size $O(\epsilon)$ which is precisely (\ref{boundaryOtestimate}) in the case $\partial E_{8\halfradius}$ maps to either $\epsilon\scurve$ or $F_z$. The proof is essentially the same as the proof of \cite[Lemma 5.7]{Morseflowtree}. 
		
		The case where $\partial E_{8\halfradius}\cap \partial \triangle_m$ maps to $\epsilon\scurve$ is unchanged. For the case the boundary maps to $F_z$, note that since the energy of $u$ is bounded above by $C_1 \epsilon$ on $E_{8\halfradius}$, the $C^1$ norm of $u_\epsilon$ on $E_{4\halfradius}$ is of $O(\epsilon^{1/2})$ by Lemma \ref{Mainestimate}. This implies that after taking a uniformly bounded conformal isomorphism $\Phi:E_{4\halfradius}\simeq E_1$, the image of $E_1$ under $u\circ \Phi^{-1}$ remains $O(\epsilon^{1/2})$-close to $z$. So for $\epsilon>0$ small, we can ensure that for $z\in C(\delta;E)$, the image of $u\circ \Phi^{-1}$ on $E_1$ maps inside $T^{\ast}C(\delta;E)$. 
		
		However, we have a local isometry $G:(T^{\ast}\tilde{C},F_z)\simeq (\mathbb{C}^2,i\mathbb{R}^2)$ sending $J$ to the standard almost complex structure on $\mathbb{C}^{2}$ (induced from taking the coordinate $\int 
		\sqrt{\phi}$ near $z$). Composing with this isometry, we get holomorphic maps $v=G\circ u\circ \Phi^{-1}: A_1\to \mathbb{C}^2$, with the imaginary part bounded above by $C_1\epsilon$. Furthermore, we can double along $i\mathbb{R}^{2}$ to get maps $\hat{v}:E_1\to \mathbb{C}^2$. Let $\tilde{v}=\epsilon^{-1}\hat{v}$ then the imaginary part of $\tilde{v}$ is bounded above by $C_1$.
		
		Let $F(z_1,z_2)=(e^{iz_1},e^{iz_2})$, then $f=F\circ \tilde{v}$ is holomorphic. Furthermore, the image is uniformly bounded since the imaginary part of $\tilde{v}$ is uniformly bounded, and so is the derivative of $F$ on the images of $\tilde{v}$. The $L^2$-norm of $Df$ on the disc of radius $1/2$ can be uniformly bounded by $\sup \abs{f}$ by Cauchy's inequality\footnote{If $f:A_1\to \mathbb{C}$ is holomorphic, and $z\in A_{1/2}$, then \[\abs{D^{n}f(z)}\leq n!\cdot 
			\frac{\abs{f}_{
					\infty,D}}{(1/4)^n}.\]}. Furthermore, since by chain rule  $Df=DF(\tilde{v})D\tilde{v}(z)$ and both the norms of $Dv$ and $DF(\tilde{v})$ are bounded on $A_{1/2}$, so is the norm of $D\tilde{v}$. So we see that there exists some $k_1>0$ such that $\norm{D\tilde{v}}_{L^2,A_{1/2}}\leq k_1.$
		Now $\norm{D{\tilde{v}}}^2_{L^2,A_{1/2}}=\epsilon^{-2}\norm{D\hat{v}}_{L^2,A_{1/2}}^2$ hence
		\[\norm{D(u\circ \Phi^{-1})}_{L^2,E_{1/2}}^2=\norm{Dv}^2_{L^2,E_{1/2}}=\frac{1}{4}\norm{D\hat{v}}^2_{L^2,A_{1/2}}\leq \frac{1}{4}k_1\epsilon^2,\]				
		where the first equality follows from $v=G\circ u\circ \Phi^{-1}$ and $G$ being an isometry, and the second equality follows from $\hat{v}$ being a doubling of $v$. Here, recall that we had composed with a conformal equivalence $E_{4\halfradius}\simeq E_1$. Hence, we have shown that the energy of $u$ is of size $O(\epsilon^2)$ on $E_{2\halfradius}$, just as claimed. \footnote{The actual proof is more or less the same, except that there are some diffeomorphisms involved in sending the local graph $\epsilon\cdot graph(dg)$ uniformly to $\mathbb{R}^n$ and comparing the almost complex structure with the standard almost complex structure $J_0$ on $\mathbb{C}^n$. The resulting function $f$ in \cite[Lemma 5.6]{Morseflowtree} is not fully holomorphic but very close to one.}
	\end{proof}
	\subsection{Domain subdivision}\label{subsection domain subdivision}
	\begin{figure}[t]
		\includegraphics[width=0.8\textwidth]{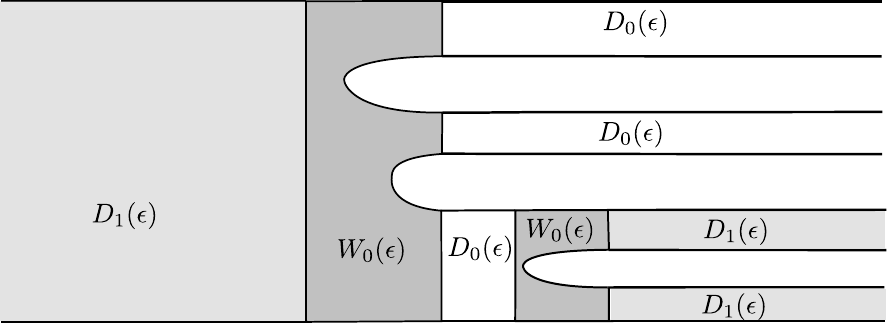}
		\centering
		\caption{$\stdm$ with the subdivisons $D_j(\epsilon)$. The white region is $D_{0}(\epsilon)-W_{0}(\epsilon)$. The dark gray region is $W_{0}(\epsilon)$. The light grey region is $D_1(\epsilon)$. The domains $D_1(\epsilon)$ and $D_0(\epsilon)$ may overlap. }
		\label{Domsubdivfig}
	\end{figure}
	To show  Theorem \ref{maintheorem}, we argue by contradiction. We assume that there exists a sequence of positive real numbers $\epsilon_n\to 0$ and a sequence of points $z_{\epsilon_n}\in C(\delta;E)$ converging to a point $z\in C(\delta;E)$ such that there exist $\epsilon_n$-BPS discs \[u_{\epsilon_n}:\mathcal{Z}\to T^{\ast}\tilde{C}\] ending at $z_{\epsilon_n}$. We will find a subsequence of $(z_{\epsilon_n},\epsilon_n)$ such that the corresponding discs lie strictly outside the desingularization region $T^{\ast}U((2+\eta)\delta)$.
	
	In order to do this, we modify the construction in \cite[Section 5.2]{Morseflowtree}, which will take the rest of Section \ref{subsection domain subdivision}. We introduce a uniformly finite number of punctures on the boundary of the domain $\mathcal{Z}$ of $u_\epsilon$ mapping to $\epsilon\scurve$. The new domain $\stdm$ admits a subdivison into domains $D_0(\epsilon)$ and $ D_{1}(\epsilon)$. Throughout this construction, we have to make choices for some auxiliary functions  $\delta_0(\epsilon)$. We now summarize their properties.
	
	\begin{itemize}
		\item $\partial D_j(\epsilon)-\partial \stdm$ consist of vertical line segments disjoint from the boundary minima.
		\item (Corollary \ref{D-1domainOt}) Over $D_{0}(\epsilon)$, we have
		\[\sup_{z\in D_{0}} \abs{Du_\epsilon(z)}\leq k\epsilon,\] 
		for some constant $k=k(\scurve,\delta)>0$.
		\item (Lemma \ref{escapeD-1}) The subdomain $D_{0}(\epsilon)$ is mapped outside of $T^{\ast}U((2+\frac{1}{2}\delta_{0})\delta)$ for some function $\delta_{0}=\delta_0(\epsilon)$ satisfying $0<\delta_{0}(\epsilon)<\frac{\eta}{10}$.
		\item The subdomain $D_1(\epsilon)$ is mapped inside $T^{\ast}U((2+\frac{9}{2}\delta_{0})\delta)$. 
	\end{itemize}
	\paragraph{Construction of domain subdivision}\label{Constructionofdomainsubdivisionparag}
	Now, we begin the construction. Fix a constant $0<\delta_{0}<\frac{\eta}{10}$ such that $u\vert \partial \mathcal{Z}$ is transverse to $\partial(T^{\ast}U((2+c\delta_{0})\delta))$ for $c\in \{1,2,3,4\}$.  Let $I\simeq \mathbb{R}$ be a boundary component of $\mathcal{Z}$. Let 
	\[b_1^c<b_2^c<....<b^c_{n(c)}, c=1,2,3,4\]
	be the points in $I$ such that $u(b_j^c)$ lies in the boundary $\partial(T^{\ast}U(2+c\delta_{0})\delta)$. Set $\infty=b^c_{k}$ for any $k>n(c)$. Let
	$B_i=\{b_1^c,....,b^c_{n(c)}\}$, $B=\cup B_i$, and $c(b):B\to \{1,2,3,4\}$ be the indexing function. For $2\leq c\leq 4$, we add a puncture at each $b_j^c$ and $b_{j+1}^c$ with the property that there exists some $b_k^{c-1}$ with $b_j^c<b_k^{c-1}<b_{j+1}^c$. In this case, we call $b_j^c$ an inward puncture and $b_{j+1}^c$ an outward puncture of type $c$.
	
	Intuitively, we are adding punctures every time the image of the boundary enters at the point $b_j^c$ and then leaves at the point $b_{j+1}^c$ the same ``level" $\partial(T^{\ast}U((2+c\delta_{0})\delta)$. Note also that at $b_{j+1}^c$, the image of the boundary points outward. Observe that all the punctures map into the neighbourhood $U((2+\eta)\delta)$.
	
	Taking out the punctures,  we arrive at a new domain $\stdm=\triangle_{2+m_1}$ with a holomorphic map $u:\stdm\to T^{\ast}\tilde{C}$. It can be readily checked that the boundary components $\tilde{I}$ of $\stdm$ separate into three different types:
	\begin{itemize}
		\item \textbf{Out}: $u(I)\subset T^{\ast}(\tilde{C}-U((2+{3}\delta_{0})\delta))$ (intervals between original punctures and $c=4$ punctures, or intervals between $c=4$ punctures and original punctures, or intervals between out $c=4$ punctures and in $c=4$ punctures.	)
		\item \textbf{In}: $u(I)\subset T^{\ast}U((2+2\delta_{0})\delta)$ (intervals between $b^2_i$ and $b^2_{i+1}$ such that there is some $b^2_i<b^1_j<b^2_{i+1}$)
		\item $\textbf{0}$: $u(I)\subset  T^{\ast}((U(2+{4}\delta_{0})\delta)-U((2+\delta_{0})\delta))$ (everything else)
	\end{itemize}
	\begin{figure}
		\centering
		\includegraphics[scale=0.8]{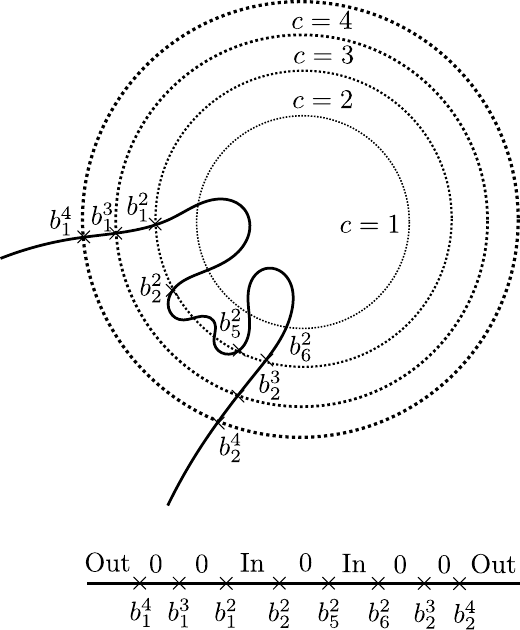}
		\caption{Illustration of $\partial u$ intersecting various $T^{\ast}U((2+c\delta_0)\delta)$. The punctures are introduced at $b_1^4,b_1^3,b_1^2,b_2^2,b^2_5,b^2_6,b_2^3$ and $b_2^4$. The complement of the punctures separate into intervals of type $0$, Out and In.}    \label{fig:puncturing}
	\end{figure}
	\begin{figure}
		\centering
		\includegraphics[scale=0.8]{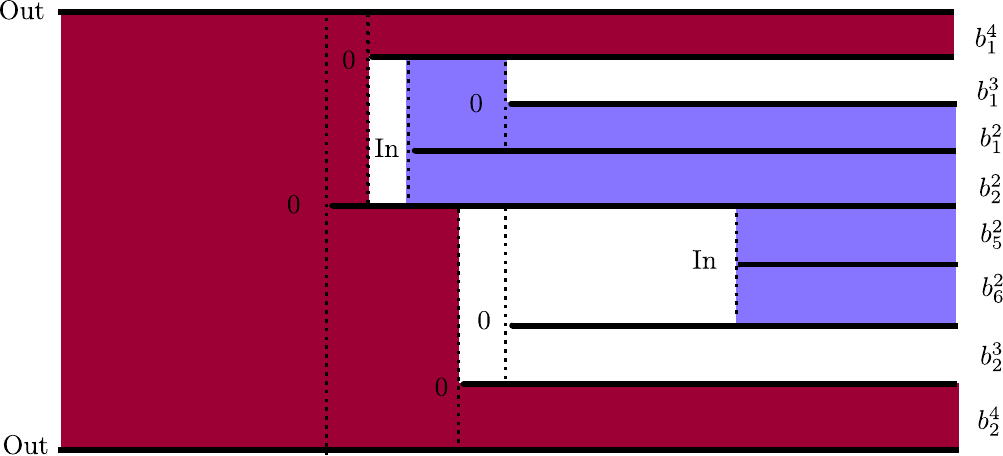}
		\caption{Removing the punctures, we arrive at our new domain $\stdm$. Here the slits are drawn as thickened black lines in the interior, and the vertical rays passing through the boundary minima are given by dotted lines. We colour the domains that are connected to \textbf{Out} boundary components via vertical rays \textit{red}. We colour the domains that are connected to \textit{In} boundary components via vertical rays \textit{blue}. The union of the blue domains give $D'$}
		\label{fig:domainsubdivision_1}
	\end{figure}
	See Figure \ref{fig:puncturing}. 
	
	One important property is that the number of added punctures is uniformly finite. 
	\begin{lemma}{\cite[Lemma 5.9]{Morseflowtree}}
		There exists a constant $R=R(\delta_0)>0$ such that the number $m_1$ of added punctures satisfies $m_1\leq R$.
	\end{lemma}
	\begin{proof}
		Each new puncture corresponds to a segment in $u(\partial \triangle_m)$ connecting the boundary $\partial T^{\ast}U((2+c\delta_{0})\delta)$ to $\partial(T^{\ast}U(2+{(c-1)}\delta_{0})\delta)$, $c=2,3,4$. The lengths of these segments admit a positive lower bound given by 
		\[\min_{c=2,3,4} d_{g^{\phi}_{\delta}}(\partial(U(2+c\delta_{0})\delta),\partial(U(2+(c-1)\delta_{0})\delta))\]
		by the definition of the Sasaki almost complex structure. Then, the proof follows from the a priori bound on the total length of the boundary components outside $T^{\ast}U(2\delta)$ (Lemma \ref{lengthestimate}).
	\end{proof}
	
	Note that a boundary component $I$ which maps into fibres $F_z$ for $z\in C(\delta;E)$ is automatically an \textbf{Out} boundary component, since $C(\delta;E)$ does not intersect $U((2+\eta)\delta)$ which contains $U((2+2\delta_0)\delta)$ since $\delta_0<\frac{\eta}{10}$. \\ 
	From now on, given a subset $S\subset \stdm$ and $l>0$, let $B_l(S)$ denote the $l$-neighbourhood of $S$ in $\stdm$. 
	For $1/4>d>0$ let $\Omega_d=\stdm-\bigcup_{I\subset \partial \stdm} B_d(I)$. Fix a small $\halfradius>0$ so that for $p\in \textbf{Out}\cup \textbf{0}$, the conformal domain $B_\halfradius(p)$ is uniformly conformally equivalent to $E_{{\halfradius}/{2}}(p)$ independent of $\epsilon$. Let $\Theta_{\halfradius}=\Omega_{\halfradius}\cup \bigcup_{I\in \textbf{Out}\cup \textbf{0}} B_{\halfradius}(I)$. 
	
	We have from Theorem \ref{Otestimte} that:
	
	\begin{lemma}{\cite[Lemma 5.10]{Morseflowtree}}
		There exist a constant $k>0$ such that if $\epsilon>0$ is sufficiently small then
		\[\sup_{z\in \Theta_{\halfradius}} \abs{Du}\leq k\epsilon.\] \label{D0estimate}
	\end{lemma}
	\begin{proof}
		By the integrated maximum principle, for $u_\epsilon=(q_\epsilon,p_\epsilon)$, $\abs{p_\epsilon}\leq \epsilon$ (see also \cite[Lemma 5.5]{Morseflowtree}). Suppose $\epsilon$ is small enough that $2E\epsilon<\hbar$. By Proposition \ref{uniformenergy}, the total energy of $u_\epsilon$ is bounded above by $2E\epsilon$, so $u\vert_{\Theta_{\halfradius}}$ satisfies the conditions in Lemma \ref{Otestimte}, after restricting to a smaller neighbourhood of radius $\epsilon$ on the boundary which is uniformly conformally equivalent to $E_{\halfradius/{2}}$. 
	\end{proof}
	\begin{figure}
		\centering
		\usetikzlibrary{decorations.pathmorphing}
\usetikzlibrary{decorations.markings}
% Set the overall layout of the tree
\tikzstyle{level 1}=[level distance=3.5cm, sibling distance=3.5cm]
\tikzstyle{level 2}=[level distance=3.5cm, sibling distance=2cm]

% Define styles for bags and leafs
%\tikzstyle{bag} = [text width=4em, text centered]
%\tikzstyle{end} = [circle, minimum width=3pt,fill, inner sep=0pt]
\tikzset{
   red/.style={draw=red,
   , postaction={decorate},
        decoration={markings}},
    blue/.style={draw=blue, 
    postaction={decorate},
        decoration={markings}}, 
            null/.style={draw=blue, 
    postaction={decorate},
        decoration={markings}}, 
}

% The sloped option gives rotated edge labels. Personally
% I find sloped labels a bit difficult to read. Remove the sloped options
% to get horizontal labels. 
\begin{tikzpicture}[
        thick,
        % Set the overall layout of the tree
        level/.style={level distance=2cm},
        level 2/.style={sibling distance=2.6cm},
        level 3/.style={sibling distance=1cm}
    ]
    \coordinate
        child[grow=left]{
           edge from parent[red]
        }
        % I have to insert a dummy child to get the tree to grow
        % correctly to the right.
        child[grow=right, level distance=0pt] {
       child{
       child{
       node {$b_2^4$}
       edge from parent [red]
       }
       child{
       child{
       node {$b_2^3$}
       edge from parent [black]
       }
       child{
       child{
       node {$b_6^2$}
       edge from parent [blue]
       }
       child{
       node {$b_5^2$}
       edge from parent [blue]
       }
       edge from parent [black]
       }
       edge from parent [black]
       }
       edge from parent [red]
       }
       child{
       child{
       child{
       node {$b_2^2$}
       edge from parent [blue]
       }
       child{
       child{
       node {$b_1^2$}
       edge from parent [blue]
       }
       child{
       node {$b_1^3$}
       edge from parent [black]
       }
       edge from parent [blue]
       }
       edge from parent [black]
       }
        child{
       node {$b_1^4$}
       edge from parent [red]
       }
       edge from parent [red]
       }
    };

\end{tikzpicture}
		\caption{Illustration of the tree $T$ obtained from $\stdm$.}
		\label{fig:tree}
	\end{figure}
	
	Recall how in Section \ref{moduliproblemss} we constructed the subdivision of $\stdm$ induced from the vertical rays passing through the boundary minima and obtained a stable $\npunct$-leaved tree $T$ from $\stdm$, whose edges were given by the connected components of the complement of the rays. Colour an edge/component blue if the corresponding component contains an \textbf{In} horizontal boundary segment. Consider the union of all the blue connected components. Observe that this union is equivalent to $D'\subset \stdm$, the union of all the vertical line segments in $\stdm$ connecting a point in a type-\textbf{In} boundary component to some other boundary point on $\partial\stdm$. The set $\partial D'-\partial \stdm$ is a collection of vertical line segments. In particular, the number of edges of $T$ is uniformly bounded with respect to $\npunct$. We state {\cite[Lemma 5.11]{Morseflowtree}} without proof since the proof is word-to-word the same. 
	
	\begin{lemma}\label{domainseparationlemma}
		For any $0<a<1$ and for sufficiently small $\epsilon>0$ we have $d(p,D')>\epsilon^{-a}$ for any point $p\in I$, where $I$ is a boundary segment of type \textbf{Out}. In particular, a vertical segment $l$ in $\partial D'-\partial \stdm$ has its end points either on the boundary minimum of a boundary segment of type \textbf{in} or on a boundary segment of type $\textbf{0}$. 
	\end{lemma}
	
	Now, colour a component of the vertical ray subdivision red if the component contains an \textbf{Out} horizontal boundary segment. The lemma states that the union of all the red components is separated away from $D'$ by distance at least $\epsilon^{-a}$.  Note that $\epsilon^{-a}$ grows much faster than $\log{\epsilon^{-1}}$. 
	
	Let $\log(\epsilon^{-1})\leq d \leq 2\log(\epsilon^{-1})$ be chosen such that  $\partial \overline{B_d(D')}-\partial \stdm$ and $\partial B_{{\frac{1}{2}}d}(D')-\partial \stdm$ consist of vertical line segments disjoint from all the boundary minimum. Intuitively, we are taking a horizontal thickening of the blue and the red subdomains by length $d$. Let $D_{0}=\stdm-B_{{\frac{1}{2}}d}(D')$ and $D_1=B_d(D')$. We see that if $p\in \partial D_{0}\cap \partial \stdm$, then $p$ lies in a boundary component of type \textbf{0} or type \textbf{Out}, and if $q\in \partial D_1-\partial \stdm$ then $q$ lies in a boundary segment of type \textbf{0} or \textbf{In}. Note also that by thickening the red subdomain (or the blue subdomain) to $D_{0}$ (or $D_{1}$), we have not increased the number of connected components of the red subdomain (or the blue subdomain). Hence, the number of the components of $D_{0}$ and $D_{1}$ are still uniformly bounded. See Figures \ref{fig:domainsubdivision_1}, \ref{fig:tree} and \ref{fig:subdivision2}.
	
	Furthermore,
	\begin{corollary}\label{cor:Oeestimate}
		\[\sup_{z\in D_{0}} \abs{Du(z)}\leq k\epsilon.\]
	\end{corollary}
	\label{D-1domainOt}
	\begin{proof}
		This follows from Lemma \ref{D0estimate}. 
	\end{proof}
	
	The following is adapted from \cite[Lemma 5.12]{Morseflowtree}. Again, the proof is word-to-word the same. 
	\begin{lemma}\label{lem:regionseparation}
		$u(D_1)\subset T^{\ast}U((2+\frac{9}{2}\delta_{0})\delta)$ and $u(D_{0})\subset T^{\ast}(\tilde{C}-U((2+\frac{1}{2}\delta_{0})\delta))$ for sufficiently small $\epsilon$. \label{escapeD-1}
	\end{lemma}
	The upshot is that since $\delta_0<\frac{\eta}{10}$, $D_1$ is mapped inside the region where small horizontal neighbourhoods of a horizontal trajectory passing through $z\in C(\delta;E)$ cannot enter, and $D_{0}$ is mapped into a region outside all the deformations, and where the metric coincides with $g^{\phi}$. Observe that since $\delta_0<\frac{\eta}{10}$, $u(D_1)$ is inside $U((2+\eta)\delta)$ and $u(D_0)$ is outside $U((2+\frac{1}{20}\eta)\delta)$.
	
	\begin{figure}
		\centering 
		\includegraphics[scale=0.8]{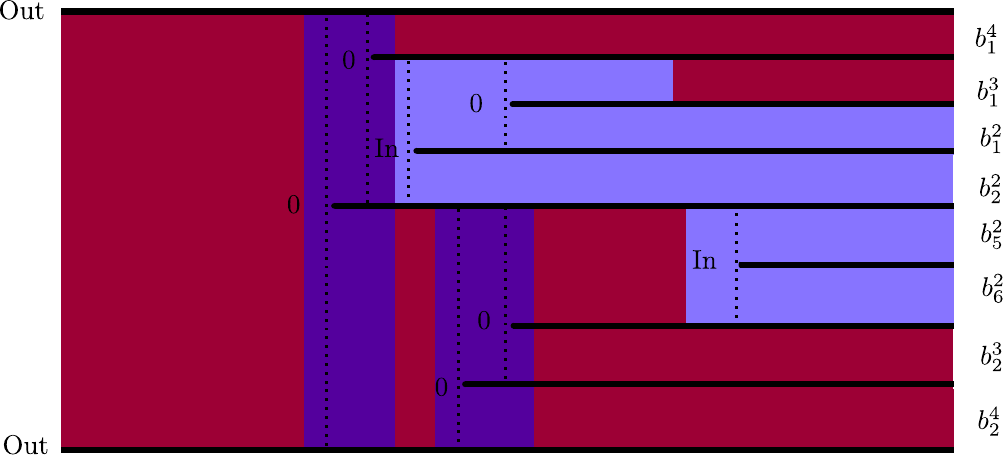}
		\caption{Here we illustrate an example of the thickenings $D_0$ (red+purple) and $D_1$. The purple domains ($W_0$) introduced in \cref{subsect convergence to gradient flows} are $O(-\log \epsilon^{-1})$-neighbourhoods of the boundary minima.}
		\label{fig:subdivision2}
	\end{figure}
	
	Now, given the sequence of $\epsilon_n$-BPS discs $u_{\epsilon_n}:\mathcal{Z}\to T^{\ast}\tilde{C}$, ending at $z_n$, we apply the same subdivision procedure by letting $\delta_{0}$ be a function of $\epsilon$ which is a very small variation of the constant $\delta_{0}$ such that $0<\delta_{0}(\epsilon)<\frac{1}{10}\eta$. By taking a subsequence, we assume that the number of added punctures,  the stable leaved tree $T$, and the red and the blue edges are, in fact, constant. The subdivision $\stdm=D_0(\epsilon) \cup D_1(\epsilon)$ can be chosen in a way that we have a decomposition of $T$ into open subsets $T_0$ and $T_1$ (with respect to the subspace topology after embedding $T$ into $\mathbb{R}^2$) that overlap in the interior of some edges. For example, in Figure \ref{fig:tree}, the edge terminating at $b_1^3$ contains a region where $D_0(\epsilon)$ and $D_1(\epsilon)$ overlap. By taking a subsequence, we may assume that the edges where they overlap become stationary as $\epsilon\to 0$.  Thus, we assume that the resulting tree $T$, and the subdivision $T=T_0\cup T_1$ (which we may regard as introducing a \textit{single} internal vertex to some black edges) are also constant. We have now finished the construction.% $\delta_0$-neighborhood of the caustics lie strictly inside $U(2\delta)$. 

	\subsection{Convergence to gradient flow lines}\label{subsect convergence to gradient flows}
	In this section, we introduce the auxiliary subdomain $W_0(\epsilon)$ of $D_0(\epsilon)$ such that the components of $W_0(\epsilon)-D_0(\epsilon)$ consist of strip-like domains, and they degenerate to solutions of gradient flow line equations. With respect to the tree $T$ above, we are essentially showing that the components corresponding to the edges of $T_0$ converge to the gradient flows. We also study limits of the auxiliary subdomains $W_0(\epsilon)$ and the $0$-special domains, which we recall to be the components of $D_0(\epsilon)-W_0(\epsilon)$ that contain a horizontal $F_z$-labelling. 
	
	We have the domain subdivision $\stdm= D_{0}(\epsilon)\cup D_1(\epsilon)$ as constructed in Section \ref{subsection domain subdivision}. Let $W_0(\epsilon)$ be a neighbourhood of the boundary minima of $D_0(\epsilon)$ such that:
	\begin{itemize}
		\item  the boundary $\partial W_0(\epsilon)$ consists of arcs in $\partial D_0(\epsilon)$ and vertical line segments,
		\item there is at least one boundary minimum on each component of $W_0(\epsilon)$.
	\end{itemize}
	For such $W_0(\epsilon)$, $D_0(\epsilon)-W_0(\epsilon)$ is a finite collection of strip regions. For an example, see \cref{fig:subdivision2}. For a connected component $W\subset W_0(\epsilon)$, we define the \textit{width} of $W$ as the maximum distance from a vertical line segment in the boundary of $W$ to a boundary minima inside $W$. We define the \textit{width} of the neighbourhood $W_0(\epsilon)$ to be the maximum of the width of the finitely many connected components of $W_0(\epsilon)$. 
	
	Given a vertical segment $l\simeq \{0\}\times[0,1]\subset D_0(\epsilon)-W_0(\epsilon)$ with $\partial l\subset \partial D_0(\epsilon)$, let $[-c,c]\times [0,1]\subset D_0(\epsilon)$ be a strip-like domain centred around $l$.  With $(s,t)\in [-c,c]\times [0,1]$, we write $u_\epsilon(s,t)=(q_\epsilon(s,t),p_\epsilon(s,t))$. Let $\epsilon b_\sigma$ denote the (1-form) section of the sheet that contains $u_\epsilon(0,\sigma)$ for $\sigma=0,1$. %By construction, we may assume that $c_\epsilon$ is of size $O(\log(\epsilon^{-1}))$. 
	
	We have the following estimate due to Ekholm {\cite[Lemma 5.13]{Morseflowtree}} which describes the degenerative behaviour of the components of $D_0(\epsilon)-W_0(\epsilon)$. 
	\begin{lemma}\label{flowlinedeg}
		For all sufficiently small $\epsilon>0$, we can find neighbourhoods $W_0(\epsilon)$ of the above type with width at most $2log(\epsilon^{-1})$, such that the following holds. Let $\Theta$ be a component of $D_0(\epsilon)-W_0(\epsilon)$ that is not a $0$-special domain. Then along any vertical line segment $l\subset \Theta$, we have
		\begin{align}
			&\frac{1}{\epsilon}\nabla_{t} p_\epsilon(0,t)-(b_1(q_\epsilon(0,0))-b_0(q_\epsilon(0,0)))=O(\epsilon)\\
			&\frac{1}{\epsilon}\nabla_s p_\epsilon(0,t)=O(\epsilon).
		\end{align}
	\end{lemma}
	Since $\Theta$ is an $\epsilon$-family of strip regions, we may write $\Theta=[-c_\epsilon,c_\epsilon]\times [0,1]$ for some $c_{\epsilon}>0$. Suppose $\Theta$ is a non-$0$-special component of $D_0(\epsilon)-W_0(\epsilon)$, then the rescaled strips $\tilde{u}_\epsilon=u_\epsilon(\epsilon^{-1}s,\epsilon^{-1}t)$ on $[-\epsilon c_\epsilon,\epsilon c_\epsilon]\times [0,\epsilon]$ locally converge to a gradient-flow equation determined by $b_{\sigma}$. Since the 1-form sections $b_{\sigma}$ are holomorphic, the resulting gradient flow equation is a \textit{holomorphic} gradient flow equation. 
	
	In Section \ref{conv flowline subsec}, we will show in \cref{vertex} that when $b_0=b_1$, the component degenerates to a point, and in \cref{non-vertex}, when $b_0\neq b_1$, the component degenerates to a flowline. Again, we will take a subsequence to ensure that the edges in $T$ whose corresponding components contain the vertical boundaries of $W_0(\epsilon)$ become constant for small enough $\epsilon$. 
	
	\begin{figure}[t]
		\includegraphics{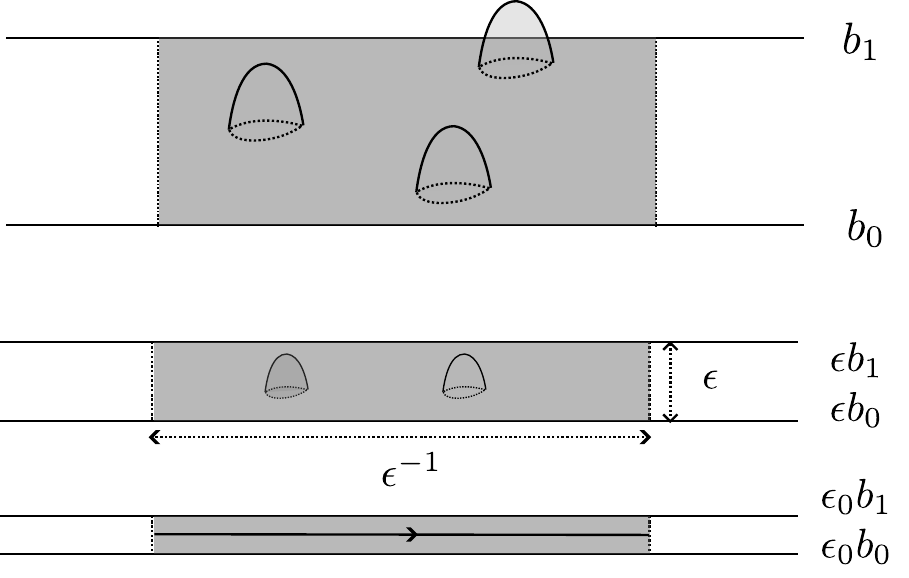}
		\centering
		\caption{Illustration of adiabatic degeneration. The shaded regions indicate $\Theta$. The small ``hills" in the diagram indicate possible deviations from the holomorphic strip, being a vertical strip along a holomorphic flow line trajectory. As $\epsilon\to 0$, the discs degenerate to vertical strips along a holomorphic flow line trajectory as the ``hills" disappear.}
		\label{Adiabatic Degeneration Fig}
	\end{figure}

	We now deal with the $0$-special domains.
	\begin{lemma}\label{zerospecialdomainmappoint}
		Let $\Theta\subset D_{0}(\epsilon)-W_{0}(\epsilon)$ be a $0$-special domain. Then $\lim_{\epsilon\to 0}d(u_\epsilon\vert_{\Theta},z)=0$. 
	\end{lemma}
	\begin{proof}
		By \cref{cor:Oeestimate}, the size of the derivative of $u_\epsilon$ on $\Theta$ is $O(\epsilon)$. Let $l$ be a vertical line segment in $\Theta$, then $l$ intersects a boundary component labelled $F_z$. So any point on $u_\epsilon(l)$ is $O(\epsilon)$-close to the point $z$. Since $\Theta\subset \stdm\subset \mathbb{R}\times [0,\npunct]$, the length of a vertical line segment in $\Theta$ is bounded above by $\npunct$. Hence as $\epsilon\to 0$, $u(l)\to z$. Furthermore, the speed of convergence is independent of $l$ since it only depends on $\npunct$ and the $O(\epsilon)$-estimate. This finishes the proof. 
	\end{proof}
	
	Furthermore, we show that the domains $W_0(\epsilon)$ are mapped very close to points in $\tilde{C}$.
	\begin{lemma}\label{Wjttopoints}
		Let $\Theta$ be a component of $W_{0}(\epsilon)$. Then, after taking a subsequence if necessary, there exists a point $z$ in $C$ such that $\lim d(u_\epsilon\vert_{\Theta},z)\to 0$. 
	\end{lemma}
	\begin{proof}
		The widths of the domains $W_{0}(\epsilon)$ are controlled by $2\log(\epsilon^{-1})$. From the $O(\epsilon)$-estimate, we see that diameters of the discs restricted to the domains $W_0(\epsilon)$ are of size $O(\epsilon\log{\epsilon^{-1}})$. Since $\epsilon\log(\epsilon^{-1})$ converges to $0$ as $\epsilon\to 0$, we see that, after taking a subsequence if necessary, that $u_\epsilon\vert_{\Theta}$ uniformly converges to a point in $\tilde{C}$.
	\end{proof} 
	
	\subsection{Convergence to horizontal geodesics}\label{conv flowline subsec}
	In this section, we investigate the convergence of strip-like domains in $D_0(\epsilon)-W_0(\epsilon)$. We separate the non $0$-special strip domains in $D_{0}(\epsilon)-W_{0}(\epsilon)$ into vertex and non-vertex regions. We show that the vertex regions converge to points, and the non-vertex regions converge to $\phi$-horizontal geodesics. In view of the tree $T$ and the subdivison $T=T_0\cup T_1$, we will be essentially showing that the edges of $T_0$ further split into those domains that converge to flowlines and those that converge to points. We modify the approach in \cite[Section 5.3]{Morseflowtree}. 
	
	Given $\Theta\subset D_{0}(\epsilon)-W_{0}(\epsilon)$ a strip region, since $\abs{Du_\epsilon}=O(\epsilon)$, possibly passing to a subsequence, we see that given a vertical line segment $l\subset \Theta$, $\pi(u_\epsilon(l))$ is contained in an $O(\epsilon)$-ball around a point. Since this ball lies outside $U(2\delta)$, we have two sheets of $\scurve$ over this point. Call the region a \textit{vertex} region if we can find a subsequence of $\epsilon$ converging to $0$ such that the endpoints of the vertical segments lie on the same sheet.  The following lemma shows that such horizontal subdomains degenerate to points. 
	\begin{lemma}
		Let $\Theta\subset D_{0}(\epsilon)-W_{0}(\epsilon)$ be a vertex region and let $\rho>0$. Then, after passing to a subsequence, there exists a point $p\in C(\delta;E)^c$ such that $u_\epsilon(\Theta)$ is contained in a $\rho$-ball around $p$ in $T^{\ast}\tilde{C}$.\label{vertex}
	\end{lemma}
	
	\begin{proof}
		We modify the proof of {\cite[Lemma 5.16]{Morseflowtree}}. After passing to a subsequence, we assume that $\pi(u_\epsilon(l))$ converges to some $p\in U((2+\eta)\delta)^c$. We know that for $0<\rho \ll \delta$, the ball $B_{\rho}(p)$ lies inside $U(2\delta)^c$. 
		
		Assume that for all small $\epsilon>0$, $u_\epsilon\vert_{\Theta}$ does not stay entirely in an $\rho$-ball around $p$. Then there exists a sequence of points $q_\epsilon\in \Theta$ such that $u_\epsilon(q_\epsilon)$ lies strictly outside the $\rho$-ball around $p$ for small enough $\epsilon$. By taking a subsequence, we may assume that $u_\epsilon(q_\epsilon)$ converges to a point $q$.  By the $O(\epsilon)$-estimate on the derivative of $u_\epsilon$ restricted to $D_0(\epsilon)$, the vertical line segment passing through $q_\epsilon$ must map outside the $\frac{1}{2}\rho$-ball around $p$ and must also uniformly converge to the point $q$. 
		
		Let $\Theta_\epsilon$ be the strip region inside $\Theta$ bound by the vertical line segment $l$ and the vertical line segment passing through $q_\epsilon$. We claim that there exists a $k>0$ such that $Area(u_\epsilon;\Theta_\epsilon)<k\epsilon^2$. Suppose for now this is true, and consider the disjoint union of balls
		\[B=B_{\rho/4}(q)\cup B_{\rho/4}(p)\]
		in $T^{\ast}\tilde{C}$. Again, by the $O(\epsilon)$-estimate and the convergence $u_\epsilon(q_\epsilon)\to q$, the boundary of $u_\epsilon\vert_{\Theta_\epsilon}$ is contained in $B \cup \epsilon\scurve$ for small enough $\epsilon$. In particular, since $\Theta$ maps outside $T^{\ast}U(2\delta)$, $u_\epsilon$ maps the horizontal boundary segments of $\Theta_\epsilon$ to the same sheets of $\scurve$. Since each sheet of $\epsilon\scurve$ over $\pi(u(\Theta))$ is uniformly geometrically bounded, we see that the curve $u_\epsilon$ restricted to $\Theta_\epsilon$ cannot leave some $O(\epsilon)$-small neighbourhood of $B$ by the boundary estimate (Proposition \ref{totaldomainestimate}). For small enough $\epsilon$, such a neighbourhood of $B$ is disconnected, but the image of $u_\epsilon$ over $\Theta_\epsilon$ must be connected, a contradiction.

		Let $\Theta_\epsilon=[a_\epsilon,b_\epsilon]\times [0,1]$. By connectedness, we see that the curves $[a_{\epsilon},b_{\epsilon}]\times \{i\},i=0,1$ must map to the same sheets of $\epsilon\scurve$ over the image of $\Theta_{\epsilon}$. Fix any primitive $f_{\epsilon}$ of $\canliouvile$ on $\scurve$ over this domain. Observe that $\abs{\nabla f_{\epsilon}}$ is bounded above by some constant $C_2>0$. By Stokes' theorem, we have	
		\begin{align}
			Area(u_{t})&= \int_{\partial([a_\epsilon,b_\epsilon]\times[0,1])} p dq \nonumber\\
			&=-\int_{\{a_\epsilon\}\times[0,1]} pdq+ \int_{\{b_\epsilon\}\times[0,1]} pdq \nonumber \\
			&\pm \epsilon(f_{\epsilon}(u_\epsilon(a_\epsilon,1))-f_{\epsilon}(u_\epsilon(b_\epsilon,1))-f_{\epsilon}(u_\epsilon(a_\epsilon,0))+f_{\epsilon}(u_\epsilon(b_\epsilon,0))). \label{eq:Oe2term}
		\end{align}

		Now since $\abs{p}=O(\epsilon)$, and $\abs{Du}=O(\epsilon)$ over $D_0(\epsilon)$, the first two terms in \eqref{eq:Oe2term} are $O(\epsilon^2)$. For the four terms after that, note that 
		\[\epsilon\abs{f_{\epsilon}(u_\epsilon(a_\epsilon,1))-f_{\epsilon}(u_\epsilon(b_\epsilon,1))}\leq \epsilon\sup \abs{\nabla f_{\epsilon}}\sup \abs{Du_{\epsilon}}.\]
		We have a similar inequality for the term  coming from $u_\epsilon(b_\epsilon,0)$ and $u_\epsilon(b_\epsilon,1)$ in the last four terms of \eqref{eq:Oe2term}. But since $\abs{Du_\epsilon}=O(\epsilon)$ and $\abs{\nabla f_{\epsilon}}$ is uniformly bounded, the area on the $\Theta_\epsilon$ region must be of $O(\epsilon^2)$. 
	\end{proof}
	
	We have the following statement on the adiabatic degeneration of non-vertex strip regions.
	
	\begin{lemma}\label{non-vertex}
		Let $\Theta(\epsilon)\subset D_{0}(\epsilon)-W_{0}(\epsilon)$ be a non $0$-special, non-vertex strip region and let $\epsilon>0$. Then, after passing to a subsequence, there exists a horizontal geodesic $\gamma$ contained in $U(2\delta)^c$ such that the image of $u_\epsilon(\Theta)$ is contained in an $\epsilon$-neighborhood of $\gamma$. 
	\end{lemma}
	\begin{proof}
		The proof is a very small modification of {\cite[Lemma  5.17]{Morseflowtree}}. We split into two cases. First, assume that $\Theta=[-c_\epsilon,c_\epsilon]\times [0,1]$ is such that $\epsilon  c_\epsilon\leq K$ for some $K$. Write $u_\epsilon=(q_\epsilon,p_\epsilon)$. Since $u_\epsilon$ is $J$-holomorphic, we have
		\begin{align*}
			&\frac{\partial q_\epsilon}{\partial s}+g^{-1}(\nabla_{t}p_\epsilon)=0,
			&\frac{\partial q_\epsilon}{\partial t}-g^{-1}(\nabla_s p_\epsilon)=0.
		\end{align*}
		Then consider the rescaling $\tilde{u_\epsilon}=(\tilde{q}_\epsilon,\tilde{p}_\epsilon)=u_\epsilon(\epsilon^{-1}s,\epsilon^{-1}t)$ defined on $[-\epsilon c_\epsilon,\epsilon c_\epsilon]\times [0,\epsilon]$. We see from Lemma  \ref{flowlinedeg} that 
		\begin{align}
			&\frac{\partial \tilde{q}_\epsilon}{\partial s}-Y=O(\epsilon),
			&\frac{\partial \tilde{q}_\epsilon}{\partial t}=O(\epsilon).
		\end{align}
		Here, $Y$ is the local gradient difference determined by the two local sheets of $\epsilon\scurve$. Pass to a subsequence for which both the rescaled lengths $\epsilon c_\epsilon$ and the points $u_\epsilon(-c_\epsilon,0)$ converge (recall that all the discs map into $P$; see Lemma \ref{boundaryestimate}). We see that the image of the strip region must lie in a small neighbourhood of a flow line. 
		
		We next consider the case where $\epsilon c_\epsilon$ is unbounded. We will show that such cases cannot happen. In order to do this, we first need to take some preliminary estimates on the length of the flowlines obtained from domains $[-\epsilon c_\epsilon,\epsilon c_{\epsilon}]\times [0,1]$. Indeed, observe that the length of the flowline is bounded below by 
		\[\lim_{\epsilon\to 0}{\epsilon c_{\epsilon}}\cdot {\inf_{C-U(2\delta)} \norm{\epsilon^{-1}Y}}_g.\]
		Here $\norm{\epsilon^{-1}Y}$ is the local gradient difference determined by the distinct local sheets of $\scurve$. However, this admits a positive lower bound on the complement of $U(2\delta)$. So we see that the length of the flowline is bounded below in terms of $\lim \epsilon c_{\epsilon}$, and some constants that depend only on $\delta$, $g$ and the geometry of $\scurve$. Furthermore, since these flow lines are contained outside $U(2\delta)$ where $g=g^{\phi}$, we see from \cref{prop:GMNholMorse} that they are, in fact, horizontal geodesics.  they must correspond to a unit speed horizontal geodesic by \cref{prop:GMNholMorse}.
		
		We now proceed with the rest of the proof. Since the strips are contained in $D_0(\epsilon)$, they map outside $U(2\delta)$. Therefore, by the previous argument there exists some $t>0$ such that a gradient flowline contained in this region must travel at least the distance $t>0$ within the flow-time $[0,1]$. Recall that the length of the boundary is bounded above by some $L$, and let $N>Lt^{-1}$.
		
		By the previous argument, we can choose some subsequence such that on $[-c_\epsilon,-c_{\epsilon}+N\epsilon^{-1}]\times [0,1]$, $\tilde{u}_{\epsilon}$ converges to a horizontal geodesic of length $Nt>L$. Indeed, since $\epsilon c_{\epsilon}$ diverges, $N\epsilon^{-1}<c_\epsilon$ for small enough $\epsilon$. This is a contradiction since the boundary length must be bounded above by $L$. So we see that we cannot have $\epsilon c_{\epsilon}$ unbounded. This finishes the proof. \end{proof}

	With Lemmas \ref{vertex} and \ref{non-vertex} established, we now modify the tree $T$ as follows. By
	definition, the boundaries of $W_0(\epsilon)$ consist of horizontal arcs in $\partial D_0(\epsilon)$ and vertical
	line segments. We introduce a vertex for edges that correspond to the regions where $D_0$-domains and $D_1$ domains overlap (For instance, the vertex $s$ in \ref{fig:tree2}).  Then we collapse each neighbourhood
	of the vertices of $T$ corresponding to $W_0(\epsilon)$, and obtain a semi-stable tree $\tilde{T}$. Now, among the edges of $\tilde{T}$, we have the edges that are in one-to-one correspondence with the components of
	$D_0(\epsilon)-W_0(\epsilon)$. We will call the edges that correspond to the flowline subdomain
	components of $D_0(\epsilon)-W_0(\epsilon)$ flowline edges. By taking a subsequence, we may
	assume that the resulting semi-stable tree $\tilde{T}$ is also constant. It is absolutely crucial
	that the tree $\tilde{T}$ by the nature of construction is a rooted semi-stable tree. See Figure \ref{fig:tree2} for an example. 
	\begin{figure}
		\usetikzlibrary{decorations.pathmorphing}
\usetikzlibrary{decorations.markings}
% Set the overall layout of the tree
\tikzstyle{level 1}=[level distance=3.5cm, sibling distance=3.5cm]
\tikzstyle{level 2}=[level distance=3.5cm, sibling distance=2cm]

% Define styles for bags and leafs
%\tikzstyle{bag} = [text width=4em, text centered]
%\tikzstyle{end} = [circle, minimum width=3pt,fill, inner sep=0pt]
\tikzset{
   red/.style={draw=red,
   , postaction={decorate},
        decoration={markings}},
    blue/.style={draw=blue, 
    postaction={decorate},
        decoration={markings}}, 
            null/.style={draw=blue, 
    postaction={decorate},
        decoration={markings}}, 
          treenode/.style = {align=center, inner sep=0pt, text centered,
    font=\sffamily},
  arn_n/.style = {treenode, circle, black, font=\sffamily\bfseries, draw=black,
    fill=white, text width=0.8em}
}

% The sloped option gives rotated edge labels. Personally
% I find sloped labels a bit difficult to read. Remove the sloped options
% to get horizontal labels. 
\begin{tikzpicture}[
        thick,
        % Set the overall layout of the tree
        level/.style={level distance=2cm},
        level 2/.style={sibling distance=2.6cm},
        level 3/.style={sibling distance=1cm}
    ]
    \coordinate
        child[grow=left]{
           edge from parent[red]
           node [above] {$A$}
        }
        % I have to insert a dummy child to get the tree to grow
        % correctly to the right.
        child[grow=right, level distance=0pt] {
       child{
       child{
       child{
       node {$b_2^4$}
       edge from parent [red]
       }
       child{
       node {$b_2^3$}
       edge from parent [red]
       }
       child{
       child{
       node {$b_6^2$}
       edge from parent [blue]
       }
       child{
       node {$b_5^2$}
       edge from parent [blue]
       }
       edge from parent [red] %{$3$}
       node [above] {$D$}
       }
        node [arn_n] {$v$}
       edge from parent [red] 
         node (n2) [above] {$C$}
         node (n3) [below] {}
       }
       node [arn_n] {$v$}
       edge from parent [red] %{$1$}
       }
       child{
       child{
       child{
       node {$b_2^2$}
       edge from parent [blue]
       }
       child{
       child{
       node {$b_1^2$}
       edge from parent [blue]
       }
       child{
       child{
          node {$b_1^3$}
          edge from parent [red]
       }
       node [arn_n] {$s$}
       edge from parent [blue]
       }
       edge from parent [blue]
       }
       edge from parent [blue]
       }
        child{
       node {$b_1^4$}
       edge from parent [red]
       }
       node [arn_n] {$v$}
       edge from parent [red] 
       node [above] {$B$}
       }
    };
\draw[dashed, shorten <=-2mm, shorten >=-2mm] 
    ([xshift=5mm,yshift=0mm] n2.north)       -- ([xshift=5mm, yshift=-1mm] n3.south);
\end{tikzpicture}
		\caption{The collapsed tree $\tilde{T}$ with the nodes $v$ indicating the vertices corresponding to the components of $W_0(\epsilon)$. This \textit{semi-stable} tree is obtained from the tree in \cref{fig:tree} by adding a uniformly finite number of internal vertices and/or collapsing subtrees. An example of a vertical cut passes through the edge $C$. Observe that the vertex $s$ indicates the region where $D_0$ and $D_1$ overlap.}
		\label{fig:tree2}
	\end{figure}
	\subsection{Proof of Theorem \ref{maintheorem}}\label{ProofSubsect}
	We now proceed with the proof of Theorem \ref{maintheorem}. The idea is to show that the degenerated flowtree of $u_{\epsilon}$ \textit{start off}} from initial trivalent rays.

\begin{proof}
	The stronger argument below will show the non-existence for $z\in \snetwork^c$. It also shows non-existence of strips that travel from $z^{+}$ to $z^{-}$ for $z\in \snetwork$. As usual, we argue by contradiction. Let $u_\epsilon:\mathcal{Z}\to T^{\ast}\tilde{C}$ be a sequence of $\epsilon$-BPS discs ending at $z_\epsilon$ with $z_\epsilon\in C(\delta;E)$ such that $z_\epsilon\to z$ and $\epsilon\to 0$. Let $D_0(\epsilon), W_0(\epsilon)$, $D_1(\epsilon)$, $T$ and $\tilde{T}$ be as in Sections \ref{subsection domain subdivision}-\ref{conv flowline subsec}. 
	\paragraph{Step 0}\label{paragraph:step0}
	
	By construction, the additional punctures we have introduced all get mapped into $T^{\ast}(U(2+\eta)\delta)$. Cut $\stdm$ along the vertical rays that pass through the boundary minima and connect to a horizontal boundary component with a cotangent fibre labelling. We know that the components of this subdivision with a cotangent fibre labelling all degenerate to points as $\epsilon\to 0$. As a result, we arrive at subdomains $\bar{\triangle}_{m_1},...,\bar{\triangle}_{m_n}$ \textit{truncated} at the $s=-\infty$ end, that is, it is the intersection of $\triangle_{m_i}$ with $[0,m_i-1]\times [a,\infty]$ for some $a$ lying strictly to the left of the  slit boundary components. Observe that the punctures of $\bar{\triangle}_{m_i}$ are the artificial punctures we had added, so the non-vertical boundaries of $\bar{\triangle}_{m_i}$ must close up in $\epsilon\Sigma_{\phi}$. Furthermore, the map $u_{\epsilon}\vert_{\triangle_{m_i}}$ necessarily degenerate to $z$ near the $\log(\epsilon^{-1})$-neighbourhood of the vertical cuts.
	
	The tree $\tilde{T}$ by construction is a rooted, directed, semi-stable tree. Therefore, starting from the unique root, we obtain a partial ordering on its vertices. We first claim that the out-going $D_0$-edges in $\tilde{T}$ must correspond to vertex components (Lemma \ref{vertex}). Indeed, such components correspond to right-infinite strips of the form $[M,\infty)\times [0,1]$ for some large $M\gg 0$, and by Lemma \ref{non-vertex}, these components cannot correspond to flowline components. 
	
	Now consider the maximal elements in the set of vertices that are terminal points of flowline edges. This ordering makes sense because we don't have outward $D_0$-unbounded edges that are flowline edges (they can be $D_1$-unbounded edges!). In other words, we want to consider those vertices $v$ such that the unique edge terminating at $v$ is a flowline edge, and if $w>v$, then the unique path from $v$ to $w$ cannot contain a flowline edge. For instance, suppose edges $A$ and edges $B$ in Figure \ref{fig:tree2} are flowline edges. Since we don't have outward $D_0$-unbounded edges, edge $B$ is a maximal flowline edge.  Similarly, suppose $D$ is a flowline edge, then $D$ is a maximal flowline edge. It can also happen that $C$ is a flowline edge. Then $C$ is maximal if $D$ is a vertex region.
	
	Now consider an unbounded edge. By construction, there exists at least one flowline edge that connects to this unbounded edge by a directed path in $\tilde{T}$. Consider the maximal one among such flowline edges. We know that after cutting the tree along such an edge, the resulting right-adjacent partial subtree cannot contain a flowline edge. For instance, in Figure \ref{fig:tree2}, the edge that is labelled $B$ may be a maximal flowline edge or the edge that is labelled $C$ or $D$ may be a maximal flowline edge. In the case that $C$ is a maximal flowline edge, we can cut the tree $\tilde{T}$ along a vertical line, as indicated in Figure \ref{fig:tree2}. However, the edges that terminate at $b_1^4,b_2^3,b_2^4,b_1^3$ cannot be flowline edges. From this discussion, it readily follows that all unbounded edges are contained in at least one right-adjacent subtree to a maximal flowline edge.
	
	Now, choose a maximal vertex, its flowline edge, and let $\triangle_{r_i}$ be the truncated domain containing the corresponding component $\Theta$ of $D_0(\epsilon)-W_0(\epsilon)$.

	\paragraph{Step 1}\label{paragraph:Step1}
	As a first step, we will show that there exists a vertical line segment in $\Theta$ such that after cutting $\triangle_{r_i}$ along the vertical line segment in $\Theta$, the image of the resulting right-adjacent component gets mapped into $T^{\ast}U((2+\eta)\delta)$ for small enough $\epsilon$. 
	
	Choose some $\frac{1}{2} \eta<\eta''<\eta$ such that $\partial U((2+\eta'')\delta)$ intersects the resulting flowline $\gamma:[-c,c]\to \tilde{C}$ transversely. By transversality, we know that there exists some $-c<c_0<c$ such that $\gamma\vert_{(c_0,c]}$ gets mapped either strictly inside $U((2+\eta'')\delta)$ or outside $U((2+\eta'')\delta)$. Suppose we are in the second situation.
	
	In that case, by Lemma \ref{lem:regionseparation}, we see that there exists some neighbourhood of the right vertical boundary of $\Theta$ that belongs strictly to the domain $D_0(\epsilon)$ for small enough $\epsilon$. Cut $\tilde{T}$ at the vertex $v$ and consider the components corresponding to the edges in the right-adjacent subtree. By assumption, none of the edges in the subtree are flowline edges, so any edge in $\tilde{T}$ that belongs in $D_0(\epsilon)$ must be a vertex edge. Now, consider the shortest path $\mathcal{P}=e_1...e_n$ such that $i(e_1)=f(e)$, the edges $e_1,...,e_n$ all correspond to horizontal components of $D_0(\epsilon)-W_0(\epsilon)$, and the edge $e_n$ corresponds to a component $\Theta_n$ that intersects with $D_1(\epsilon)$. But along such paths, all the components of $D_0(\epsilon)-W_0(\epsilon)$ (or $W_0(\epsilon)$) that correspond to the edges $e_1,...,e_{n-1}$ (or vertices) degenerate to points. Since $\gamma$ gets mapped outside $U((2+\eta'')\delta)$ for $c_0<s\leq c$, and since $D_1(\epsilon)$ gets mapped inside $U((2+\frac{9}{2}\delta_0)\delta)$ where $\delta_0$ was chosen to be less than $\frac{\eta}{10}$, we see that there can be no such paths. In that case, the non-vertical boundary of the right-adjacent component will never close up, which is a contradiction. 
	
	\paragraph{Step 2}\label{paragraph:Step2}
	Recall that $\frac{1}{2}\eta<\eta''<\eta$. Now that we have established that the flowline must have remained inside $U((2+\eta'')\delta)\cap U((2+\frac{1}{2}\eta))^c$ for $c_0<s\leq c$, we are going to prove the claim that the right-adjacent component must get mapped strictly inside $T^{\ast}(U((2+\eta)\delta))$. 
	
	Recall that $\Theta$ is still mapped inside $T^{\ast}U((2+\eta)\delta)$. By lemma \ref{non-vertex}, we know that $u_{\epsilon}\vert_{\Theta}\circ \epsilon^{-1}$ uniformly converges to a flowline on $[c_\epsilon,c_{\epsilon}]\times [0,\epsilon]$ where $c_{\epsilon}\to c$. In particular, we see that we can find a vertical segment $l_{\epsilon}$ that degenerates to some $\gamma(s)$ for $c_0<s<c$. We claim that the integral of the holomorphic Liouville form over the rescaled arc $\epsilon^{-1}u_{\epsilon}(l_{\epsilon})$ must converge to zero. 
	
	There are several ways to show this. Firstly, we observe that after taking the cotangent fibre rescaling, $\epsilon^{-1}u_{\epsilon}(l_{\epsilon})$ must uniformly converge to a continuous arc that gets mapped inside the cotangent fibre over $\gamma(s)$. This follows since the gradient estimate gives $\abs{Du_{\epsilon}}\vert_{l_{\epsilon}}=O(\epsilon)$, and so the image of the projection converges to a point, but $\abs{D(\epsilon^{-1}\cdot u_{\epsilon})\vert_{l_{\epsilon}})}$ stays uniformly finite, that we may apply the Arzela-Ascoli theorem, and conclude that the integral of the Liouville form on $l_{\epsilon}$ converges to zero. In fact, as pointed out in \cite[Remark 5.8]{Morseflowtree}, the $O(\epsilon)$-gradient estimate can be improved to a higher derivative estimate, and so $\int (\epsilon^{-1}\cdot u)\vert_{ l_{\epsilon}}^{\ast}\canliouvile\to 0$ follows more or less immediately. Alternatively, one can do this directly, since (i) for $u_{\epsilon}=(q_{\epsilon},p_{\epsilon})$, the fibre rescaling does not affect $q_{\epsilon}$, (ii) by integrated maximum principle, there exists some $C>0$ such that $\abs{p_{\epsilon}}\leq C\epsilon$, and (iii) the height of the vertical segment is uniformly bounded by $r_i$, we get
	\[\abs{\int (\epsilon^{-1}\cdot u)\vert_ {l_{\epsilon}}^{\ast}\lambda}\leq Cr\abs{\nabla q_{\epsilon}}\]
	and the right-hand side uniformly converges to zero. Since the real Liouville form is the real part of $\lambda$, we get the same convergence result for the integral of $\canliouvile$. 
	
	We now show the claim that the image of $u_{\epsilon}$ on the right-adjacent component $\bar{\triangle}_{\gamma}$ of $\triangle_{r_i}$ must get mapped strictly inside $T^{\ast}(U((2+\eta)\delta))$. To see this, observe that by the finality condition on $v$, there are no flowline edge subdomains that belong to the right-adjacent component. Since $D_1(\epsilon)$ gets strictly mapped inside $T^{\ast}(U((2+\eta)\delta))$, we see that by arguing inductively as before, there is no way for the image of the right-adjacent component to escape $T^{\ast}(U((2+\eta)\delta))$. This proves the claim.
	
	\paragraph{Step 3}\label{paragraph:Step3}
	We now show that the flowline edge $e$ must have degenerated to a flowline over the initial trivalent rays $\mathbb{R}_{\geq 0}\cdot e^{\frac{2\pi i k}{3}}, k=0,\pm 1$. Recall that these are the locus of points where $Im(z^{\frac{3}{2}})$ vanishes. Now, the horizontal boundaries of the right-adjacent component must close up since the corresponding punctures were something that we only artificially added in. So, the flowline uniquely determines the relative homotopy class $\alpha$ of the limit of the closed-up boundaries of  $\epsilon^{-1}u_{\epsilon}$ on $\scurve$, coming from the non-vertical boundaries of the right-adjacent component $\bar{\triangle}_{\gamma}$. In particular, the flowline condition tells us that the endpoints of $\alpha$ must lie on the \textit{distinct} lifts of the point $\gamma(s)$ on $\scurve$. Since the metric on $U((2+\eta)\delta)$ is K\"{a}hler, and the spectral curve is holomorphic, the same argument as in Proposition \ref{eq:boundaryparts} tells us that the $\epsilon^{-1}$-rescaling of the integral of the holomorphic symplectic form $\Omega$ on $\bar{\triangle}_{\gamma}$ must be real. Now, the rescaling of the vertical cut $l_{\epsilon}$ converges to an arc contained inside the cotangent fibre. So by Stokes' theorem, we see that the integral $\epsilon^{-1}\int u_{\epsilon}^{\ast}\lambda$ converges to $\int \alpha^{\ast}\lambda$ since $\int (\epsilon^{-1}\cdot u)\vert_{ l_{\epsilon}}^{\ast} \lambda$ will converge to zero. 
	
	Recall that the primitive of $\lambda$ on $\{(p^z)^2-z=0\}$ was given by $2(p^z)^3/3$. Then $\int \alpha^{\ast}\lambda=\primitive(\gamma(s))=\pm \frac{4}{3}(\gamma(s))^{3/2}$ which is real if and only if $\gamma(s)$ belongs to one of the initial trivalent rays.  In particular, we see that the limit of the image $u_{\epsilon}$ over $\Theta$ must have lain on the spectral network $\snetwork(0)$. This proves the claim.
	
	\paragraph{Step 4}\label{paragraph:Step4}
	
	We have seen thus far that after cutting the domain along a vertical segment contained in the "maximal" flowline component, the right-adjacent component gets mapped inside $U(2+\eta'')\delta$ and the flowline degenerates to the initial walls of the spectral network. We now claim that cutting $\stdm$ along all such vertical segments, the resulting unique component containing the negative strip-like end must belong entirely inside $D_0(\epsilon)$. Furthermore, we will show that the image of the negative strip-like end degenerates to a broken Morse sub-tree (in the sense of \cite[Section 2]{Morseflowtree}) on the spectral network. 
	
	Here, we need to use the fact that the spectral curve is rank $2$. The previous arguments all extend to the case where our Lagrangian is geometrically bounded at horizontal infinity, but otherwise just simply branched, given that one is able to prove the necessary isoperimetric inequalities and gradient estimates.
	
	Cut out the tree $\tilde{T}$ at the edges corresponding to the vertical line segments as above. Let $\bar{T}$ be the resulting truncated tree. Then, since the initial trivalent rays never return to $B_{2\delta}$, we see that the corresponding edge components must belong to $D_0(\epsilon)$. Now, suppose such edges meet at a vertex. This vertex must also belong to $D_0(\epsilon)$ because the initial trivalent rays never return to $B_{2\delta}$. The vertex must belong to the set where the corresponding initial rays intersect. However, these initial rays never meet because we are in the rank $2$ situation. Therefore, we see that the edge components that meet at a vertex must lie on the same walls of $\snetwork(0)$ and must travel in the direction of $\snetwork(0)$. However, since the tree $\tilde{T}$ is \textit{rooted}, there is precisely \textit{one} in-going edge at the corresponding vertex. In particular, the in-going edge cannot traverse back the walls of the spectral network because this edge must also have the same ordered pair of sheets labelling as the walls on $\snetwork(0)$ by the connectedness of each boundary component. Indeed, if the wall is labelled $12$, then all the out-going horizontal components of $D_0(\epsilon)$ adjacent to the component of $W_0(\epsilon)$ must have horizontal boundary labelling $12$, and so the in-going horizontal boundary labelling must also be $12$. So we see that the components remain inside $D_0(\epsilon)$.
	
	Now argue by induction, consider the shortest path contained in $\bar{T}$ connecting a positive external edge of $\bar{T}$ to a $D_1(\epsilon)$-component, and consider the unique sub-tree sharing the same initial vertex. Since the walls on $\snetwork(0)$ never return to the branch points, all the edges belong to $D_0(\epsilon)$, and so we see that all the left-adjacent components to the vertical segments belong to $D_0(\epsilon)$. In particular, they degenerate to Morse flow trees whose image is entirely contained in $\snetwork(0)$. 	This cannot be so if $z$ is on $C(\delta;E)$, a contradiction.
\end{proof}

We end this section with a brief discussion on generalizing to the higher-rank situation. We give the following definition.
\begin{definition}\label{definition:tamespectralcurve}
	We say that a properly embedded Lagrangian submanifold $\scurve\subset T^{\ast}\tilde{C}$ is a \emph{tame spectral curve} if there exists a metric $g$ on $\tilde{C}$, $\injradius,\heightR,\sheetgap>0$ and a compact subset $K\subset \tilde{C}$ such that the following holds. 
	\begin{itemize}
		\item The metric $g$ is complete, geometrically bounded, and the minimal injectivity radius of $g$ is bounded below by $r>0$. 
		\item With respect to the metric $g$, $\scurve$ is contained in $D_R^{\ast}M$.
		\item $\scurve\to \tilde{C}$ is a degree $n$ simple branched cover.
		\item  Outside $K$, $\scurve\to \tilde{C}$ is a proper covering map, and for $m\in T^{\ast}(\tilde{C}-N_{\injradius}(K))$, the smooth sheets of $\epsilon \scurve$ over $B_{\injradius}(m)$ uniformly $C^{\infty}$-converges to the zero section, independent of $m$, as $\epsilon\to 0$, with respect to $g$. Furthermore, over $B_{\injradius}(m)$, the sheets of $\scurve$ are given by $df_1,...,df_n$ such that $\abs{df_i-df_j}\vert_{B_{\injradius}(m)}>\sheetgap$.
		\item For all $m\in B_{\injradius}(z)$, and $m'\in \pi^{-1}(m)$ , $B^{g}_{\sheetgap}(m')\cap \scurve$ is connected. 
	\end{itemize} 
\end{definition}
For such tame spectral curves, the arguments in \cref{Floer Theory on Open Manifolds Sect 1} apply directly. We don't have strict energy controls as in \cref{Spectral Curves}. The only place where rank $2$ (and the flatness of the $g^{\phi}$-metric) is used in \cref{Adia Degen} is the truncated reverse isoperimetric inequality in \cref{lengthestimate} and Step 4 in the proof of \cref{Mainestimate} the adiabatic degeneration argument in the proof of \cref{maintheorem}. In the higher rank case, the truncated reverse isoperimetric inequality needs to be generalized to handle the situation where we don't have everything flat outside $K$. The only place where we needed flatness was relating $\int dd^c h$ to $2\omega$. The proof can be modified with a Weinstein neighbourhood argument and will appear elsewhere. On the other hand, the argument in the proof of \cref{maintheorem} breaks down if some walls return to the branch points. This is where the \textit{existence} of higher rank spectral networks becomes essential. 
\section{Wall-crossing analysis}\label{Wall-Crossing Analysis Section}
In this section, we compute the Floer cohomology local system $z\mapsto HF(\scurve,F_z)$ and prove Theorem \ref{FullNon-Abelianization}). To do this, we must show that the Floer-theoretic parallel transport along a path $\alpha$ contained in $C(\delta;E)$ is given by the pushforward of $\mathcal{L}$, and the Floer-theoretic parallel transport along the ``short paths" (see Section \ref{introNonAb}) admits the form \eqref{wallcrossingterm}. In Section \ref{subsection:Hofer-type Gromov Compactness}, we define and study the relevant passive continuation strips. In Section \ref{subsubsection:wall-orientation}, we set up some conventions, fix the branch cut and the sheet ordering data once and for all, and specify the path groupoid generators on $C$ that we will use throughout the section. In Section \ref{subsubsection:Floerdata}, we specify the Floer data that we will use for the Lagrangian pair $(\epsilon\scurve,F_z)$. In Section \ref{subsubsection:shortpaths}, we study the moduli problem for parallel transports along the ``short paths". The main result is Proposition \ref{proposition:shortpathspassive} which explains the form \eqref{wallcrossingterm} up to sign. In Section \ref{subsubsection:longarcs}, we study the moduli problem for parallel transports along arcs contained in $C(\delta;E)$. The main result is Proposition \ref{No-Go moving}; we show using Theorem \ref{maintheorem} that for infinitesimal fibre parallel transports, the relevant continuation strips are all constant strips. This explains the form \eqref{pushforwardoutside} up to sign. 

In Section \ref{gradings and spin structures}, we specify the necessary grading and spin structure data in order to compute the Floer cohomology local system. In Section \ref{gradingsss}, we define the grading functions. In Section \ref{subsubsection:spin structures}, we use twisted local systems to compute the signs of parallel transport maps. In Section \ref{NonAbsss}, we use the results in Section \ref{subsubsection:spin structures} to compute the Floer-theoretic parallel transport maps and prove Theorem  \ref{FullNon-Abelianization}. 

\subsection{Moduli problem for parallel transports}\label{subsection:Hofer-type Gromov Compactness}
In this section, we define and study the various moduli problem for Floer-theoretic parallel transport maps associated to horizontal and vertical geodesic arcs on the base. We use the conventions from Section \ref{continuationstripsss}. 

\subsubsection{Wall-chamber data}\label{subsubsection:wall-orientation}

In this section, we fix the branch cut data, a choice of a ``positive sheet" of $\phi$ for each component of $C-\snetwork(0)$, the set of base points $\mathcal{P}_C$ for the path groupoid over $\tilde{C}$ and the generators for the path groupoid morphisms. We will need the following definition.

\begin{definition}
	Let $\gamma$ be an oriented horizontal trajectory in $C$. Then the \emph{positive sheet} $+\sqrt{\phi}$ \emph{along} $\gamma$ is the unique sheet of $\sqrt{\phi}$ such that the line element $\sqrt{\phi}(\gamma(s))\cdot \gamma'(s)ds$ is real and positive, for any smooth parametrization of $\gamma$ that respects the chosen orientation.
	
	Let $\mathcal{V}\subset C^{\circ}$ be a vertical neighbourhood of $\gamma$, and let $+\sqrt{\phi}$ be the positive sheet along $\gamma$. Then we say that the point $z$ in $V-\gamma$ lies \emph{above} $\gamma$ if the integral $\int Im(+\sqrt{\phi})$ along the unique vertical segment between a point on $\gamma$ and $z$ is positive. Otherwise, we say the point $z$ lies \emph{below} $\gamma$.  
\end{definition}
Note that for small enough $\mathcal{V}$, $\mathcal{V}-\gamma$ consists of two connected components, the one that lies above $\gamma$, and the one that lies below $\gamma$.

Now let $w$ be a wall on $\snetwork(0)$. We always orient the walls in the outward direction, travelling away from the branch points. This orientation on the wall $w$ picks out a unique positive sheet of $\sqrt{\phi}$ along $w$.

\begin{definition}\label{definition:chamberabovewall}
	Let $w$ be a wall. We define $\mathcal{Z}^h(w)$ to be the unique component of $C-\snetwork(0)$ containing the point that lies above $w$.
\end{definition}
Note that the conformal equivalence in Proposition \ref{domaindecomposition} defined using the positive sheet $+\sqrt{\phi}$ sends $w$ to the lower right bottom corner.

\begin{lemma}\label{positivewallpair}
	Let $\mathcal{Z}^h$ be a component of $C-\snetwork(0)$. Then $\mathcal{Z}^h=\mathcal{Z}^h(w)$ for at most two walls. In fact, $w$ is unique if and only if $\mathcal{Z}^h$ has the conformal type of the upper half plane. 
\end{lemma}
\begin{proof}
	Given a conformal equivalence of $\mathcal{Z}^h$ with a finite horizontal strip (which sends $\phi$ to $dz^2$), $w$ corresponds to either the right bottom boundary or the left upper boundary. Reversing the parametrization by $z\to -z$ swaps the two. 
\end{proof}

\paragraph{Branch-cut data}\label{paragraph:branchcutdata}
We fix the branch-cut data. Let $b$ be a zero of $\phi$. By Proposition \ref{zeroofphilocalform}, there exists a neighbourhood of zero $(U_b,\phi)$ and a biholomorphism $(U_b,b,\phi)\simeq (D,0,zdz^2)$ whose germ at $b$ is unique up to a phase factor of $e^{2\pi k i/3},k=0,1,2$. Choose a phase factor once and for all and introduce a branch cut on the negative real axis. 

Label the wall corresponding to the positive ray $\mathbb{R}_{>0}e^{i\cdot 0}$ by $w_0$, the wall corresponding to $\mathbb{R}_{>0}\cdot e^{i\cdot 2\pi/3}$ by $w_1$, and the wall corresponding to $\mathbb{R}_{>0}\cdot e^{i4\pi/3}$ by $w_{-1}$. Give $\pm$ labels for the two sheets of $\sqrt{\phi}$, $+$ and $-$ with respect to the branch cut. Then given a wall, if the unique positive sheet of $\sqrt{\phi}$ along the wall is the $+$-sheet, we label the wall $-+$. Otherwise, we label the wall $+-$. So we see that at each vertex of the spectral network $\snetwork(0)$, the three walls are labelled $+-$, $+-$ and $-+$. In particular, $w_0$ is now labelled $-+$ and $w_1$ and $w_{-1}$ are labelled $+-$. We do this for each of the zeroes of $\phi$. See Figure \ref{Arcsandnetworkfig}.

%We have a unique sheet of $\sqrt{\phi}$ along $w$ such that $\int \sqrt{\phi}$ along $w$ with respect to the outward orientation is positive. W Then $w$ separates this vertical neighbourhood into regions where $\int Im \sqrt{\phi}$ is increasing and where $\int Im \sqrt{\phi}$ is decreasing, respectively.  containing the former region.  and it comes along with a conformal equivalence sending the wall to the lower right bottom corner, or equivalently a choice of $+\sqrt{\phi}$).

\paragraph{Chamber sheet data}\label{paragraph:chambersheetdata}
We fix the ``positive sheet" of $\sqrt{\phi}$ over each component of $C-\snetwork(0)$. Let $\mathcal{Z}^h$ be a component of $C-\snetwork(0)$ and choose an oriented generic horizontal trajectory $\gamma(\mathcal{Z}^h)$. This orientation picks out a positive sheet of $\sqrt{\phi}$ along $\gamma$. We use the positive sheet of $\sqrt{\phi}$ with respect to $\gamma(\mathcal{Z}^h)$ to identify $\mathcal{Z}^h$ with the corresponding horizontal subdomain in $\mathbb{C}$. From now on,  we'll write $\mathcal{Z}^h(\delta;E)=C(\delta;E)\cap \mathcal{Z}^h$, and we will abuse notation and let $\mathcal{Z}^h$ also denote its representative as a horizontal subdomain in $\mathbb{C}$. On the other hand, when we write $\mathcal{Z}^h(w)$ for $w$ a wall, we will find its representative as a horizontal subdomain in $\mathbb{C}$, using the trivialization induced from the positive sheet of $\sqrt{\phi}$ along $w$. Note that there exists a single wall $w$ on the boundary of the closure of $\gamma(\mathcal{Z}^h)$ respecting the choice of $+\sqrt{\phi}$ making $\mathcal{Z}^h=\mathcal{Z}^h(w)$.  
\paragraph{Path groupoid data}\label{paragraph:pathgroupoidbasepoints}
We now fix the path groupoid base points $\mathcal{P}_C$ and the path groupoid generators. Let $\mathcal{V}$ be a component of $\mathcal{V}(\delta;E)$, and let $w$ be its core wall. On this component, we take the trivialization induced by the positive sheet of $\sqrt{\phi}$ with respect to $w$. For each wall $w$, we choose a point $b(w)\in \mathcal{V}(\delta;E)$ and the adjacent points $b^{u}(w)=b(w)+i\eta(w)$ and $b^{d}(w)=b(w)-i\eta(w)$ for some $\eta(w)>0$. We choose them in a way that the adjacent points contained in a component of $\mathcal{Z}^h(\delta;E)$ are connected by either a vertical or a horizontal arc fully contained in $\mathcal{Z}^h(\delta;E)$.  We can always manage this to happen by taking $\delta$ much smaller than $\min \frac{\abs{a-b}}{2}$ where the minimum is taken over all the horizontal strips $\mathcal{Z}(a,b)$ as in Proposition \ref{horizontaldomaindevision}. 

Having made these choices, we choose the set 
\begin{align}\label{definition:groupoidbasepoints}
	\mathcal{P}_C:=\{b^{\bullet}(w): w \text{ is a wall,} \bullet=h,v\},
\end{align}
to be the set of base points for the path groupoid of $\tilde{C}$.  
\begin{definition}\label{definition:groupoidgenerators}
	Let $w$ be a wall, and let $w'$ be a wall that lies on the left bottom boundary of $\mathcal{Z}^h(w)$. The arc $\alpha(w,w')$ is the unique horizontal arc contained in $\mathcal{Z}^h(w)\cap C(\delta;E)$ connecting $b^{u}(w)$ to $b^{d}(w')$. The arc $\alpha(w)$ is the shortest vertical arc connecting $b^{d}(w)$ to $b^{u}(w)$. The arc $\gamma(w,w'')$ is the unique vertical arc contained in $\mathcal{Z}^h(w)$ connecting $b^{u}(w)$ to $b^d(w'')$ where $w''$ is the wall on the right top boundary of $\mathcal{Z}^h(w)$. For any other $w'$, we set $\alpha(w,w')$ and $\gamma(w,w')$ to be the emptyset. 
\end{definition}
We then set the path groupoid generators to be \[\{\alpha(w,w')^{\pm 1},\gamma(w,w')^{\pm 1},\alpha(w)^{\pm 1}: w, w'\text{ a wall on } \snetwork(0)\}.\]

\begin{figure}[t]
	\includegraphics[width=0.5\textwidth]{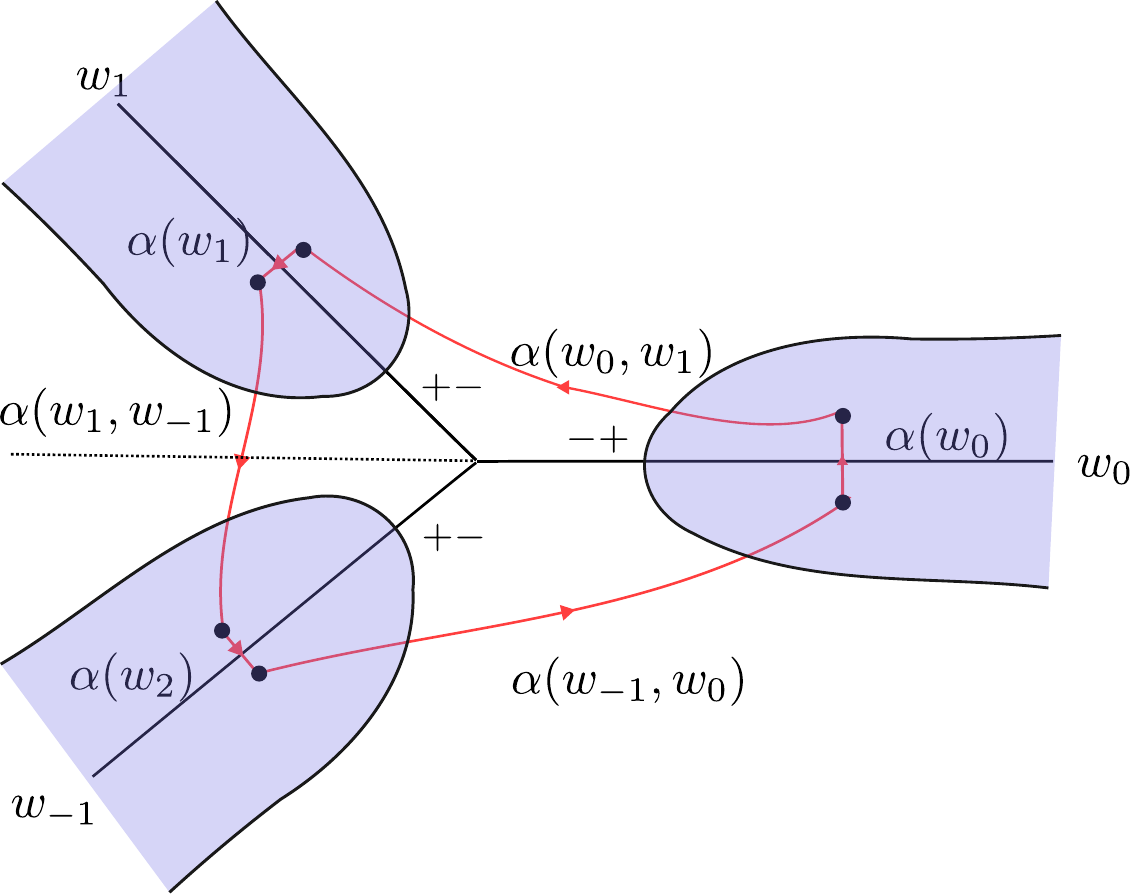}
	\centering
	\caption{The spectral network near a zero and the arcs $\alpha$ indicated on the diagram.}
	\label{Arcsandnetworkfig}
\end{figure}

\subsubsection{Floer data}\label{subsubsection:Floerdata}
We now study the passive continuation strips associated to fibre parallel transports. To define the moduli spaces, we fix a regular Floer datum for the pair $\epsilon\scurve$ and $F_z$ for $z\in C(\delta;E)$ and $0<\epsilon<\epsilon_0(\delta;E)$.  We start with the following lemma:
\begin{lemma}\label{lemma:Sasakilocalisation}
	There exists an auxiliary function $\rho:[1,\infty)\to [1,\infty]$ satisfying $\rho(r)=r$ for $r\gg 1$ such that $u$ is a $J_{g^{\phi}_{\delta}}$-strip bounded between $\epsilon\scurve$ and $F_z$ if and only if it is a $J_{con}$-strip bounded between $\epsilon\scurve$ and $F_z$, where $J_{con}$ is the conical deformation of $J_{g^{\phi}_{\delta}}$ obtained using $\rho$. 
\end{lemma}
\begin{proof}
	Choose a smooth, positive increasing function $\rho:[1,\infty)\to [1,\infty)$ such that $\rho(r)=1$ for $r<3$ and $\rho(r)=r$ for $r>5$. Let $J_{con}$ be the $\rho$-conically deformed almost complex structure. By Lemma \ref{verticalconfinement}, since $J_{con}$ is of general contact type, the discs lie in $D_{2.5}T^{\ast}M$, where $J_{con}=J_{g^{\phi}_{\delta}}$. This finishes the proof. 
\end{proof}
%For the rest of this section, we abuse notation and write $J$ for $J_{con}$. 
\begin{corollary}
	For $z\in C(\delta;E)$ and $0<\epsilon<\epsilon_0$, the Floer datum $(\epsilon\scurve,F_z,J_{con})$ is regular.\label{regularityofJ}
\end{corollary}
\begin{proof}
	By Theorem \ref{maintheorem} and Lemma \ref{lemma:Sasakilocalisation} we see that there are no non-trivial $J_{con}$ holomorphic strips bounded between $F_z$ and $\epsilon\scurve$ for $z\in C(\delta;E)$ and $0<\epsilon<\epsilon_0$. So all the strips are constant, which are regular because local configurations near the intersection points coincide with the intersection of $\mathbb{R}^2$ and $i\mathbb{R}^2$ in $\mathbb{C}^2$ equipped with the standard complex structure. 
\end{proof}

We now fix some $0<\epsilon<\epsilon_0$ for the rest of the section. Following   \cref{regularityofJ}, we set the Floer datum to be $(\epsilon\scurve,F_z,J_{con})$. 
\subsubsection{Moduli problem for arcs $\alpha(w)$}\label{subsubsection:shortpaths}

We now study the moduli problem associated to Floer-theoretic parallel transports along the arcs $\alpha(w)$. As before, let $w$ be a wall and let $\mathcal{V}(w)$ be the component of $\mathcal{V}(\delta;E)$ containing $w$ as its core. 
Choose a bump function $\xi: \mathcal{V}(w)\to [0,1]$ such that $\xi=1$ in a small neighbourhood of the arc $\alpha(w)$. Take the flat conformal coordinate $\int \sqrt{\phi}$ over $\mathcal{V}(w)$ using the positive sheet induced by the outward orientation on the wall $w$. The outward orientation allows us to order the lifts of $z\in \mathcal{V}(w)$ to $z^{\pm}$. Note that the canonical ordering introduced in Proposition \ref{horizontaldomaindevision} agrees with the ordering of the lifts $z^{\pm}$. 

Consider the autonomous Hamiltonian $S^v=\xi(x,y)p^y$. Recall that $2\eta(w)$ is the distance between $b^{d}(w)$ and $b^u(w)$. Let $\chi_s^v=\chi^v_s$ denote the Hamiltonian isotopy generated by $S^v$. The Hamiltonian isotopy $\chi_s^v$ has the following property:
\begin{lemma}\label{Invarianceunderverticaltransports}
	$\scurve$ and its $\mathbb{R}_{>0}$-rescalings are invariant under $\chi_s^v$. 
\end{lemma}
\begin{proof}
	The Hamiltonian vector field is given by:
	\begin{align}\label{eq:verticalHam}
		X_{S^v}=-p^y(\frac{\partial\xi}{\partial x}\frac{\partial}{\partial p^x}+\frac{\partial\xi}{\partial y}\frac{\partial}{\partial p^y})+\xi(x,y)\frac{d}{dy}.
	\end{align}
	However, on $T^{\ast}(\mathcal{V}(w))$, $\scurve$ equals $\{p^x=\pm 1, p^y=0\}$. So the vector field $X_{S^v}$ restricts there as $\rho(x,y)\frac{d}{dy}$, the flow of which preserves the set $\{p^x=\pm 1,p^y=0\}$. 
\end{proof}

By Corollary \ref{regularityofJ},  Floer data $(\epsilon\scurve,F_{b^u(w)},J_{con})$ and $(\epsilon\scurve,F_{b^d(w)},J_{con})$ are all regular. Let $J^{short}$ be a uniformly admissible family of almost complex structures on $\mathcal{Z}$ such that $J^{short}(s,t)=J_{con}$ for $s\ll 0$ and $J^{short}(s,t)=(\chi_1^v)^{\ast}J_{con}$ for $s\gg 0$. Let $\mathcal{M}^{short}(w)$ be the moduli space of $J^{short}$-holomorphic maps $u:\mathcal{Z}\to T^{\ast}\tilde{C}$ satisfying the following boundary conditions
\begin{align}\label{moduliproblem:short} 
	\begin{cases}
		u(s,0)\subset  \epsilon\scurve\\
		u(s,1)\subset F_{z}\\
		\lim_{s\to -\infty} u(s,t)\in F_{b^d(w)} \cap \epsilon\scurve \\
		\lim_{s\to +\infty} u(s,t)\in F_{b^d(w)} \cap \epsilon\scurve \\
		%\lim_{s\to -\infty} u(s,t)= \lim_{s\to +\infty} u(s,t).
	\end{cases}
\end{align}
By Lemma \ref{Invarianceunderverticaltransports}, $\mathcal{M}^{short}$ coincides with the moduli space of \textit{passive} continuation strip equation associated to $\epsilon\scurve$ and $\chi_s^v$. We choose a generic $J^{short}$ so that $\mathcal{M}^{short}$ is transversely cut out. We have the decomposition
\begin{align}\label{equation:shortdecomposition}
	\mathcal{M}^{short}=\mathcal{M}^{short,diag}(w)\sqcup \mathcal{M}^{short,nondiag}(w)
\end{align}
where $\mathcal{M}^{short,diag}(w)$ is the moduli of passive continuation strips that travel from $(\epsilon b^{d}(w))^{\pm}$ to $(\epsilon b^{d}(w))^{\pm}$, and $\mathcal{M}^{short,nondiag}(w)$ is the moduli of passive continuation strips that travel from $(\epsilon b^{d}(w))^{\pm}$ to $(\epsilon b^d(w))^{\mp}$. Let  $\mathcal{M}^{short,-}(w)$ denote the moduli space of continuation strips that travel from $(\epsilon b^{d}(w))^{+}$ to $(\epsilon b^{d}(w))^{-}$.

\begin{proposition}
	\label{proposition:shortpathspassive}
	$\mathcal{M}^{short,diag}(w)$ consist of constant maps and $\mathcal{M}^{short,-}(w)$ is empty. 
\end{proposition}
\begin{proof}
	%By Proposition \ref{Invarianceunderverticaltransports}, $u\in \mathcal{M}^{short}(w)$ satisfies the following:
	%\[\begin{cases}
		%u(s,0)\subset t\scurve&\\
		%u(s,1)\subset F_{z}.
		%\end{cases}\]
		%In particular, 
		Let $u\in \mathcal{M}^{short,diag}(w)$ be a strip, then $\lim_{s\to -\infty} u(s,t)=\lim_{s\to +\infty} u(s,t)=(b^{d}(w))^{\pm}$. From the boundary conditions \cref{{moduliproblem:short}}, we see that the energy of the strip $u$ depends only on the primitive of $\lambda_{re}$ on $\epsilon \scurve$ and its end points. Since the end points of the strip are the same, we see that $\int u^{\ast}\omega=0$. Since the energy vanishes and $J^{short}$ is $\omega$-compatible, $u$ is constant, and so the moduli space $\mathcal{M}^{short,diag}(w)$ must consist of constant maps. 
		
		For $u\in \mathcal{M}^{short,-}$, we have $\lim_{s\to-\infty}u(s,t)=(b^d(w))^{+}$ and $\lim_{s\to-\infty}u(s,t)=(b^d(w))^{-}$. So the energy of discs in $\mathcal{M}^{short,-}$ is equal to $\epsilon \primitive(\epsilon^{-1}(b^d(w))^{-})-\epsilon \primitive(\epsilon^{-1}(b^d(w))^{+})$ which is negative by \cref{horizontaldomaindevision} . By positivity of energy, $\mathcal{M}^{short,-}(w)$ must be empty.
	\end{proof}
	
	\subsubsection{Moduli problem for arcs contained in $C(\delta;E)$}
	\label{subsubsection:longarcs}
	
	We now study the moduli problem associated to Floer-theoretic parallel transports along arcs contained in $C(\delta;E)$. As before, let $\mathcal{Z}^h$ be a horizontal chamber and let $\mathcal{Z}^h(\delta;E)=C(\delta;E)\cap \mathcal{Z}^h$. Let $+\sqrt{\phi}$ be the positive sheet of $\phi$ picked out by $\gamma(\mathcal{Z}^h)$. Let $z^{\pm}$ be the corresponding ordering on the lifts of $z\in \mathcal{Z}^h(\delta;E)$. 
	
	Choose a compactly supported smooth positive bump function $\rho(\mathcal{Z}^h):\mathcal{Z}^h\to [0,1]$ once and for all such that $\rho(\mathcal{Z}^h)=1$ on $C(\delta;E)\cap \mathcal{Z}^h$ and $\rho=0$ on $\mathcal{Z}^h\cap U((2+\eta)\delta)$. We will consider the following two Hamiltonians. %we will discuss in this section will be of the following two form
	\begin{align}\label{Hamiltonianform}
		H^h&:= \rho{(\mathcal{Z}^h)}p^x\\
		H^v&:= \rho{(\mathcal{Z}^h)}p^y.
	\end{align}
	By Proposition \ref{horizontaldomaindevision}, we see that we have the canonical ordering on the lifts of $z$ to $\scurve$, for $z\neq \snetwork(\pi/2)$. On the other hand, we also have the ordering on the lifts given by the choice of $+\sqrt{\phi}\vert_{\mathcal{Z}^h}$. For convenience, we may regard the first type of ordering as an {energy} ordering, and the second type of ordering as a {sheet} ordering. Note that for points contained in the right hand side of $\mathcal{Z}^h-\snetwork(\pi/2)$, the sheet ordering and the energy ordering coincide, but they become opposite when we cross $\snetwork(\pi/2)$. 
	
	For $z\in \mathcal{Z}^h(\delta;E)$, let $\alpha^{h}_z$ and $\alpha_z^{v}$ be arc-length parametrized horizontal and vertical arcs, respectively, beginning at $z$. Let $\psi^{h}_{\alp}$ denote the time-$\alp$ flow of the autonomous Hamiltonian $H^{h}$. Similarly, let $\psi^{v}_{\alp}$ denote the time-$\alp$ flow of the autonomous Hamiltonian $H^{v}$. Note that the time-$\alp$ flow $\psi^{\bullet}_\alp$ sends $F_z$ to $F_{\alpha_z^{\bullet}(\alp)},\bullet=v,h$. From now on, the superscript $\bullet$ will denote either $v$ or $h$. Given $z\in \mathcal{Z}^h(\delta;E)$, we will only consider those $\alp\in \mathbb{R}$ such that $\alpha^{\bullet}_z(\alp)\in \mathcal{Z}^h(\delta;E)$.
	
	We will need the following formula.
	\begin{lemma}\label{newprimitive}
		Let $z\in \mathcal{Z}^h(\delta;E)$. The primitive of $\primitive^{\bullet}_{\alp}$ of $\lambda_{re}$ on $(\psi^{\bullet}_{\alp})^{-1}(\epsilon \scurve)$ at $\epsilon z^{\pm}$ satisfies:
		\[\primitive^{\bullet}_{\alp}(\epsilon z^{\pm})=\epsilon\primitive(\psi^{\bullet}_{\alp}(z^{\pm})).\]
		In particular, $\primitive^h_{\alp}(\epsilon z^{\pm})=\epsilon\primitive(z^{\pm})\pm \epsilon \alp$ and $\primitive^v_{\alp}(\epsilon z^{\pm})=\epsilon\primitive(z^{\pm})$.
		
	\end{lemma}
	
	\begin{proof}
		By Lemma \ref{movingboundarycomputation},
		\begin{align}\label{deformedpotential}
			(\primitive^{\bullet}_\alp\circ (\psi^{\bullet}_{\alp})^{-1})(\epsilon z^{\pm})&=\epsilon\primitive(z^{\pm})+\int(H^{\bullet}\circ (\psi^{\bullet}_{\alp})^{-1})(\epsilon z^{\pm})-\lambda_{re}(X_{H^{\bullet}})(\psi^{\bullet}_{\alp})^{-1}(\epsilon z^{\pm}).
		\end{align}
		One can check directly that the integrand in the second term on the right hand side of \eqref{deformedpotential} vanishes. The next statement follows from the observation that $W$ depends only on the horizontal distance to a zero of $\phi$ on the boundary of $\mathcal{Z}^h$; $\psi_{\alp}^{h}$ changes this distance by $s$ whereas $\psi_{\alp}^{v}$ leaves this distance invariant. (see Lemma  \ref{horizontallengthcontrol}) 
	\end{proof}
	
	Choose a smooth strictly increasing elongation function $l:(-\infty,\infty)\to [0,1]$ once and for all such that $l(s)=0$ for $s\leq -2$ and $l(s)=1$ for $s\geq 2$. Write:
	\[J^{\bullet}_{\pm\alp}(s,t)=(\psi_{\pm \alp l(s)}^{\bullet})^{\ast}J_{con},\]
	and consider the moduli spaces $\mathcal{M}^{\bullet}_{\pm \alp,z}$ of solutions of 
	\begin{align}\label{continuationstripseq}
		\begin{cases}
			\bar{\partial}_{J^{\bullet}_{\pm\epsilon}} u=0 &  \\ 
			u(s,0)\subset (\psi^{\bullet}_{\pm \alp l(s)})^{-1}(\epsilon\scurve)&\\
			u(s,1)\subset F_{z}&\\
			\lim_{s\to -\infty} u(s,t)\in F_z \cap \epsilon\scurve \\
			\lim_{s\to +\infty} u(s,t)\in F_z \cap \epsilon\scurve.
		\end{cases}
	\end{align}
	Intuitively, these are the continuation strips that contribute to the parallel transport along the paths $s\mapsto \alpha_z^{\bullet}(\pm \alp s),\bullet=v,h$. Keep in mind that the flow $\psi^{\bullet}_{\alp l(s)}$ preserves $\epsilon\scurve$ and $J_{con}$ invariant in a neighbourhood of $z$. Hence, $F_z\cap {\psi^{\bullet}_{\pm \alp}}^{-1}(\epsilon\scurve)=F_z \cap \epsilon\scurve$. As before, we split the moduli space $\mathcal{M}^{\bullet}_{\pm \alp,z}$ into the diagonal part and the non-diagonal part:
	
	\begin{align}
		\mathcal{M}^\bullet_{\pm \alp,z}=\mathcal{M}^{diag,\bullet}_{\pm \alp,z}\sqcup \mathcal{M}^{nondiag,\bullet}_{\pm \alp,z}.
	\end{align}
	The diagonal part consists of solutions of \eqref{continuationstripseq} that travel from $\epsilon z^{\pm}$ to $\epsilon z^{\pm}$ with respect to the sheet ordering. The non-diagonal part consists of solutions of \eqref{continuationstripseq} that travel from $\epsilon z^{\mp}$ to $\epsilon z^{\pm}$. The moduli space $\mathcal{M}^{nondiag,\bullet}_{\pm\alp,z}$ further decomposes into $\mathcal{M}^{+,\bullet}_{\pm \alp,z}$  and $\mathcal{M}^{-,\bullet}_{\pm \alp,z}$ consisting of passive continuation strips travelling from $z^{-}$ to $z^{+}$ and $z^{+}$ to $z^{-}$, respectively. 
	
	We now state the main analytic result of Section \ref{subsection:Hofer-type Gromov Compactness}.
	\begin{proposition}\label{No-Go moving}
		Given $z\in \mathcal{Z}^h(\delta;E)$, there exists some $\alp(z)>0$ such that for any $0<\alp<\alp(z)$, the following holds.
		\begin{itemize}
			\item The moduli spaces $\mathcal{M}^{diag,\bullet}_{\pm \alp,z}$ consist of constant strips. 
			\item The moduli spaces $\mathcal{M}^{nondiag,\bullet}_{\pm \alp,z}$ are empty.
		\end{itemize}
		In particular, the moduli spaces $\mathcal{M}^{\bullet}_{\pm \alp,z}$ are regular.
	\end{proposition}
	
	For the proof of Proposition \ref{No-Go moving}, we will need the following statement which we will show in Section \ref{subsubsection:proofofHofertype}. 
	\begin{proposition}\label{Hofertypegromovcompacness}Let $z\notin \snetwork(\pi/2)$ and let $\alp_n$ be a sequence of positive real numbers converging to zero. Let $u_n\in \mathcal{M}^{+,\bullet}_{\pm \alp_n,z}$ be a sequence of non-constant $J^{\bullet}_{\alp_n}$-holomorphic strips with respect to the energy ordering. Then $u_n$ Gromov converges to a non-constant broken strip bounded between $F_z$ and $\epsilon\scurve$.
	\end{proposition} 
	
	\begin{proof}
		We first treat the case of $\mathcal{M}^{v}_{\alp,z}$. The case of $\mathcal{M}^v_{-\alp,z}$ is entirely analogous. We prove the first assertion in Proposition \ref{No-Go moving}. Let $u$ be a solution of the equation
		\begin{align}\label{equation:diagvertical}
			\begin{cases}
				\bar{\partial}_{J^{v}_{\epsilon}} u=0 &  \\ 
				u(s,0)\subset ({\psi^v_{\alp l(s)}})^{-1}\big(\epsilon\scurve\big)&\\
				u(s,1)\subset F_{z}&\\
				\lim_{s\to -\infty} u(s,t)= \epsilon{z}^{\pm}\\
				\lim_{s\to +\infty} u(s,t)= \epsilon{z}^{\pm}.
			\end{cases}
		\end{align}
		The action of the pair $((\psi^v_{\alp})^{-1}(\epsilon\scurve),F_z)$ is given by $ W_{\alp}^{v}$ and the pair $(\epsilon\scurve,F_z)$ by $\epsilon W=W_{0}^v$. By Lemma \ref{newprimitive}, the geometric energy is equal to 
		\begin{align}\label{eq:diagonalgeometricenergy}
			Area(u)&= W_{\alp}^{v}(\epsilon z^{\pm})- W_{0}^{v}(\epsilon{z}^{\pm}) +\alp \int_{-\infty}^{\infty} H^v(u(s,t))l'(s)ds\nonumber\\
			&=\alp \int_{-\infty}^{\infty}H^v(u(s,t))l'(s)ds.
		\end{align}
		For small enough $\alp$, $({\psi^v_{\alp l(s)}}^{-1})(\epsilon\scurve)\cap \supp H^v$ lies inside $D_1^{\ast}\tilde{C}$ for all $s\in [0,1]$. Hence $\sup H^v_s(u)$ is bounded above. So as $\alp\to 0$, \eqref{eq:diagonalgeometricenergy} uniformly converges to zero. 
		
		Now the Hamiltonian isotopy $\psi^v(z)$ leaves $\epsilon\scurve$ and $F_z$ invariant in some neighbourhood of $F_z$ since the generating function is locally of form $\pm \rho(x,y)p^y$. In this neighbourhood, the configuration $(T^{\ast}\tilde{C},J^v_{\alp},({\psi^v_{\alp l(s)}})^{-1}\big(\epsilon\scurve\big),F_z)$ is isometric to the standard configuration  $(T^{\ast}\mathbb{R}^2,J_{std},\{p^x=\pm \epsilon,\:p^y=0\},\{x=y=0\})$. In the latter configuration, the equation \eqref{equation:diagvertical} becomes the standard $J_{std}$-holomorphic strip equation with non-moving boundary conditions. Applying the boundary estimate, we see that the disc cannot escape such a neighbourhood of $F_z$,  for small enough $\alp$. However, in the standard configuration, there are no non-constant $J_{std}$-holomorphic strips. So we conclude that for small enough $\alp$, $\mathcal{M}^{diag,v}_{\alp,z}$ must consist of constant strips. 
		
		We now show that the non-diagonal part of $\mathcal{M}^v$ is empty and hence prove the second assertion. We first treat the case $z\in \snetwork(\pi/2)$. The difference $W(z^{+})-W(z^{-})$ of the primitive between the two lifts  vanishes by Lemma \ref{horizontallengthcontrol}. Then the difference of the action of the intersection point vanishes at $z$, and the geometric energy is again of size $O(\alp)$. By the previous observation, it follows that for some $\alp(z)>0$, all the passive continuation strips associated to the path $s\mapsto \alpha^{v}_z (s)$ for $0<\alp<\alp(z)$ and $s\in [0,\pm \alp]$ are the constant strips. 
		%\begin{itemize}
		%	\item The symplectic manifold is $T^{\ast}\mathbb{R}^2$.
		%	\item The Lagrangian is $\Sigma=\{p^x=\pm 1,\:p^y=0\}$.
		%	\item The fibre is just the fibre of $T^{\ast}\mathbb{R}^2$.
		%\end{itemize}
		
		Now suppose that $z\in \mathcal{Z}^h(\delta;E)-\snetwork(\pi/2)$. Without loss of generality, we assume that $z$ lies on the right-hand side of $\mathcal{Z}^h(\delta;E)-\snetwork(\pi/2)$. We see that $\mathcal{M}^{-,v}_{\epsilon,z}$ must be empty for small enough $\epsilon$ by  Equation \eqref{geometricenergymovingboundarysequence} below and positivity of energy. 
		
		We now show that $\mathcal{M}^{+,v}_{\alp,z}$ is empty for small enough $\alp$ for $z\notin \snetwork(\pi/2)$. Suppose there exists a strictly decreasing sequence of positive real numbers $0<\alp_n<1,\alp_n\to 0$, such that the moduli spaces $\mathcal{M}^{+,v}_{\alp_n,z}$ are all non-empty. We have a sequence of $J^{v}_{\alp_n}$-holomorphic strips $u_n$ satisfying the equation
		\begin{align}\label{movingboundaryfamilyeq}
			\begin{cases}
				\bar{\partial}_{J^{v}_{\alp_n}} u=0 &  \\ 
				u(s,0)\subset ({\psi^v_{\alp_n l(s)}})^{-1}\big(\epsilon\scurve\big)&\\
				u(s,1)\subset F_{z}&\\
				\lim_{s\to -\infty} u(s,t)=\epsilon {z}^{\pm}\\
				\lim_{s\to +\infty} u(s,t)= \epsilon{z}^{\mp},
			\end{cases}
		\end{align}However, by Proposition \ref{Hofertypegromovcompacness}, the sequence of the strips $u_n$ Gromov converges to a non-constant broken $J_{con}$-strip between $F_z$ and $\epsilon\scurve$. By Theorem \ref{maintheorem}, such strips cannot exist, a contradiction. 
		
		For the case $\bullet=h$, the proof is entirely analogous except that $W^h_{\alp}(\epsilon z^{\pm})-W^h_0(\epsilon z^{\pm})$ is now equal to $\epsilon \alp$, by Proposition \ref{deformedpotential}.  Hence we get
		\begin{align}\label{eq:diagonalgeometricenergyhorizontal}
			Area(u)= \alp\epsilon+\alp \int_{-\infty}^{\infty}H^h(u(s,t))l'(s)ds.
		\end{align}
		The same monotonicity argument applies for diagonal continuation strips and for $z\in \snetwork(\pi/2)$. We treat $\mathcal{M}^{nondiag,h}_{\alp,z}$ as before, using \eqref{geometricenergymovingboundarysequence} and Proposition \ref{Hofertypegromovcompacness}. This finishes the proof of Proposition \ref{No-Go moving}.
	\end{proof}
	\subsubsection{Proof of \cref{Hofertypegromovcompacness}}\label{subsubsection:proofofHofertype}
	We now proceed with the proof of Proposition \ref{Hofertypegromovcompacness} for $\mathcal{M}_{\alp_n,z}^{+,\bullet}$. The case where $\alp_n$ is replaced by $-\alp_n$ is entirely analogous. We first establish lower and upper bounds for energy. Recall that we had the following expression for the geometric energy
	\begin{align}\label{geometricenergymovingboundarysequence}
		\int_{\mathcal{Z}} \abs{du_n}^2_{J^{\bullet}_{\alp_n}}=\int_{\mathcal{Z}} u_n^{\ast}\omega= W_{\alp_n}^{\bullet}(\epsilon z^{+})- W_{0}^{\bullet}(\epsilon z^{+})+\alp_n \int_{-\infty}^{\infty} H^{\bullet}(u_n(s,t))l'(s)ds.
	\end{align}
	From Lemma \ref{newprimitive}, we see that \eqref{geometricenergymovingboundarysequence} is bounded above by $2E+\alp_n C$ for some $C>0$ and bounded below by $\hbar=\frac{\epsilon}{2}(W(z^{+})-W(z^{-}))$ for small enough $\alp_n$. Note that this lower bound depends on the fixed $\epsilon$. 
	
	We have the following blow-up analysis for the continuation strips.
	\begin{lemma}
		Let $u_n$ be a sequence of ${J^v_{\alp_n}}$-holomorphic strips with moving boundary conditions as above. Let $p>2$. Then there exists a constant $C=C(p)$ such that
		\begin{align}\norm{Du_n}_{\infty}\leq C.
		\end{align}
	\end{lemma}
	\begin{proof}\label{uniformboudnmoving}
		This is standard blow-up analysis. See \cite[Theorem 3.3]{Oh1992RemovalOB} and \cite[Section 4.6]{McduffSalamon} for details. The only subtle part is that the movie 
		\[\mathcal{K}_{n}:=\{(s,p):s\in [0,1], p\in \psi^{-1}_{\alp_n l(s)}(\epsilon\scurve)\cap K\}\]
		is equal to $[-2,2]\times \epsilon \scurve$ outside some compact subset of $[-2,2]\times [0,1]\times T^{\ast}\tilde{C}$, and $C^{\infty}$-converges to $\epsilon \scurve$. This fact allows us to apply Gromov compactness theorem for totally real submanifolds (\cite[Theorem 1.1]{Frauenfelder_2015} and Remark \cite[Remark 4.1]{Frauenfelder_2015}). Note that otherwise, we would have had to produce a family of symplectic forms that uniformly tame the movie manifold. 
	\end{proof}

	%The manifolds $\mathcal{K}_n$ are not compact, but they are uniformly geometrically bounded, and their non-Lagrangian parts are compact. So we may apply the result in 
	
	%We appeal to the paper \cite{Frauenfelder_2015} on Gromov compactness for holomorphic curves with totally real boundary conditions. Since the non-Lagrangian part of $\mathcal{K}_n$ is compact, the mean value inequality \cite[Section 2.2]{Frauenfelder_2015} and the removal of singularity \cite[Section 2.3]{Frauenfelder_2015} hold for uniformly finite constants. Therefore, the disc bubble develops just as in the discussion in \cite[Section 2.5]{Frauenfelder_2015}. 

	%Alternatively, one could appeal to  But again, these disc bubbles cannot exist because they are localised, and the boundary conditi

	\begin{proof}[Proof Of Proposition \ref{Hofertypegromovcompacness}]
		This is again standard. The only subtle part is that there is a uniform lower bound on the energy of non-constant continuation strips. Therefore, the Gromov limit of the continuation strips undergo \textit{finitely} many breakings. 
	\end{proof}
	\subsubsection{Subdividing path groupoid generators}\label{subsubsection:pathgroupoidgenerators}
	With Proposition \ref{proposition:shortpathspassive} and \ref{Hofertypegromovcompacness} established, we subdivide the arcs $\alpha(w,w')$ and $\gamma(w,w')$ (see Definition \ref{definition:groupoidgenerators}) into smaller paths. Regard $\alpha(w,w')$ as a closed bounded interval in $\mathbb{R}$. By Proposition \ref{No-Go moving}, there exists an open cover $I_z$ of $\alpha(w,w')$ indexed by $z\in \alpha(w,w')$ such that if $z'\in I_z$, then the passive continuation strips from $z$ to $z'$ are all constant. By Lebesgue's number lemma, there exists some $\delta(w,w')>0$ such that any set of diameter $<\delta(w,w')$ is contained in some $I_z$. Take a partition of the interval $\alpha(w,w')$ into segments of length $<\delta(w,w')$. Each subinterval of the partition belongs in some $I_z$. Choose one such $I_z$ for each subinterval once and for all. By adding these points $z$, further refine the partition, and obtain a sequence of points $b(w,w')^{0},....,b(w,w')^{m(w,w')}$, which are in increasing order regarded as points in the interval $\alpha(w,w')$, with $b^u(w)=b(w,w')^{0}$ and $b^d(w')=b(w,w')^{m(w,w')}$. Then the points have the following property that:
	\begin{itemize}
		\item[] for $0\leq i<m(w,w')$, there exists $0\leq j\leq m(w,w')$ such that the passive continuation strips from $b(w,w')^i$ to $b(w,w')^j$ and $b(w,w')^{i+1}$ to $b(w,w')^j$ are all constants.
	\end{itemize}
	Do the same for $\gamma(w,w'')$ and obtain a sequence of points $c(w,w'')^{0}$,...,$c(w,w'')^{k(w,w')}$ which are in increasing order as points in the interval $\gamma(w,w'')$, with $b^u(w)=c(w,w'')^{0}$ and $b^d(w'')^=c(w,w'')^{k(w,w')}$ so that:
	\begin{itemize}
		\item[] for $0\leq i<k(w,w'')$, there exists $0\leq j\leq k(w,w')$ such that the passive continuation strips from $c(w,w'';)^i$ to $c(w,w'')^j$ and $c(w,w'')^{i+1}$ to $c(w,w'')^j$ are all constants.
	\end{itemize}
	We will now write $b(w,w')^{k}\to b(w,w')^{l}$ for the horizontal arc between $b(w,w')^{k}$ and $b(w,w')^{l}$ contained in $\alpha(w,w')$ for $w,w'$ walls and $0\leq k,l\leq k(w,w')$, and similarly for $c(w,w')^{i}\to c(w,w')^{j}$ for $0\leq i,j\leq m(w,w')$. 
	
	\subsection{Computation of family Floer cohomology local system}\label{gradings and spin structures}
	In this section, we compute the family Floer cohomology local system and prove Theorem \ref{maintheoremimprecise}. In Section \ref{gradingsss}, we define the grading data for the spectral curve. In Section \ref{subsubsection:spin structures}, we compute the signs of the continuation strips, using the discussions in Sections \ref{continuationmapsss} and \ref{pathgroupoidrepsss}. In Section \ref{NonAbsss}, we use the sign comparison formula to prove Theorem \ref{maintheoremimprecise}.
	
	\subsubsection{Grading}\label{gradingsss}
	Let $I$ be the complex structure on $C$. We have the following almost complex structure on $T^{\ast}\tilde{C}$
	\[\tilde{I}:=\begin{bmatrix}
		I & 0\\ 
		0 & I^t 
	\end{bmatrix}\]
	with respect to  $TT^{\ast}\tilde{C}=H\oplus V$. Let $\omega_I$ be a non-degenerate $2$-form defined by
	\[\omega_I=g^S(I,)\]
	where $g^S$ is the Sasaki metric on $T^{\ast}\tilde{C}$. Let $\omega_{Im}$ denote the imaginary part of the holomorphic volume form $\Omega$, then the 2-form
	\[\omega_I+i\omega_{Im}\]
	is non-degenerate and gives a preferred section of $\omega_{T^{\ast}\tilde{C}}^{\otimes 2}$. 
	
	The corresponding phase for $\scurve$ is constant since $\omega_{Im}\vert \scurve=0$ and so we choose the grading function to be the constant map $0$. Similarly, we choose the grading function on any of the fibres to be the constant map $-2$. This implies that the chain complex $CF(\scurve,F_z)$ is concentrated in degree $0$, for $z\in C^{\circ}$ (recall that $C^{\circ}$ is the complement of the zeroes and the poles of $\phi$).  
	
	\subsubsection{Spin structures} \label{subsubsection:spin structures}
	
	We now compute signs, using the discussions in Section \ref{pathgroupoidrepsss}. We first make the following simple observation.
	\begin{lemma}\label{lemma:pullbackspinstructure}
		Let $\mathfrak{s}$ be a spin structure on $C$. Then the pullback spin structure $\pi^{\ast}\mathfrak{s}$ on $\scurve^{\circ}$ does not extend to $\scurve$.
	\end{lemma}
	Choose now an almost flat $GL(1;\mathbb{Z}_2)$-local system $\mathcal{B}$ on $\scurve^{\circ}$. The corrected spin structure $\tilde{\mathfrak{s}}$ now extends to $\scurve$. Let $\mathcal{L}$ be an almost flat $GL(1;\mathbb{C})$-local system $L$ on $\scurve^{\circ}$. Tensoring with $\mathcal{B}$, we obtain a genuine $GL(1;\mathbb{C})$-local system on $\scurve$. As before, we have a twisted local system on $P_{\scurve}$. By Proposition \ref{pathgroupoidrepresentation}, the induced family Floer cohomology local system on $C$ is equivalent to the spin pull-back of the twisted family Floer cohomology local system on $C$.
	
	We write the decomposition
	\begin{align}\label{eq:ordereddecomposition}
		CF(\scurve,F_z;\mathbb{Z})=\abs{\mathfrak{o}}_{z^{+}}\oplus \abs{\mathfrak{o}}_{z^{-}}
	\end{align}
	with the ordered basis $\{(+1,0),(0,+1)\}$. Observe that the complex is concentrated in the degree $0$. 
	
	We now look at the case of $\epsilon\scurve$ in detail. Again, let $\mathcal{Z}^h$ be a horizontal chamber and $\mathcal{Z}^h(\delta;E)$ be the unique connected component of $C(\delta;E)$ contained in $\mathcal{Z}^h$. Let $z\in \mathcal{Z}^h\cap M_C$ and let $z'\in \mathcal{Z}^h\cap M_C$ be a point connected to $z$ by a geodesic arc $\alpha^{\bullet}_z$ of length $d$ less than $\alp(z)$ in the sense of Proposition \ref{No-Go moving}. Let $u\in \mathcal{M}^{\bullet}_{(-1)^i d,z}$ , where $i=0$ if the positive sheet picked out by $\alpha$ coincides with the positive sheet of $\sqrt{\phi}$ on $\mathcal{Z}^h(\delta;E)$, or $1$ otherwise. By Proposition \ref{No-Go moving}, $u$ is a constant map. The geodesic arc $\alpha_z^{\bullet}$ together with its velocity vectors lifts to an arc in the sphere bundles $P_{C}$ and $P_{\scurve}$. The following lemma computes the corresponding parallel transport for the \emph{twisted} Floer local system.
	\begin{lemma}
		The corresponding component in the parallel transport map for the \emph{twisted} Floer local system is $\Phi^{\tloc}:\tloc_{z^{\pm}}\to \tloc_{{z'}^{\pm}}$.  
	\end{lemma}
	\begin{proof}
		We obtain a Cauchy-Riemann operator on the disc after glueing the half-strip operators at the strip-like ends. The induced loop of Lagrangian spaces on the disc is the same as the concatenation of the short path \eqref{eq:Lagrangianpath} together with its inverse, with the standard trivialization. Thus, the resulting orientation on the determinant line is the trivial one. This tells us that the corresponding component in the induced parallel transport map is given by $\Phi^{\tloc}$.
	\end{proof}
	
	Now, the choice of $\mathfrak{s}$ and $\tilde{\mathfrak{s}}=\mathfrak{s}\otimes \mathcal{B}$ implies that the induced parallel transport map on $C$ must be twisted by $\mathcal{B}$. Since $L=\mathcal{L}\otimes \mathcal{B}$, $\mathcal{B}$ cancels out and we obtain: 
	\begin{lemma}\label{lemma:signcomputation}
		The corresponding component in the parallel transport map for the Floer local system is $\Phi^{\mathcal{L}}$.  
	\end{lemma}
	%\begin{figure}[t]
	%	\includegraphics[height=7cm]{Diagrams_Prelim/Good_Open_Cover.pdf}
	%	\centering
	%	\caption{An illustration of the open cover data. The dashed line is the branch cut.}
	%	\label{figure:opencoverdata}
	%\end{figure}
	%For $b\in zero(\phi)$, fix a simple closed curve $\gamma_b$ that winds around $b$ once, contained in $U(\delta)$. 
	
	\subsubsection{Proof of non-abelianization}\label{NonAbsss}
	We now use Lemma \ref{lemma:signcomputation} to compute Floer-theoretic parallel transports along the arcs $\alpha(w,w')$ and $\gamma(w,w')$. By construction, given any pair $b(w,w)^{i}, b(w,w)^{i+1}$, there exists some $b(w,w')^j$ such that the continuation strips from $b(w,w)^j$ to $b(w,w)^{i}$ and $b(w,w)^{i+1}$ are all constants, and so necessarily lie on $T^{\ast}\mathcal{Z}^h$. For $\alpha$ an arc contained in a horizontal chamber $\mathcal{Z}^h$, let ${\Phi^{\mathcal{L}}}(\alpha)^{\pm}$ denote the parallel transport map of $\mathcal{L}$ restricted to the $\pm$-lift of $\alpha$ to $\scurve^{\circ}$ with respect to the sheet ordering of the chamber  $\mathcal{Z}^h$. By Lemma \ref{lemma:signcomputation},
	\begin{proposition}\label{paralleltransportcomputation}
		We have 
		\begin{align}
			\Gamma_{\mathcal{L}}(\tilde{\mathfrak{s}})({b(w,w)^{j}\to b(w,w)^{k}})=\begin{bmatrix}
				{\Phi^{\mathcal{L}}}^{+} & 0\\ 
				0 &  {\Phi^{\mathcal{L}}}^{-}
			\end{bmatrix}. \label{computation:twistedparallel}
		\end{align}
		with respect to the ordered basis \eqref{eq:ordereddecomposition}.%orientation of $\mathbb{Z}$ determined by orientations on $\mathfrak{o}_{\tilde{v}}$.
	\end{proposition}
	
	Then composing the parallel transport map from $b(w,w')_{i}$ to $b(w,w')_j$ and the parallel transport map from $b(w,w')_j$ to $b(w,w')_{i+1}$, we see that 
	\begin{corollary}\label{cor:longarctransport}
		The parallel transport along $({b(w,w')_{i}\to b(w,w')_{i+1}})$ is given by the matrix
		\begin{align}
			\Gamma_{\mathcal{L}}(\tilde{\mathfrak{s}})({b(w,w')_{i}\to b(w,w')_{i+1}})&=\begin{bmatrix}
				{\Phi^{\mathcal{L}}}^{+} & 0\\ 
				0 & {\Phi^{\mathcal{L}}}^{-}
			\end{bmatrix}.
		\end{align}
		The parallel transport along $\alpha(w,w')$ is given by the matrix
		\[\begin{bmatrix}	\Phi^{\mathcal{L}}(w,w')^{+} & 0\\ 
			0 &  \Phi^{\mathcal{L}}(w,w')^{-}.
		\end{bmatrix}.\]
		The parallel transport along $\gamma(w,w'')$ is given by the matrix
		\[\begin{bmatrix}	\Phi^{\mathcal{L}}(w,w'')^{+} & 0\\ 
			0 &  \Phi^{\mathcal{L}}(w,w'')^{-}
		\end{bmatrix}.\]
	\end{corollary}
	Finally, repeating the argument in the proof of \ref{paralleltransportcomputation}, we obtain the following:
	\begin{corollary} \label{parallelvert}
		The parallel transport along $\alpha(w)$ is given by the matrix\[\begin{bmatrix}
			1 & 	\mu(w)\\
			0&  1\\
		\end{bmatrix}\]
	\end{corollary}
	Note that we get an \textit{upper triangular matrix} because the moduli spaces $\mathcal{M}^{short,-}(w)$ are empty. Here the basis of $\mathbb{Z}^2$ are chosen with respect to the orientations on the orientation lines. We now compute the number $\mu(w)$ explicitly using $\Phi^{\mathcal{L}}$ and finish the proof of the main theorem. The proof is essentially due to \cite[Section 5.6]{GNMSN}. We will abbreviate $\Phi^{\mathcal{L}}(\alpha^{\pm}(w_i,w_j))$ by $\Phi^{\mathcal{L}}(w_i,w_j)^{\pm}$ . 
	\begin{proposition}	\label{uniquedetermination}
		Let $z$ be a zero of $\phi$ and order the three walls $w_0,w_{\pm 1}$ as above. Let $\Gamma_{\mathcal{L}}(w_i,w_j)$ denote the parallel transport map with respect to $\alpha(w_i,w_j)$.  
		Then we have
		\begin{align}\label{BPSnumbercomputation}
			\mu(w_0)&=-\Phi^{\mathcal{L}}(w_1,w_{-1})^{+}\Phi^{\mathcal{L}}(w_0,w_1)^{-}\Phi^{\mathcal{L}}(w_{-1},w_{0})^{-}	\\
			\mu(w_1)&=-\Phi^{\mathcal{L}}(w_1,w_{-1})^{-}\Phi^{\mathcal{L}}(w_0,w_1)^{-}\Phi^{\mathcal{L}}(w_{-1},w_{0})^{+}		\\
			\mu(w_{-1})&=-\Phi^{\mathcal{L}}(w_0,w_1)^{+}\Phi^{\mathcal{L}}(w_1,w_{-1})^{-}\Phi^{\mathcal{L}}(w_{-1},w_0)^{-}.
		\end{align}
	\end{proposition}
	\begin{proof}
		Consider the concatenation of paths $\alpha(w_0)$, $\alpha(w_0,w_1)$, $\alpha(w_{1})$, $\alpha(w_1,w_{-1})$, $\alpha(w_{-1})$ and $\alpha(w_{-1},w_0)$ in that order. This gives a loop encircling $z$ once and on $C$ it is contractible. Notice that when we go from $\mathcal{Z}^h(w)$ to $\mathcal{Z}^h(w')$ along the loop, we reverse the ordering of the basis. The configuration is illustrated in Figure \ref{Arcsandnetworkfig}. 
		
		Let $\Gamma_{\mathcal{L}}(w)$ denote the parallel transport map with respect to $\alpha(w)$. Then from homotopy invariance, we have 
		\begin{align}\label{homotopyinvariance}
			Id=  \Gamma_{\mathcal{L}}(w_{-1},w_{0})\circ \Gamma_{\mathcal{L}}(w_{-1})\circ \Gamma_{\mathcal{L}}(w_1,w_{-1})\circ \Gamma_{\mathcal{L}}(w_1)\circ \Gamma_{\mathcal{L}}(w_0,w_1)\circ \Gamma_{\mathcal{L}}(w_0)
		\end{align}
		which we rewrite in the form
		\begin{align*}
			Id=
			&\begin{bmatrix}
				0 & \Phi^{\mathcal{L}}(w_{-1},w_0)^{-} \\
				\Phi^{\mathcal{L}}(w_{-1},w_0)^{+} & 0 \\
			\end{bmatrix}\begin{bmatrix}
				1 & 	\mu(w_{-1}) \\
				0 & 1 \\
			\end{bmatrix}\begin{bmatrix}
				0 & \Phi^{\mathcal{L}}(w_1,w_{-1})^{-} \\
				\Phi^{\mathcal{L}}(w_1,w_{-1})^{+} & 0 \\
			\end{bmatrix}\\
			&\begin{bmatrix}
				1 & \mu(w_{+1}) \\
				0 & 1 \\
			\end{bmatrix}\begin{bmatrix}
				0 & \Phi^{\mathcal{L}}(w_0,w_{1})^{-} \\
				\Phi^{\mathcal{L}}(w_0,w_1)^{+} & 0 \\
			\end{bmatrix}\begin{bmatrix}
				1 & \mu(w_0) \\
				0 & 1 \\
			\end{bmatrix}
		\end{align*}
		Expanding the matrix out, it follows that $\mu(w_0),\mu(w_1),\mu(w_2)$ are given by products of the transport coefficients ${\Phi^{\mathcal{L}}}^{\pm}$ as in (\ref{BPSnumbercomputation}).
	\end{proof}

	%The almost flat $GL(1;\mathbb{C})$-local system $\mathcal{L}$ on $\scurve^{\circ}$ gives rises to a path groupoid representation of an almost flat $GL(1;\mathbb{C})$-local system on $\scurve^{\circ}$ as follows:
	
	%\begin{itemize}
	%\item The $\mathbb{C}$-vector spaces correspond to ${\mathcal{L}(\tilde{\alpha})}, \tilde{\alpha}\in \pi^{-1}(\mathcal{P}_C)$.
	%\item Given a path $\gamma$, simply take
	%$hol(\mathcal{L})(\gamma):{\mathcal{L}(\tilde{\alpha})}\to {\mathcal{L}(\tilde{\alpha}')}$.
	%break it in terms of finitely many paths $\alpha_i$ that are contained in components of $C-S(0)$ and short paths crossing walls of $S(0)$. Let $\nabla(\alpha_i)$ denote the flat rank 1 $\mathbb{Z}$-connection given by the difference of $\mathfrak{s}_0(w)$ and the restriction of the spin structure on $\mathcal{R}^h(w)$. Then write 
	%\[\beta(\alpha):=\prod hol(\beta_i)hol((\nabla(\pi(\alpha_i)).\] This is independent of the choice of paths $\alpha_i$ and the representative $\alpha$. We denote the corresponding almost flat $\mathcal{Z}$-local system as $\beta(\mathcal{P}_{\scurve^{\circ}})$.
	%\end{itemize}
	Summarizing everything, we obtain our main theorem.
	
	\begin{proof}(Theorem  \ref{FullNon-Abelianization})
		Corollaries \ref{cor:longarctransport}, \ref{parallelvert}, and \cref{uniquedetermination} give the full description of the path groupoid representation of the Floer cohomology local system, in terms of $\mathcal{L}$.  This is the non-abelianization. 
	\end{proof}
\printbibliography
\end{document}